\documentclass{amsart}
\usepackage{a4wide}
\usepackage{amsmath}
\usepackage{amssymb}
\usepackage{bbold}
\usepackage{graphicx}
\usepackage[portuguese,english]{babel}
\usepackage[utf8]{inputenc}
\usepackage[T1]{fontenc}
\usepackage{hyperref}
\usepackage{enumitem}
\usepackage{multicol}
\usepackage{mathtools}
\usepackage{multirow}
\usepackage{array}
\usepackage{booktabs}
\usepackage{tikz}
\usepackage{tikz-cd}
\usetikzlibrary{arrows}
\usetikzlibrary{shapes}
\usetikzlibrary{positioning}
\usepackage[toc,page]{appendix}
\usepackage[all]{xy}
\usepackage{float}
\usepackage{srcltx}
\usepackage{verbatim}

\newcommand{\om}{\omega}

\newcommand\R{\mathbb{R}}
\newcommand\C{\mathbb{C}}
\newcommand\N{\mathbb{N}}

\newcommand\GW{\mathrm{GW}}

\newcommand\Q{\mathbb{Q}}
\newcommand\X{\mathbb{X}}
\newcommand\Z{\mathbb{Z}}

\newcommand\CP{\mathbb{CP}}

\newcommand{\Ss}{\mathcal{S}}
\newcommand\QH{{QH}}

\newcommand \bbcp{\mathbb C\mathbb P}

\DeclareMathOperator{\Symp}{Symp}
\DeclareMathOperator{\Ham}{Ham}
\DeclareMathOperator{\Diff}{Diff}
\DeclareMathOperator{\Aut}{Aut}
\DeclareMathOperator{\pt}{pt}
\DeclareMathOperator{\AGL}{AGL}

\newcommand{\Muc}{{M}_{\mu,c_1,\hdots,c_{n-1}}}
\newcommand{\Mucc}{{M}_{\mu,c_1,\hdots,c_{n}}}

\newcommand{\Mucccc}{{M}_{\mu,c_1,c_2,c_3,c_4}}
\newcommand{\Muhalf}{{M}_{\mu,c_i=1/2}}

\newcommand{\omuc}{\omega_{\mu,c_1,\hdots, c_{n-1}}}

\newcommand{\conezero}{\mathcal{\#}\{c_1=0\}}

\newtheorem{theorem}{Theorem}[section]
\newtheorem{lemma}[theorem]{Lemma}
\newtheorem{proposition}[theorem]{Proposition}
\newtheorem{corollary}[theorem]{Corollary}
\newtheorem{definition}[theorem]{Definition}
\newtheorem{remark}[theorem]{Remark}
\newtheorem{question}[theorem]{Question}
\newtheorem{conjecture}[theorem]{Conjecture}
\newtheorem*{theorem*}{Theorem}
\newtheorem*{lemma*}{Lemma}
\newtheorem*{proposition*}{Proposition}
\newtheorem*{corollary*}{Corollary}
\newtheorem*{definition*}{Definition}
\newtheorem*{remark*}{Remark}

\begin{document}

\title[Loops in the fundamental group of $\Symp (\bbcp^2\#\,5\overline{ \bbcp}\,\!^2,\omega)$]{Loops in the fundamental group of $\Symp (\bbcp^2\#\,5\overline{ \bbcp}\,\!^2,\omega)$ which are not represented by  circle actions}

\author[S. Anjos]{S\'ilvia Anjos}
\address{SA: Center for Mathematical Analysis, Geometry and Dynamical Systems \\ Department of Mathematics \\  Instituto Superior T\'ecnico \\  Av. Rovisco Pais \\ 1049-001 Lisboa \\ Portugal}
\email{sanjos@math.ist.utl.pt}

\author[M. Barata]{Miguel Barata}
\address{MB: Utrecht Geometry Center \\ Utrecht University \\ Budapestlaan 6 \\ 3584 CD Utrecht\\ The Netherlands}
\email{m.lourencohenriquesbarata@uu.nl}

\author[M. Pinsonnault]{Martin Pinsonnault}
\address{MP: Department of Mathematics \\  University of Western Ontario \\ Canada}
\email{mpinson@uwo.ca}

\author[A. A. Reis]{Ana Alexandra Reis}
\address{AAR: Department of Mathematics \\  Instituto Superior T\'ecnico \\  Av. Rovisco Pais \\ 1049-001 Lisboa \\ Portugal}
\email{ana.alexandra.reis@tecnico.ulisboa.pt}

\date{\today}

\thanks{The first author is partially supported by FCT/Portugal through projects UID/MAT/04459/2019 and PTDC/MAT-PUR/29447/2017. The third author is partially supported by NSERC Discovery Grant RGPIN-2020-06428. All authors except the third are supported by Calouste Gulbenkian Foundation through the program "New Talents in Mathematics".}

\begin{abstract}

We study generators of the fundamental group of the group of symplectomorphisms  $\Symp (\bbcp^2\#\,5\overline{ \bbcp}\,\!^2, \omega)$ for some particular symplectic forms. It was observed by J. K\c{e}dra in \cite{Ke}  that  there are many symplectic 4-manifolds $(M, \omega)$, where $M$ is neither rational nor  ruled,  that admit no circle action and $\pi_1 (\Ham (M,\omega))$ is nontrivial. On the other hand, it follows from \cite{AbrMcD}, \cite{Pin}, \cite{AnjPin} and \cite{AnjEden} that  the fundamental group of the group $ \Symp_h(\bbcp^2\#\,k\overline{ \bbcp}\,\!^2,\omega)$, of symplectomorphisms that act trivially on homology,  with $k \leq 4$,  is generated by circle actions on the manifold. We show that,  for some particular symplectic forms $\omega$, 
the set of all Hamiltonian circle actions generates a proper subgroup in 
 $\pi_1(\Symp_{h}(\bbcp^2\#\,5\overline{ \bbcp}\,\!^2,\omega)).$ 
Our work depends on Delzant classification of toric symplectic manifolds, Karshon's classification of Hamiltonian $S^1$-spaces and the computation of Seidel elements of some circle actions.

\end{abstract}

\subjclass[2010]{Primary 53D35; Secondary 57R17,57S05,57T20}
\keywords{symplectic geometry, symplectomorphism group, fundamental group, Hamiltonian circle actions}

\maketitle

\section{Introduction}\label{section: intro}

Let $(M, \omega)$ be a closed simply connected symplectic manifold. The symplectomorphism group $\Symp(M, \omega)$, equipped with the standard $C^\infty$-topology, is an infinite dimensional Fréchet Lie group. In general, symplectomorphism groups are viewed as intermediate objects between Lie groups and general groups of diffeomorphisms. Of course, this philosophy can be understood in many different ways. One interesting question is to compare the homotopy types of various symplectomorphism groups with those of compact Lie groups, and see to what extend their homotopical and algebraic properties are related. For instance, recall that if $G$ is a compact Lie group, then any element of its fundamental group $\pi_1(G)$ is represented by a continuous  homomorphism $S^1\to G$. Therefore, it is natural to ask whether the same holds for symplectomorphism groups, namely, 

\begin{question}\label{MainQuestion} Suppose that $\pi_1(\Symp(M,\omega))$ is nontrivial. Is every element represented by a continuous homomorphism $S^1 \mapsto \Symp(M, \omega)$ (i.e. a circle action on $M$)? If not, can we characterize homotopy classes that are represented by circle actions?
\end{question}

In~\cite{Ke}, J. K\c{e}dra showed that the answer to the first part of the question is negative in general. 

\begin{theorem}[\cite{Ke}] \label{Kedra}
Let  $(M, \omega)$ be a symplectic blow-up (in a small ball) of a closed simply connected K\"ahler surface, which is neither a rational nor a ruled surface up to blow-up. Then $(M, \omega)$ admits no symplectic circle action although $\pi_1(\Symp(M,\omega))$ is nontrivial. 
\end{theorem}

A concrete example is obtained by taking a K3 surface with any symplectic form.  Another type of example was found by O. Buse in her work on symplectomorphism groups of irrational ruled surfaces~\cite[Proposition 3.3]{Bus}. More precisely, although $\mathbb T^2 \times S^2$ admits Hamiltonian circle actions, she showed that there is an element $\gamma \in \pi_1 (\Ham (\mathbb T^2 \times S^2))$ for which the rational Samelson product $[\gamma, \gamma]_\Q$ does not vanish, which implies that $\gamma$ cannot be  represented by such an action. 

In the present paper, we consider the symplectic rational surfaces $(\bbcp^2\#\,n\overline{ \bbcp}\,\!^2, \omega)$. For $1\leq n\leq 5$, the topological group $\Symp_h(\bbcp^2\#\,n\overline{ \bbcp}\,\!^2, \omega) $ of symplectomorphisms that act trivially on homology, has been studied by several authors  (see \cite{AbGrKi}, \cite{AbrMcD}, \cite{Pin}, \cite{AnjPin}, \cite{AnjEden}, \cite{Eva}, \cite{Seidel}, \cite{LiLiWu2}). In the case $n=5$, P. Seidel~\cite{Seidel} and J. Evans~\cite{Eva} proved that, in the monotone case, this group is homotopy equivalent to the group of orientation-preserving diffeomorphisms of $S^2$ preserving 5 points. Recently, J. Li, T. J. Li and W. Wu in \cite{LiLiWu2} completely determined the group of connected components of  
$\Symp_{h}(\bbcp^2\#\,5\overline{ \bbcp}\,\!^2, \omega)$, called the Torelli symplectic mapping class group, 
 as well as the rank of its fundamental group, for any given symplectic form $\omega$. In order to explain their results which are of interest to us we first recall the definition of reduced forms, and postpone further details to Section \ref{reduceddefinitions}. 

\noindent For $\X_n = \CP^2 \#\, n\overline{ \bbcp}\,\!^2$ let $\{ L, V_1, \hdots, V_n\}$ be a standard basis for  $H_2(\X_n; \Z)$, where $L$ is the class representing a line, and the $V_i$ are the exceptional classes. 

\begin{definition}\label{reducedform}
Consider $\X_n$ with the standard basis $\{ L, V_1, \hdots, V_n\}$ of $H_2(\X_n; \Z)$. Given a symplectic form $\omega$ such that each class $ L, V_1, \hdots, V_n$ has $\omega$-area $\nu, \delta_1, \hdots, \delta_n$, then $\omega$ is called {\bf reduced}~if 
\begin{equation}\label{reduced}
 \nu  > \delta_1 \geq  \hdots \geq \delta_n > 0 \quad \mbox{and} \quad  \nu \geq \delta_1 +\delta_2 + \delta_3. 
\end{equation}
\end{definition}  

We recall in Section  \ref{reduceddefinitions} why any symplectic form on $\X_n$ is diffeomorphic to a reduced one. Note that diffeomorphic symplectic forms yield homeomorphic symplectomorphism groups. Therefore it suffices to understand the symplectomorphism group $\Symp(\X_n, \omega)$ for any reduced form $\omega$. In this section we also recall  that $(\X_n, \omega)$ can be naturally identified  with  $(n-1)$-point blow-ups of the manifold $ (S ^{2} \times S ^{2},  \mu \sigma \oplus \sigma) $, denoted by $\Muc$, where $\sigma$ denotes the standard symplectic form on $S^2$ that gives area 1 to the sphere,  $\mu \geq 1$, and  $c_1, \hdots, c_{n-1}$ denote the capacities of the blow-ups. 

If $n \leq 3$ it is well known that the group $\Symp_h (\Mucc)$ is connected (see for example \cite{LiLiWu}) and it follows from \cite{Pin}, \cite{AnjPin} and \cite{AnjEden} that the fundamental group of $\Symp_{h}(\Mucc)$ is always generated by Hamiltonian circle actions. More precisely, in these cases, the full rational homotopy type of $\Symp_h (\Mucc)$, with $n \leq 3$, is generated by loops in  the fundamental group, represented by circle actions, via Samelson products. On the other hand it was shown by J. Li and T.-J. Li in \cite{LiLi} that if $n \leq 3$ then $\pi_1(\Symp_h (\Mucc))$ is a free abelian group.
  
In \cite{LiLiWu2} the authors show that besides the monotone case, there is a one dimensional family of symplectic manifolds $\Mucccc$ for which the Torelli symplectic mapping class group $\pi_0(\Symp_h(\Mucccc))$ is isomorphic to $\pi_0(\Diff^+(S^2,4))$, where $\Diff^+(S^2,4)$  is the group of orientation-preserving diffeomorphisms of $S^2$ preserving 4 points. This family is defined by the values $ \mu >1$ and $c_i= 1/2$ for all $ i \in \{1,2,3,4\}$. From now on we use the notation $\Muhalf$ to denote this family of symplectic manifolds. 
For all the remaining symplectic forms, the group $\pi_0(\Symp_h(\Mucccc))$ is trivial. Moreover, in  \cite[Section 5.3]{LiLiWu2} the authors show that $\pi_1(\Symp_h(\Muhalf)) = \Z^5$, hence the fundamental group is a free abelian  group.   

In this note we study generators of  the fundamental group of $\Symp_h(\Muhalf)$. Our main result is the following theorem that gives a negative answer to the first part of Question~\ref{MainQuestion}. 
\begin{theorem} \label{main}
If $1 < \mu  \leq \frac32$ then  the set of all  Hamiltonian circle actions generates a proper subgroup of rank 4 in the fundamental group of $\Symp_h(\Muhalf)$. Moreover, if $\mu >\frac32$ then $\pi_1( \Symp_h(\Muhalf))\otimes \Q$ is generated by Hamiltonian circle actions.
\end{theorem}

To our knowledge, this is the first example of symplectic rational surface where the fundamental group of $\Symp (\X_n, \omega)$ is not generated by circle actions. In Section 5 we discuss the existence of more symplectic forms $\omega$ in $\X_5$ for which  a similar phenomenon may occur. 

\begin{remark}
Although there is a generator of $\pi_1( \Symp_h(\Muhalf))\otimes \Q$ which cannot be represented by a Hamiltonian circle action when $1 < \mu  \leq \frac32$, one can find its quantum homology representative (see Proposition \ref{propquantumrepresentative}). 
\end{remark}

Our techniques allows us to completely describe the elements of $\pi_1(\Symp_h(\Muhalf))$ that are represented by Hamiltonian circle actions, answering the second part of Question~\ref{MainQuestion} as well. In particular, we obtain the following result.

\begin{theorem}\label{main2}
For any value of $\mu >1$, there exist infinitely many homotopy classes in the fundamental group $\pi_1( \Symp_h(\Muhalf))$ that cannot be represented by Hamiltonian circle actions. 
\end{theorem}

As a final remark, it seems very likely that Theorem~\ref{main} holds not only rationally but also in the integer case, that is, that the fundamental group $\pi_1(\Symp(\Muhalf))$ is generated by Hamiltonian circle actions whenever $\mu > \frac32$. Although we are not able to prove this stronger claim, we note that our quantum homology calculations imply the existence of five circle actions representing homotopy classes in $\pi_1(\Symp(\Muhalf))$ that can be shown to be not only linearly independent but also primitive.

{\it Organization of the paper}.
In Section 2 we review the main tools we need to prove the theorems above, namely Karshon's classification of Hamiltonian circle actions, Delzant's classification of toric manifolds, and the definitions of the quantum homology ring of a symplectic manifold and of the Seidel morphism. We also recall the  results of \cite{LiLiWu2} regarding $\pi_0$ and $\pi_1$ of the symplectomorphism group $\Symp_h(\Mucccc)$ relevant to our work.  In Section 3 we give a presentation of the quantum homology ring $QH_\ast(\Mucccc)$, that follows from applying the formulas for the quantum product on a rational surface  obtained by B. Crauder and R. Miranda in \cite{CM}.  We dedicate Section 4 to obtaining our main results:  we choose a tentative set of five generators of the rational fundamental group and prove, using the Seidel morphism, that these elements are linearly independent. We conclude this section giving  a classification of all Hamiltonian circle actions on $\Muhalf$, which allows us to determine which homotopy class of loops can be represented by such an action.

Finally, in the last section we propose some further questions that arose naturally on the course of this work. Appendix \ref{QHcomp} contains computations on the quantum ring  while appendix \ref{auxiliaryrelations}  is devoted to the proof of an auxiliary relation between elements in $\pi_1(\Symp_h(\Mucccc))$. 

{\it Acknowledgments}. The first and third authors are very grateful to Jun Li, Tian-Jun Li and Weiwei Wu for showing and explaining to them their results in a preliminary version of \cite{LiLiWu2}. 
The first author also would like to warmly thank Dusa McDuff for useful and enlightening discussions. The second and fourth authors kindly thank the support of the Gulbenkian Foundation which gave them the opportunity to take part on this work. Finally, we would like to thank the referees for their hard and thorough work in reviewing the paper. We greatly appreciate their comments and questions. Moreover, we think that the modifications based on their suggestions and corrections have vastly improved the paper. 

%
%

\section{Background}\label{sec m1: background}

\subsection{Hamiltonian circle actions, decorated graphs and Delzant polygons}

In the forthcoming sections we will study loops in the fundamental group of $\Symp_h(\Mucccc)$. Since these loops will appear as Hamiltonian circle actions, we will make extensive use of Karshon's classification of Hamiltonian circle actions and Delzant's classification of toric actions on symplectic manifolds. For convenience, we give a quick overview on how these classifications work.

Karshon's classification~\cite{Kar} yields a bijection between certain {\it decorated graphs} and $4$-tuples $(M^4,\om,\rho,\Phi)$ consisting of a symplectic $4$-manifold $(M^4,\om)$, and an effective Hamiltonian circle action $\rho$ with a given moment map $ \Phi : M \rightarrow \mathbb{R}$. Given such a tuple $ (M^4, \omega, \rho, \Phi) $, the associated decorated graph is constructed as follows. Each component $C$ of the fixed point set is either a single point or a symplectic surface, and fixed points on which the moment map is not extremal are isolated. For each such component $C$, there is a vertex $\langle C \rangle $, labeled by the real number $\Phi(C)$. A vertex that corresponds to a fixed symplectic surface is said to be "fat" and is given two more labels: the area label $\frac{1}{2 \pi} \int _{C} \omega$, and the genus $g$ of the surface. A \emph{$\mathbb{Z}_{k}$-sphere} is a gradient sphere in $M$ on which $ S^{1} $ acts with isotropy $ \mathbb{Z}_{k} $, $k\geq 2$. For each $\mathbb{Z}_{k}$-sphere containing two fixed points $ p $ and $ q $, the graph has an edge connecting the vertices $ \langle p \rangle $ and $ \langle q \rangle $ labeled by the integer $ k $.

Labeled graphs associated to effective Hamiltonian circle actions are characterized by the following properties. If we order the vertices according to their moment map labels, then
\begin{itemize}
\item there are exactly two extremal vertices;
\item fat vertices are extremal, and if the graph contains two fat vertices, then their genus label must coincide;
\item the area label of any fat vertex must be strictly positive;
\item a vertex is connected to no more than two edges, and no edge is connected to a fat vertex;
\item the moment map labels must be strictly monotone along each chain of edges;
\item if $e_{1},\ldots,e_{\ell}$ is a chain of edges, and if $k_{1},\ldots,k_{\ell}$ are the orders of their stabilizers, then $\gcd(k_{i}, k_{i+1}) = 1$ for $i = 1,\ldots\ell-1$, and $(k_{i-1} + k_{i+1})/k_{i}$ is an integer for $i= 2,\ldots,\ell-1$.
\end{itemize}
We call such graphs \emph{admissible}. 

\begin{theorem}(Karshon \cite{Kar})\label{graphiso}
Each $4$-tuple $(M^4,\om,\rho,\Phi)$ corresponds to a unique admissible labelled graph. Conversely, to each admissible labelled graph corresponds a $4$-tuple $(M^4,\om,\rho,\Phi)$ that is unique up to $S^1$-equivariant symplectomorphisms preserving the moment map.
\end{theorem}

Furthermore, it can be shown that each Hamiltonian action on a $4$-dimensional manifold can be obtained from a circle action on a symplectic ruled surface by performing a sequence of $S^1$-equivariant symplectic blow-ups. At the graph level, equivariant symplectic blow-ups correspond to the simple transformations  pictured in Figure~\ref{Karshon1} and~\ref{Karshon2}. Together with Lalonde-McDuff-Li-Liu's uniqueness theorem~\cite{LalMcD, LiLiu} stating that any two cohomologous symplectic forms on blow-ups of ruled surfaces are diffeomorphic, this gives an effective algorithm to enumerate all effective circle actions on any given $4$-manifold.

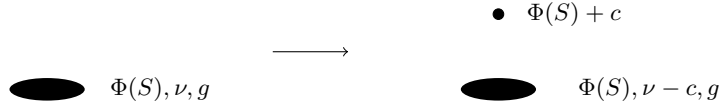
\begin{figure}[thp]
\begin{center}
\begin{tikzpicture}[font=\small]

	\fill[black] (2,-2) ellipse (0.5 and 0.15);
	\node at (3.5,-2) {$\Phi(S),\nu , g $};

	\draw[->] (5,-1.5) -- (6,-1.5);

	\filldraw [black] (8,-1) circle (2pt);
	\node at (9,-1) {$\Phi(S) + c$};
	\fill[black] (8,-2) ellipse (0.5 and 0.15);
	\node at (10,-2) {$\Phi(S),\nu - c, g$};

\end{tikzpicture}
\end{center}
\caption{Blowing-up at a point inside an invariant surface at the minimum value of $ \Phi $}\label{Karshon1}
\end{figure}

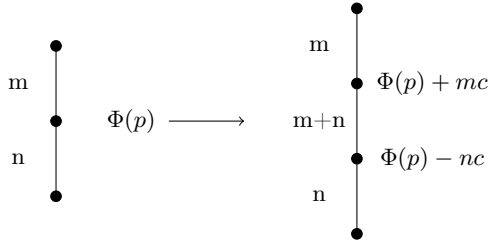
\begin{figure}[thp]
\begin{center}
\begin{tikzpicture}[font=\small]

	\filldraw [black] (2,1.5) circle (2pt);
	\filldraw [black] (2,0.5) circle (2pt);
	\node at (3,0.5) {$ \Phi(p) $};
	\filldraw [black] (2,-0.5) circle (2pt);

	\draw (2,1.5) -- (2,0.5);
	\node at (1.5,1) {$m$};
	\draw (2,0.5) -- (2,-0.5);
	\node at (1.5,0) {$n$};

	\draw[->] (3.5,0.5) -- (4.5,0.5);

	\filldraw[black] (6,2) circle (2pt);
	\filldraw [black] (6,1) circle (2pt);
	\node at (7,1) {$ \Phi(p) + m c $};
	\filldraw [black] (6,0) circle (2pt);
	\node at (7,0) {$ \Phi(p) - n c $};
	\filldraw [black] (6,-1) circle (2pt);

	\draw (6,2) -- (6,1);
	\node at (5.5,1.5) {$m$};
	\draw (6,1) -- (6,0);
	\node at (5.5,0.5) {$m+n$};
	\draw (6,0) -- (6,-1);
	\node at (5.5,-0.5) {$n$};
\end{tikzpicture}
\end{center}
\caption{Blowing-up at an interior fixed point}\label{Karshon2}
\end{figure}

Since we are mainly interested in the continuous map $\rho:S^1\to\Ham(M,\om)$, we do not need to keep track of the moment map associated to a Hamiltonian circle action. As any two moment maps only differ by a constant, we can either consider graphs only up to a uniform translation of their moment map labels, or normalize the moment map by setting $\min_{x\in M} \Phi(x) =0$. Finally, note that the reparametrization of the circle $t\mapsto -t$ corresponds to changing the signs of the moment map labels.

There is an analogous classification of Hamiltonian toric actions that we now briefly describe in the special case of $4$-manifolds. Given a $4$-tuple $(M^{4},\om,\rho,\Phi)$ consisting of a symplectic $4$-manifold $(M,\om)$, an effective Hamiltonian toric action $\rho:T^2\to \Ham(M,\om)$, and a moment map $\Phi:M\to\mathfrak{t}^*\simeq \R^2$, the image $\Phi(M)$ is always a \emph{Delzant polygon}, that is, a polygon satisfying the following three properties: 
\begin{itemize}
\item simplicity, i.e., there are two edges meeting each vertex; 
\item rationality, i.e., the edges meeting at the vertex $ p $ are rational in the sense that each edge is of the form $ p + t u_{i} $, $ t \in [0,\ell_i] $, where $ \ell_i \in \R$ and $ u_{i} \in \mathbb{Z}^{2} $;
\item smoothness, i.e., for each vertex, the corresponding $ u_{1} $, $ u_{2} $ can be chosen to be a $ \mathbb{Z} $-basis of $ \mathbb{Z}^{2} $.
\end{itemize}
Moreover, the pre-image in the manifold of a vertex of the polygon  $ \Phi(M) $ is a fixed point for the torus action while the pre-image of an edge is an invariant 2-sphere. The pre-image of the interior of the polygon consists of free torus orbits. These facts are explained in \cite{De}.

Delzant's classification \cite{De} states that equivalence classes of $4$-tuple $(M^{4},\om,\rho,\Phi)$ up to equivariant symplectomorphisms that preserve the moment maps, are classified by Delzant polygons in $\R^2$. 
If we disregard the moment map and only considers an effective toric action as an injective homomorphism $\rho:\mathbb{T}^2\to\Ham(M,\om)$, it is natural to declare two actions as equivalent if they only differ by a reparametrization of the torus or by a conjugation by an element of $\Symp(M,\om)$. In this setting, the classification theorem yields a bijection
\begin{gather*}
\{\text{Conjugacy classes of toric actions on}~4\text{-manifolds up to reparametrizations}\}\\
\updownarrow\\
\{\text{Delzant polygons in~}\R^2 \text{~up to~} \AGL(2;\Z) \text{~action}\}
\end{gather*}

If we restrict the action to the sub-circle $\{ e\} \times S^1$, we get  a compact four dimensional $S^1$-space.  The moment map for the $S^1$-action is the composition of the $\mathbb{T}^{2}$-moment map with the projection $\R^2 \to \R$ to the second coordinate. The fixed surfaces are the pre-images, under the $\mathbb{T}^{2}$-moment map, of the horizontal edges of the Delzant polygon. Such a surface has genus zero, and its normalized symplectic area is equal to the length of the corresponding horizontal edge. The isolated fixed points are the pre-images of those vertices of the polygon that do not lie on horizontal edges. The $\Z_k$-spheres, $k\geq 2$,  are the pre-images of edges with slope $\pm k/b$ in reduced form, where $b$ is relatively prime to $k$. With this information we can construct the graph for the $S^1$-space out of the Delzant polygon. This  is explained by Karshon in \cite[Section 2.2]{Kar}. Note that, similarly, we can restrict the toric action to the sub-circle $S^1 \times \{ e\}$ in order to obtain another compact four dimensional $S^1$-space. In this case the moment map for the $S^1$-action is the composition of the $\mathbb{T}^{2}$-moment map with the projection $\R^2 \to \R$ to the first coordinate. This relation between polygons and decorated graphs will be particularly useful in the following subsections.

\subsection{Quantum homology and  Seidel morphism}\label{DefQuantumHomology}

Following \cite{McDSal} consider the (small) quantum homology ring $\QH_*(M; \Pi) =  H_*(M, \Q) \otimes_\Q \Pi$ with coefficients in the ring $\Pi:= \Pi^{\mathrm{univ}}[q,q^{-1}]$ where the  $q$ is a polynomial variable of degree 2 and $ \Pi^{\mathrm{univ}},$ called the universal Novikov ring,  is a generalised Laurent series ring in a variable $t$ of degree 0: 
\begin{align}
  \label{eq:pi-univ}
  \Pi^{\mathrm{univ}}:= \left\{ \sum_{\kappa \in \R} r_\kappa t^\kappa \,\big|\, r_\kappa \in \Q, \ \#\{ \kappa > c\mid r_\kappa \neq 0\} < \infty, \forall c \in \R \right \} \,.
\end{align}
The quantum homology $\QH_*(M; \Pi)$ is $\Z$--graded  so that $\deg (a \otimes q^d t^\kappa)= \deg (a) +2d$ with $a \in H_*(M)$. \ The   {\it quantum intersection product} $a*b \in \QH_{i+j -\dim M}(M; \Pi)$, of classes $a \in H_i(M)$ and $b \in H_j(M)$ has the form 
$$a*b = \sum_{B \in H_2^S(M;\Z)} (a*b)_B \otimes q^{-c_1(B)}t^{-\omega(B)},$$
where $H_2^S(M;\Z)$ is the image of $\pi_2(M)$ under the Hurewicz map. The homology class $(a*b)_B \in H_{i+j-\dim M +2c_1(B)}(M)$ is defined by the requirement that 
$$(a*b)_B \cdot_M c = \GW^M_{B,3}(a,b,c) \quad \mbox{ for all } c \in H_*(M).$$   
In this formula $ \GW^M_{B,3}(a,b,c) \in \Q$ denotes the Gromov--Witten invariant that counts the number of spheres in $M$ in class $B$ that meet cycles representing the classes $a,b,c \in H_*(M)$. The product $*$ is extended to $\QH_*(M)$ by linearity over $\Pi$, and is associative (see \cite[Proposition 11.1.9]{McDSal} for a proof of this fact). It also respects the $\Z$--grading and gives $\QH_*(M)$ the structure of a graded commutative ring, with unit $[M]$.

The {\it Seidel morphism} is a homomorphism $\Ss$ from $\pi_1 (\Ham (M, \omega))$ to the degree $2n$ multiplicative units $\QH_{2n}(M)^{\times}$ of the small quantum homology, first  introduced by Seidel in \cite{Se}.  One way of thinking of it is to say that it  "counts" pseudo-holomorphic sections of the bundle $M_\Lambda  \to S^2$ associated to the loop $\Lambda \subset \Ham(M,\omega)$ via the clutching construction (as in \cite[Section 2]{McDTolm}): let $(M, \omega)$ be a closed symplectic manifold and $\Lambda=\{ \Lambda_\theta \}$ be a loop in $\Ham(M, \omega)$ based at identity. Denote by $M_\Lambda$ the total space of the fibration over $S^2$ with fiber $M$ which consists of two trivial fibrations over 2--discs, glued along their boundary via $\Lambda$. Namely, we consider $S^2$ as the union of the two 2-discs $D_0$ and $D_\infty$ such that $D_0$ is the closed unit disc centered at 0 in the Riemann sphere $S^2 = \C \cup \{ \infty\}$ and $D_\infty$ is another copy of this disc, embedded in $S^2 = \C \cup \{ \infty\}$, via the orientation reversing map $r\, e^{i\theta} \mapsto r^{-1} \, e^{i\theta}$. 
The total space is
\begin{align*}
  M_\Lambda =  {\big(M \times D_0 \big) \bigsqcup  \big(M \times D_\infty\big)}{\slash\sim} \quad \mbox{with } (e^{2i\pi\theta}, \Lambda_\theta(x) )_0 \sim (e^{2i\pi\theta},x)_\infty.
\end{align*}
This construction only depends on the homotopy class of $\Lambda$. Moreover, the family (parametrized by $S^2$) of symplectic forms of the fibers, can be extended to give a closed form, $\Omega$, on $M_\Lambda$ (see Sternberg \cite{Stern}). By adding to $\Omega$ the pullback of a suitable area form on the base we get a nondegenerate  form. More precisely,  $\omega_{\Lambda,\kappa} = \Omega + \kappa\cdot \pi^*(\omega_0)$ is symplectic, where $\omega_0$ is the standard symplectic form on $S^2$ (with area 1), $\pi$ is the projection to the base of the fibration and $\kappa$ a big enough constant to make $\omega_{\Lambda,\kappa}$ non-degenerate. (Once chosen, $\kappa$ will be omitted from the notation.)

So we end up with the following Hamiltonian fibration: $$\xymatrix{\relax (M, \omega) \ar@{^(->}[r] & (M_\Lambda,\omega_\Lambda) \ar[r]^{\pi} & (S^2, \omega_0).}$$
In \cite{McDTolm}, McDuff and Tolman observed that, when $\Lambda$ is a circle action (with associated moment map $\Phi_\Lambda$), the clutching construction can be simplified since, then, $M_\Lambda$ can be seen as the quotient of $M \times S^3$ by the diagonal action of $S^1$, $e^{2\pi i \theta}\cdot(x,(z_1,z_2))=(\Lambda_\theta(x),(e^{2\pi i \theta} z_1, e^{2\pi i \theta} z_2))$. 
The symplectic form also has an alternative description in $M \times_{S^1} S^3$. Let $\alpha \in \Omega^1(S^3)$ be the standard contact form on $S^3$ such that $d\alpha = \chi^*(\omega_0)$ where $\chi: S^3 \to S^2$ is the Hopf map and $\omega_0$ is the standard area form on $S^2$ with total area 1. For all $c\in\R$, $\omega+cd\alpha -d(\Phi_\Lambda\alpha)$ is a closed 2--form on $M \times S^3$ which descends through the projection, $p:  M \times S^3 \rightarrow M \times_{S^1} S^3$, to a closed 2--form on $M_\Lambda$:
\begin{align}
  \label{eq:omega-c}
  \omega_c = p( \omega + c d\alpha -d(\Phi_{\Lambda} \alpha))
\end{align}
which extends $\Omega$. Now, if $c>\max \,\Phi_\Lambda$, $\omega_c$ is non-degenerate and coincides with $\omega_{\Lambda,\kappa}$ for some big enough $\kappa$.

A quantum class lying in the image of $\mathcal S$ is called a \emph{Seidel element}. In \cite{McDTolm}, McDuff and Tolman were able to calculate the leading term of Seidel's elements associated to Hamiltonian circle actions whose maximal fixed point component, $F_{\max}$, is semifree, that is, the action is semifree on some neighborhood of  $F_{\max}$. Recall that a circle action is semifree if the stabilizer of every point is trivial or the whole circle.  Moreover, when the codimension of $F_{\max}$ is 2, their result immediately ensures that if there exists an invariant almost complex structure $J$ on $M$ so that $(M,J)$ is Fano, i.e. so that there are no $J$--pseudo-holomorphic spheres in $M$ with non-positive first Chern number, all the lower order terms vanish. In the presence of $J$--pseudo-holomorphic spheres with vanishing first Chern number, there is a priori no reason why arbitrarily large multiple coverings of such objects should not contribute to the Seidel elements.  In fact, as explained in \cite{AL}, when the almost complex manifold $(M,J)$  is only NEF, i.e. $c_1(B) \geq 0$ for every class $B \in H_2(M)$ with a $J$-holomorphic sphere representative, and not Fano, then  there are indeed infinitely many contributions to the Seidel elements. More precisely,  it is shown in \cite{AL} that if $M$ is a 4-toric manifold then these quantum classes can still be expressed by explicit closed formulas.  Moreover, these formulas only depend on the relative position of representatives of elements of $\pi_2(M)$ with vanishing first Chern number as edges of the moment polygon. In particular, they are directly readable from the polygon. 

We now recall  the precise results from \cite{AL} that we will use in the forthcoming sections. Consider 
a 4--dimensional closed symplectic manifold $(M,\omega)$, endowed with a toric structure $(\rho,\Phi)$.  Suppose the associated Delzant polygon  $P = \Phi(M)$ has $m \geq 4$ edges, and consider a Hamiltonian circle action $\Lambda$ on $(M, \omega)$, with moment map $\Phi_\Lambda$, such that $\Lambda$ is a subcircle of the toric action $\rho:\mathbb{T}^{2}\to\Ham(M,\om)$.

We assume additionally, that the fixed point component of $\Lambda$ on which $\Phi_\Lambda$ is maximal is a 2--sphere, $F_{\max} \subset M$, whose momentum image is an edge $D$ of $P$. We denote by $A \in H_2(M;\Z)$ the homology class of $F_{\max}$ and by $\Phi_{\max} = \Phi_\Lambda(F_{\max})$.

In this case, McDuff--Tolman's result  \cite[Theorem 1.10]{McDTolm} ensures that the Seidel element associated to $\Lambda$ is 
\begin{align}\label{McDuff-Tolman-Formula}
  \mathcal S(\Lambda) = A\otimes q t^{\Phi_{\max}} + \sum_{B \in H_2^S\!(M;\Z)^{>0}} a_B \otimes q^{1-c_1(B)} t^{\Phi_{\max}-\omega(B)} 
\end{align}
where $H_2^S(M;\Z)^{>0}$ consists of the spherical classes of positive symplectic area, that is, $\omega(B)>0$ and $a_B\in H_*(M;\Z)$ denotes the contribution of $B$. As mentioned above, when $(M,J)$ is Fano for some $S^1$- invariant $\omega$-compatible almost complex structure $J$  then all the lower order terms vanish and we end up with $ \mathcal S(\Lambda)=A \otimes qt^{\Phi_{\max}}$. 

In the non-Fano case, one has to be careful about the number and relative position of edges, in the vicinity of $D$, corresponding to spheres in $M$ with vanishing first Chern number. We denote the number of such edges by $\conezero$. We denote the edges and the corresponding homology classes in $M$ in a cyclic way, that is, $D$, which we denote by $D_m$ below, has neighbooring edges $D_{m-1}$ on one side and $D_{m+1}=D_1$ on the other, and they respectively induce classes $A_m$, $A_{m-1}$, and $A_{m+1}=A_1$ in $H_2(M;\Z)$.

 Figure \ref{fig:cases} shows the relevant parts of the different polygons we need to consider. Dotted lines represent edges with positive first Chern number and we indicate near each edge with non-trivial contribution the homology class of the corresponding sphere in $M$. For example, in Case \textit{(3c)}, only three homology classes contribute: $A_{m-1}$, $A_m$, and $A_1$; $A_{m-1}$ and $A_1$ have vanishing first Chern number while $c_1(A_m) \neq 0$.

\begin{figure}[thp]
\begin{center}
\begin{tikzpicture}[scale=0.7, roundnode/.style={circle, draw=black!80, thick, minimum size=7mm}, font=\footnotesize]
 
    \draw (0,-0.3) -- (17,-0.3); 
    \draw (4,0) -- (4,7.8); 
    \draw (8,0) -- (8,7.8); 
   
   \draw[dotted,thick] (1.5,1.5) --(1.5,2.8);
   \draw[dotted,thick] (1.5,1.5) --(2.5,0.5);
   \draw[dotted,thick] (1.5,2.8) --(2.5,3.8);
   \node at (1.1,2.2) {$A_m$};
   \node at (2.5,2.2) {$(1)$};  
   \node at (1.5, 0) {$\conezero =0$};
   
   \draw[dotted,thick] (5,1.5) --(5,2.8);
   \draw[dotted,thick] (5,1.5) --(6,0.5);
   \draw[thick] (5,2.8) --(6,3.8);
   \draw[dotted,thick] (7.2,3.8) --(6,3.8);
   \node at (4.6,2.2) {$A_m$};
    \node at (5.1,3.4) {$A_1$};
   \node at (6.5,2.2) {$(3a)$};  
   \node at (6, 0) {$1$};
   
   \draw[thick] (5,5.5) --(5,6.8);
   \draw[dotted,thick] (5,5.5) --(6,4.5);
   \draw[dotted,thick] (5,6.8) --(6,7.8);
   \node at (4.6,6.2) {$A_m$};
   \node at (6.5,6.2) {$(2a)$};

  \draw[thick] (9,5.5) --(9,6.8);
   \draw[dotted,thick] (9,5.5) --(10,4.5);
   \draw[thick] (9,6.8) --(10,7.8);
    \draw[dotted,thick] (10,7.8) --(11.2,7.8);
   \node at (8.6,6.2) {$A_m$};
   \node at (9.1,7.4) {$A_1$};
   \node at (10.5,6.2) {$(2b)$};  

  \draw[dotted,thick] (9,1.5) --(9,2.8);
   \draw[dotted,thick] (9,1.5) --(10,0.5);
   \draw[thick] (9,2.8) --(10,3.8);
   \draw[thick] (11.2, 3.8) --(10,3.8);
  \draw[dotted,thick] (11.2,3.8) --(12.2,2.8);
   \node at (8.6,2.2) {$A_m$};
    \node at (9.1,3.4) {$A_1$};
  \node at (10.6,4.1) {$A_2$};
   \node at (10.5,2.2) {$(3b)$};  
   \node at (12.5, 0) {$2$};

   \draw[dotted,thick] (14,1.5) --(14,2.8);
   \draw[thick] (14,1.5) --(15,0.5);
   \draw[thick] (14,2.8) --(15,3.8);
   \draw[dotted,thick] (16.2, 3.8) --(15,3.8);
   \draw[dotted,thick] (16.2, 0.5) --(15,0.5);
   \node at (13.6,2.2) {$A_m$};
    \node at (14.1,3.4) {$A_1$};
  \node at (13.9,1) {$A_{m-1}$};
   \node at (15.7,2.2) {$(3c)$};

\end{tikzpicture}
\end{center}
\caption{Cases appearing in Theorem \ref{Seidelelements}}
 \label{fig:cases}
 \end{figure}
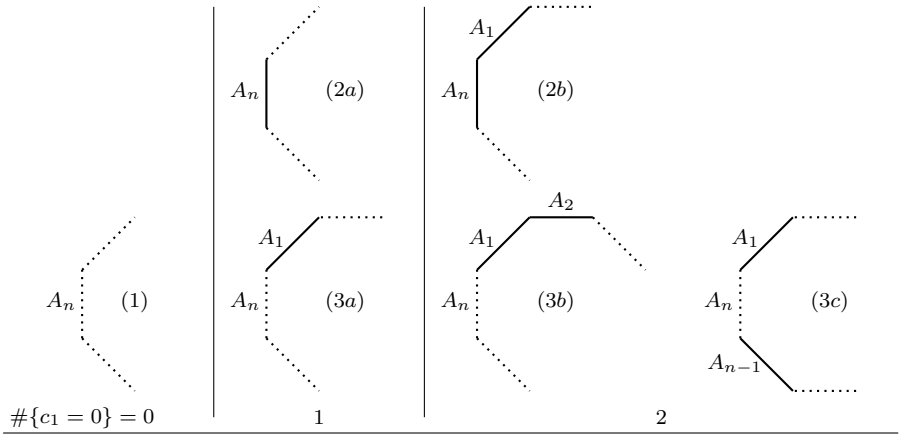

 Now, the following theorem gives the explicit expression of the Seidel element associated to $\Lambda$ when $\conezero \leq 2$.
 
\begin{theorem}[\cite{AL}, Theorem 4.5] \label{Seidelelements}
Let $(M,\omega)$ be a 4-dimensional closed symplectic manifold, endowed with a toric structure.  With the notation described above, assume the Delzant polygon P has $m \geq 4$ edges, and  that the  fixed point component of the Hamiltonian action $\Lambda$ on which $\Phi_\Lambda$ is maximal is a 2--sphere, $F_{\max} \subset M$. Additionally, assume  that $(M,J)$ is NEF,  for some $S^1$-invariant $\omega$-compatible almost complex structure $J$.  Then, in the cases described by Figure \ref{fig:cases}, the Seidel element associated to $\Lambda$ is

\begin{itemize}
\item[\textit{(1)}]$ \displaystyle \mathcal S(\Lambda)=A_m\otimes qt^{\Phi_{\max}}$
\item[\textit{(2a)}] $  \displaystyle \mathcal S(\Lambda)=A_m\otimes q\,\frac{t^{\Phi_{\max}}}{1-t^{-\omega(A_m)}}$
 \item[\textit{(2b)}]  $  \displaystyle \mathcal S(\Lambda)=\left( A_m\otimes q\,\frac{t^{\Phi_{\max}}}{1-t^{-\omega(A_m)}}-A_1\otimes q\,\frac{t^{\Phi_{\max}-\omega(A_1)}}{1-t^{-\omega(A_1)}} \right) \frac{1}{1-t^{-\omega(A_m)-\omega(A_1)}} $
\item[\textit{(3a)}]$  \displaystyle \mathcal S(\Lambda)=A_m\otimes qt^{\Phi_{\max}}-A_1\otimes q\,\frac{t^{\Phi_{\max}-\omega(A_1)}}{1-t^{-\omega(A_1)}} $
\item[\textit{(3b)}] $  \displaystyle \mathcal S(\Lambda)  =  A_m\otimes qt^{\Phi_{\max}}-A_1\otimes q\,\frac{t^{\Phi_{\max}-\omega(A_1)}}{1-t^{-\omega(A_1)}}  \\
 \hspace{.5cm} -\left(A_1\otimes q\,\frac{t^{\Phi_{\max}}}{1-t^{-\omega(A_1)}}-A_2\otimes q\,\frac{t^{\Phi_{\max}-\omega(A_2)}}{1-t^{-\omega(A_2)}}\right)\frac{t^{-\omega(A_1)-\omega(A_2)}}{1-t^{-\omega(A_1)-\omega(A_2)}} $
\item[\textit{(3c)}] $  \displaystyle \mathcal S(\Lambda)=A_m\otimes qt^{\Phi_{\max}}-A_{m-1}\otimes q\,\frac{t^{\Phi_{\max}-\omega(A_{m-1})}}{1-t^{-\omega(A_{m-1})}}-A_1\otimes q\,\frac{t^{\Phi_{\max}-\omega(A_1)}}{1-t^{-\omega(A_1)}}.$
\end{itemize}
\end{theorem}

\subsection{The fundamental group of \texorpdfstring{$\Symp ( \Mucccc)$}{Symp (Mucccc)}} \label{reduceddefinitions}

In this section we recall the main results obtained by Li-Li-Wu  \cite{LiLiWu2} on the Torelli symplectic mapping class group and on the rank of the fundamental group of the group Symp$_{h}(\Mucccc)  $ of symplectomorphisms that act trivially on homology, for any given symplectic form.  
 
First, note that diffeomorphic symplectic forms define symplectomorphism groups that are homeomorphic, and that symplectomorphism groups are invariant under rescalings of symplectic forms. Consequently, we can restrict ourselves to symplectic forms belonging to a fundamental domain for the action of $ \Diff \times \mathbb{R_*} $ on the space $ \Omega _{+} $ of orientation-compatible symplectic forms defined on the n--fold blow-up $ \mathbb{X}_{n} $.
The cohomology class of a reduced class  $\omega$ (see definition \ref{reducedform}  in Introduction) is $\nu L  -  \delta_1 V_1 -  \hdots  -  \delta_n V_n$. Let $\mathcal J_\omega$ be the space of compatible almost complex structures on $\X_n$. For any $J \in \mathcal J_\omega$ on $\X_n$, the first Chern class $c_1:= c_1 (TM) \in H^2(\X_n; \Z)$ is the Poincaré dual to $K:= 3L - \sum_i V_i$. Let $\mathcal K$ be the symplectic cone of $\X_n$, that is, 
$$ \mathcal{K} = \{ A \in H^2(\X_n; \Z) \, | \,  A = [\omega]  \ \mbox{for some symplectic form} \ \omega  \in  \Omega _{+} \}$$
Now if $C$ stands for the Poincar\'{e} dual of the symplectic cone of $ \mathbb{X}_{n} $, then by  {\it uniqueness of symplectic blow-ups} proved by McDuff in \cite{McD}, the diffeomorphism class of the form $ \omega$ only depends on its cohomology class. 
Therefore, it is enough to describe a fundamental domain of the action of $ \Diff \times \mathbb{R_*} $ on $C$.
Moreover, the canonical class $K$ is unique up to orientation preserving diffeomorphisms \cite{LiLiu2}, so it suffices to describe the action of the diffeomorphisms fixing $K$, $\Diff_K$, on
$$ C_{K} = \{ A \in  H_{2} ( \mathbb{X}_{n} ; \mathbb{R} ) \ : \ A = PD [ \omega ]  \ \mbox{for some} \ \omega \in \Omega_{K} \} $$
 where $ \Omega _{K}$ is the set of orientation-compatible symplectic forms with $K$ as the symplectic canonical class.
By  results in \cite{LiLiu2}, the set of reduced classes is a fundamental domain of $C_{K} ( \mathbb{X}_{n}  )$ under the action of $ \Diff_{K}$. A proof of this result is also given in \cite[Theorem 1.4]{KaKe}. We now consider the following change of basis in $H_2(\X_n; \Z)$. Consider the symplectic manifold $(S^2 \times S^2, \mu \sigma \oplus \sigma )$ where the homology class of the base $ B \in H _{2} ( S ^{2} \times S ^{2} )$  represented by $ S ^{2} \times \{ pt \} $ has area $ \mu, $ and the homology class of the fiber $ F \in H _{2} ( S ^{2} \times S ^{2} )$ represented by  $ \{ pt \} \times S ^{2} $ has area $ 1 $. Recall that  $\Muc=(S^2 \times S^2 \#\, (n-1)\overline{ \bbcp}\,\!^2, \omuc)$ is obtained from $(S^2 \times S^2, \mu \sigma \oplus \sigma )$,  by performing $n-1$ successive blow-ups of capacities  $ c_{1}, \hdots, c_{n-1}$. This can be naturally identified with $(\X_n, \omega)$. One easy way to understand the equivalence is as follows: let $ \{ B, F, E_{1}, \hdots, E_{n-1} \} $ be the basis for $H_{2}(\Muc ; \mathbb{Z}) $ where the $ E_{i}$ represent the exceptional spheres arising from the blow-ups. We identify $L$ with $ B+F-E_{1} $, $ V_{1}$ with $ B - E_{1}$, $ V_{2} $ with $ F - E_{1} $, and $ V_{i} $ with $ E_{i-1} $, with $3 \leq i \leq n$. Then the  {\it uniqueness of symplectic blow-ups} due to D. McDuff (see \cite[Corollary 1.3]{McD}) implies that  the symplectomorphism type of a symplectic blow-up of a rational ruled manifold along an embedded ball of capacity $ c \in (0,1) $ depends only on the capacity $ c $ and not on the particular embedding used in obtaining the blow-up. Using this result and after rescaling, we conclude that for parameters satisfying the relations 

\begin{equation}\label{sympequiv}
 \mu = \dfrac{\nu - \delta_{2}}{\nu - \delta_{1}} , \quad  c_{1}= \dfrac{\nu - \delta_{1} - \delta_{2}}{\nu - \delta_{1}}, \quad \mbox{and} \quad  c_{i} = \dfrac{\delta_{i+1}}{\nu - \delta_{1}}, \quad  2 \leq i \leq n-1.
 \end{equation}
there exists a symplectomorphism between two symplectic manifolds encoded by these parameters such that 
$$ \nu L  -  \delta_1 V_1 -  \hdots  -  \delta_n V_n = \mu B + F -c_1 E_1 - \hdots -c_{n-1} E_{n-1}. $$
Summarizing the above, we showed that 
\begin{lemma}
Every symplectic form on $S^2 \times S^2 \#\, (n-1)\overline{ \bbcp}\,\!^2$ is, after rescalling, diffeomorphic to a form Poincaré dual to  $\mu B + F -c_1 E_1 - \hdots -c_{n-1} E_{n-1}$ with 
$$ 0 < c_{n-1} \leq \hdots \leq c_1 \leq 1 \leq \mu \quad \mbox{and} \quad c_i+c_j \leq 1. $$ 
\end{lemma}

Recall (see \cite{LiLi}) that the normalized reduced symplectic cone is defined as the space of reduced symplectic classes having area 1 on $L$, the line class. Note that cohomologous symplectic forms on a rational or ruled surface are diffeomorphic (cf. \cite{LalMcD2,LiLiu2}). We represent such a class by $(1 | \delta_1, \hdots, \delta_n)$, or $(\delta_1, \hdots, \delta_n) \in \R^n.$ For $3 \leq n \leq 8$ such a cone is a $n$-simplex with one facet removed, where the monotone class is one of the vertices, namely $M_n = (\frac13, \hdots, \frac13)$.
We are interested in the case when the manifold is $\X_5$ where the normalized reduced cone is convexly generated by 5 half-closed intervals $\{MO,MA,MB,MC,MD\}$, with vertices $M= (\frac13,\frac13,\frac13,\frac13,\frac13)$ which corresponds to the monotone case, $O=(0,0,0,0,0)$, $A=(1,0,0,0,0)$, $B=(\frac12,\frac12, 0,0,0)$, $C=(\frac13,\frac13,\frac13,0,0)$ and 
$D=(\frac13,\frac13,\frac13,\frac13,0)$ (for more details see  \cite{LiLi}).

Let $N_\omega$ be the number of symplectic -2 spheres classes. Then Li-Li-Wu proved the following
\begin{theorem}[\cite{LiLiWu2}, Theorem 1.2]
Consider $\X_5$ with any symplectic form $\omega$. Then the rank of the fundamental group of $\Symp_h (\X_5, \omega)$ satisfies 
$$rank( \pi_1 (\Symp_h (\X_5, \omega))) = N_\omega -5  + rank ( \pi_0 (\Symp_h(\X_5, \omega))),$$
where the rank of $\pi_0 (\Symp_h(\X_5, \omega))$ means the rank of its abelianization. 
\end{theorem}
Moreover, along the edge $MA$, when $\nu =1, \delta_1 > \delta_2 =\delta_3=\delta_4=\delta_5$ and $\delta_1+ \delta_2 +\delta_3 =1$, or equivalently, when $\mu > 1$ and  $c_1=c_2=c_3=c_4= 1/2$, it follows from \cite[Lemma 5.10]{LiLiWu2} and its proof (in particular from sequence (37)) that  $ \pi_1 (\Symp_h (\X_5, \omega))=\Z_5$.  This is the case we will study in detail in the forthcoming sections. In particular we will show that a generating set of the fundamental group of $\Symp_h(\X_5, \omega)$ can be realized by Hamiltonian circle actions except in some particular interval of values of $\mu$. 

\section{Quantum homology of \texorpdfstring{$\Mucccc$}{Mucccc}}\label{QuantumHomology}

In \cite{CM}, Crauder and Miranda compute the quantum cohomology of a general rational surface, which includes the case of the blown-up manifold $\mathbb{CP}^2 \# 5 \overline{\mathbb{CP}^2}$. Using Poincaré duality, this allows us to construct a presentation for the quantum homology ring $QH_\ast(\Mucccc)$, so that we can then compare different Seidel elements. The relations
$$
\begin{cases}
L = B+F-E_1 \\
V_1 = B-E_1 \\
V_2 = F-E-1 \\
V_i = E_{i-1},~\mathrm{for}~ 2 \leq  i \leq n, 
\end{cases}.
$$
give an explicit way of translating information in terms of the classes $\{L, V_1, \ldots, V_n\}$ to one in terms of $\{B, F, E_1, \ldots, E_{n-1} \}$. 

An explicit formula for the quantum product in terms of the classes $L, V_i$ is given in Proposition 5.3. of \cite{CM}. The coefficients that appear in these can be computed with the help of the Tables in Section 4 of \cite{CM}, giving us a closed formula for the products we're interested in. As an example, the product of two classes, different from the class of a single point $p \in H_0(\mathbb{X}_5, \mathbb{Z})$, in $\mathbb{CP}^2 \# 5 \overline{\mathbb{CP}^2}$ is given by
\begin{align*}
(dL-&\sum\limits_{i} m_i V_i) \ast (d'L-\sum\limits_{i} m'_i V_i) = \left( dd' - \sum\limits_{i} m_i m'_i \right) pt^{[p]}  + \sum\limits_{k} m_k m'_k V_k t^{[V_k]}+ \\
 &+ \sum\limits_{j,k}(d-m_j-m_k)(d'-m'_j-m'_k)(L-V_j-V_k)t^{[L-V_j-V_k]} + \\
& +\left(2d-\sum_{i} m_i \right) \left(2d'-\sum_{i} m'_i \right)(2L-V_1-V_2-V_3-V_4-V_5)t^{[2L-V_1-V_2-V_3-V_4-V_5]} +  \\
&+ \sum\limits_{j} (d-m_j)(d'-m'_j)X t^{[L-V_j]} + \\
&+ \sum_{j,k,l,n} (2d-m_j-m_k-m_l-m_n) (2d'-m'_j-m'_k-m'_l-m'_n) X t^{[2L-V_j-V_k-V_l-V_n]}
\end{align*}
where $i,j,k,l,n$ always represent distinct indices, $X \in H_4(\mathbb{X}_5,\mathbb{Z})$ is the class of the manifold, $d \in \mathbb{Z}_{>0}, ~m_i, m'_j \in \mathbb{Z}_{\geq 0}$ and $t^{[A]}$ means $t$ to the power of the negative symplectic area of the corresponding class $A$ (in the symplectic viewpoint). 

The next proposition gives a description of the ring $QH_\ast(\Muhalf)$. For the sake of simpler notation, let 
\begin{align*}
& b_{ij}= (B-E_i -E_j) \otimes q \,  \frac{t^{\frac12}}{1-t^{1- \mu}}, \quad   f_{ij}= (F-E_i -E_j) \otimes q \, \frac{t^{\frac12}}{1-t^{1- \mu}} \quad \mbox{and} \\ 
&  e_{i}= E_i \otimes q \, \frac{t^{\frac12}}{1-t^{1- \mu}}, 
\end{align*}
and, as before, let distinct letters in the indices correspond to distinct elements. Its proof is just computing the quantum products by the formula above and then translating them to a formula in terms of $\{B, F, E_1, \ldots, E_4\}$. It follows from  \cite[Proposition 5.3]{CM} that we have the presentation for $QH_\ast (\Muhalf)$ given below. 

\smallskip 
\begin{proposition}\label{QHring}
With the notation above,  when $\mu> 1$, as  a $\Pi^{\rm univ}$-algebra, we have 
$$QH_\ast(\Muhalf) \simeq \Pi^{\rm univ}[f_{ij},~b_{ij},~e_i] / I_{\mu,c_i=1/2},$$
where $\Pi^{\rm univ}$ is the universal Novikov ring and $I_{\mu,c_i=1/2}$ is the ideal generated by 
\begin{align*}
  & (1) \ b_{ij}b_{k\ell}=1;   & &  (7) \ f_{ij} f_{k\ell}=0;  \\
 &  (2) \  b_{ij}b_{ik}=b_{ij} f_{ij} +f_{j\ell} +1;   & &  (8) \ f_{ij} f_{ik}=  f_{ij} ( b_{ij} +1);  \\
 &  (3)\  b_{ij}^2 = 2 b_{ij} f_{ij} +  f_{ij} +  f_{k\ell} + 1;  & & (9) \ f_{ij}^2 = 2 f_{ij} ( b_{ij} +1);\\
  &  (4) \ f_{ik} ( b_{ij} +1)=0;  & &  (10) \ (f_{ij} +  f_{k\ell} ) ( b_{ij} +1)=0; \\
 & (5) \ b_{ij}\left(f_{ij} +e_i + \frac{t^{1-\mu}}{1-t^{1- \mu}}\right) = e_j+ \frac{t^{1-\mu}}{1-t^{1- \mu}}; 
&  &  (11) \ f_{ij}\left(b_{ij} +e_i + \frac{1}{1-t^{1- \mu}}\right) =0; \\
 &  (6) \  b_{ij}\left(e_k + \frac{t^{1-\mu}}{1-t^{1- \mu}}\right) =  f_{k\ell} + e_\ell+  \frac{t^{1-\mu}}{1-t^{1- \mu}} ; & & (12) \ f_{ij} \left(e_k + \frac{t^{1-\mu}}{1-t^{1- \mu}}\right)=0;
\end{align*}
\begin{align*}
 & (13) \ e_ie_j = (2b_{ij}+2f_{k\ell}+e_k+e_\ell)\frac{t^{1-\mu}}{1-t^{1- \mu}} + \frac{2t^{1-\mu}+t^{2-2\mu}}{(1-t^{1- \mu})^2}; \\
 &  (14) \ e_i^2 = b_{ij}f_{ij} +\frac{f_{ij}}{1-t^{1-\mu}} + (2b_{ij}+f_{k\ell}+2e_j)\frac{t^{1-\mu}}{1-t^{1- \mu}} + \frac{2t^{1-\mu}+t^{2-2\mu}}{(1-t^{1- \mu})^2}. 
 \end{align*}
 \end{proposition}

\begin{remark}
Of course our description doesn't give a minimal set of generators nor is that the intention of Proposition \ref{QHring}. The generators were picked with the intent of simplifying the computations of the Seidel morphism and they also give some simple insight into the ring structure: for instance, relation $f_{ij} f_{k\ell}=0 $ implies that there are zero divisors.
\end{remark}

\section{Generators of \texorpdfstring{$\pi_1(\Symp(\Mucccc))$}{pi1(Symp(Mucccc))}}

\subsection{Hamiltonian circle actions in \texorpdfstring{$\Mucccc$}{Mucccc}}
In this section we list all equivalence classes of Hamiltonian circle actions on symplectic manifolds whose symplectic cohomology class belongs to the edge $MA$ of the reduced symplectic cone. Recall that along this edge we have  $\mu > 1$ and $c_1=c_2=c_3=c_4= 1/2$. 
Recall also that  Karshon's classification \cite[Section 6.2]{Kar} implies that every compact four dimensional Hamiltonian $S^1$-space can be obtained from a \emph{minimal space}, which can be $\CP^2$, a Hirzebruch  surface, or an irrational ruled manifold (see \cite[Section 6.3]{Kar}), by a sequence of equivariant symplectic blow-ups at fixed points.  It follows that the only possible Hamiltonian circle actions on the symplectic manifolds belonging to the edge $MA$ are  the ones corresponding to the labelled graphs of Figure \ref{graphMAgeneric}, where the values of $a$ and $b$ represent the symplectic area of the invariant spheres and depend on which sphere we perform the blow-up. In our figures we omit the genus label since, in our case, the invariant surfaces are always embedded spheres. Moreover, since the symplectic area of the spheres is positive, i.e. $a,b > 0$, and $c_1=c_2=c_3=c_4= 1/2$  then  we can only have $a+b= 2\mu -2$. 
\begin{figure}[thp]
\begin{tikzpicture}[scale=0.8,font=\footnotesize]
       \fill[black] (2,1.5) ellipse (1 and 0.2);
	\node at (4.2,1.5) {$1,a$};
	\filldraw [black] (0.5,0) circle (2pt);
	\filldraw [black] (1.5,0) circle (2pt);
	\node at (4.5 ,0) {$ \displaystyle{\frac12}$};
	\filldraw [black] (2.5,0) circle (2pt);
	\filldraw [black] (3.5,0) circle (2pt);
        \fill[black] (2,-1.5) ellipse (1 and 0.2);
	\node at (4.2,-1.5) {$0,b$};
\end{tikzpicture}
\caption{Graphs representing Hamiltonian circle actions on symplectic manifolds belonging to the ray $MA$}
\label{graphMAgeneric}
\end{figure}

\subsection{Circle actions and homotopy classes of loops}
A labelled graph only determines a circle action up to symplectomorphisms. Equivalently, a labelled graph defines a conjugacy class of circles in $\Symp(M,\om)$. Consequently, any such graph defines an element of
\[\pi_{1} \left(\Symp(M,\om)\right)\slash \Symp(M,\om)\simeq \pi_1\left(\Symp(M,\om)\right)\slash \pi_0(\Symp(M,\om))\]
where the action is by conjugation. The analysis of this action is done in two stages.

For any symplectic manifold $(M,\om)$ belonging to the ray $MA$, the action of $\Symp(M,\om)$ on homology induces a short exact sequence 
\[1\to \Symp_{h}(M,\om)\to\Symp(M,\om)\stackrel{f}{\longrightarrow}\Aut_{c_1,[\om]}\left(H_2(M,\Z)\right)\to 1\]
where $\Aut_{c_1,[\om]}\left(H_2(M,\Z)\right)$ is the group of automorphisms of the lattice $H_2(M,\Z)$ preserving the intersection form and the classes dual to $c_1(M,\om)$ and $[\om]$. The fact that the map $f$ is onto follows from three results in~\cite{LiWu} that we briefly recall. Firstly, by~\cite[Proposition 4.14]{LiWu}, the group $\Aut_{c_1,[\om]}\left(H_2(M,\Z)\right)$ is generated by reflections about spherical homology classes $A$ satisfying 3 conditions: $A\cdot A=-2$, $c_1(A)=0$, and $\om(A)=0$. Such a class is called a $(K, [\om])$-null spherical class, where $K$ denotes the symplectic canonical class. Secondly, a symplectic Dehn twist along a Lagrangian sphere $L$ induces the reflection $R([L])$ in homology. Finally, the result follows from Proposition~5.6 in~\cite{LiWu} which proves existence of Lagrangian spheres representing $(K, [\om])$-null spherical classes.

In our case, along $MA$, the automorphism group is isomorphic to the Weyl group $D_4$ given by the trivalent Dynkin diagram (see~\cite{LiLiWu2}). It is easy to see that it fixes the classes $2B+2F-E_1-E_2-E_3-E_4$ and $F$, while it acts transitively on the 8 exceptional classes ${E_1,\ldots,E_4, F-E_1,\ldots,F-E_4}$. In particular, the only element of $\Aut_{c_1,[\om]}$ that fixes the four exceptional classes $E_i$ is the identity.

In order to keep track of the action of $\Aut_{c_1,[\om]}\simeq \Symp(M,\om)\slash \Symp_h(M,\om)$ on Hamiltonian circle actions, we consider \emph{extended  graphs} as defined in~\cite[Section 5, p.33]{Kar} decorated with \emph{homology labels}. Starting with the standard labelled graph of Figure~\ref{graphMAgeneric}, we add extra (dotted) edges that represent free invariant spheres connecting each interior fixed point to extrema of the moment map. Each such sphere is the closure of a free $\C^*$-orbit, where the $\C^*$ action is defined from the choice of a generic $S^1$-invariant almost-complex structure. Since the action of $\Symp(M,\om)$ preserves the genericity of almost-complex structures, an extended graph defines a configuration of invariant spheres that is well defined up to conjugation. We then label the edges of the extended graph with homology classes according to the sequence of blow-ups that is used to construct the Hamiltonian $S^1$-manifold. Geometrically, this amounts to labelling invariant spheres with their homology class following a specific sequence of equivariant blow-ups performed on $\left(S^2\times S^2,\mu\sigma\otimes\sigma\right)$, starting with the two fixed surfaces labelled $B=[S^2\times\pt]$ and $F=[\pt\times S^2]$. The possible extended labelled graphs are shown in Figures~\ref{firstcase}, \ref{secondcase}, and~\ref{graphMA}. By construction, the group $\Symp(M,\om)$ acts on its corresponding extended labelled graph with kernel $\Symp_h(M,\om)$. In our case, these extended labelled graphs classify $\Symp_h$-equivalence classes of $S^1$-manifolds on the edge $MA$ endowed with a given framing $\phi:H_2(M,\Z)\to\Z\langle B,F,E_1,E_2,E_3,E_4\rangle\simeq \Z^{1,5}$.

Recall from~\cite{LiLiWu2} that for $(M,\om)$ belonging to the edge $MA$, the symplectomorphism group $\Symp_h(M,\om)$ is not connected. Indeed, 
$\pi_0 (\Symp_h(M,\om))= \pi_0(\Diff^+(S^2,4))\simeq P_4(S^2) / \Z_2$, where $P_4(S^2)$ is the pure braid group of 4 strings in $S^2.$ Since an extended labelled graph only determines an element in $\pi_1(\Symp(M,\om))/\pi_0(\Symp_h(M,\om))$ we have to understand how $\pi_0(\Symp_h(M,\om))$ acts on $\pi_1(\Symp(M,\om))$. We postpone  this analysis to Section~\ref{Section:Seidelelements}. 

\subsection{Extended labelled graphs along the edge \texorpdfstring{$MA$}{MA}}\label{Extendedgraphs}
We now describe a finite set of one parameter families of extended labelled graphs, parametrized by $\mu$, that correspond to symplectic manifolds belonging to the edge $MA$ of the reduced symplectic cone. The number of elements in these families depend on the range of $\mu$. As explained above, each such graph corresponds to a $\Symp_h(M,\om)$-conjugacy class of Hamiltonian circle actions.

Notice that these actions only exist as long as the symplectic area of the classes corresponding to the fixed spheres is positive. 
Assuming  $1 <\mu \leq  \frac32$  we have  4  circle actions represented by the graphs in Figure \ref{firstcase}. We can consider for example: $z_{0,12}, z_{0,13},z_{0,14}$ and $z_1$. We do not consider flips of these graphs as they represent actions which are inverse to these ones. 
\begin{remark}\label{actionssymplectomorphic}
Note that removing the homology labels and the dotted edges from the graphs representing the four actions $z_{0,12}, z_{0,13},z_{0,14}$ and $z_1$, we get exactly the same underlying labelled graph. It follows that these four Hamiltonian circle actions are conjugated by symplectomorphisms that act non-trivially on homology. 
\end{remark}

\begin{figure}[thp]
\begin{minipage}{.4\textwidth}
\begin{tikzpicture}[scale=1,font=\footnotesize]
\node at (2,2.1 ) {$B  - E_l -E_m $};
       \fill[black] (2,1.5) ellipse (1 and 0.2);
	\node at (4.2,1.5) {$1,\mu-1$};
	\filldraw [black] (0.5,0) circle (2pt);
	\draw[dashed, black!60] (0.5,0) --(1.5,1.5); 
	\node at (0.7,0.8) {$ E_\ell$};
	\draw[dashed, black!60] (0.5,0) --(1.5,-1.5); 
	\node at (0.3,-0.7) {$ F- E_\ell$};
	\filldraw [black] (1.2,0) circle (2pt);
	\draw[dashed, black!60] (1.2,0) --(1.6,1.5); 
	\node at (1.7,0.8) {$ E_m$};
	\draw[dashed, black!60] (1.2,0) --(1.6,-1.5); 
	\node at (1.8,-0.2) {$ F-E_m$};
	\node at (4.5 ,0) {$ \displaystyle{\frac12}$};
	\filldraw [black] (2.5,0) circle (2pt);
	\draw[dashed, black!60] (2.5,0) --(2.3,1.5); 
	\node at (2,0.3) {$ F-E_i$};
	\draw[dashed, black!60] (2.5,0) --(2.3,-1.5); 
	\node at (2.6,-0.7) {$ E_i$};
	\filldraw [black] (3.5,0) circle (2pt);
	\draw[dashed, black!60] (3.5,0) --(2.5,1.5); 
	\node at (3.6,0.7) {$F- E_j$};
	\draw[dashed, black!60] (3.5,0) --(2.5,-1.5); 
	\node at (3.4,-0.7) {$ E_j$};
        \fill[black] (2,-1.5) ellipse (1 and 0.2);
	\node at (4.2,-1.5) {$0,\mu-1$};
	\node at (2,-2.1) {$B-E_i -E_j$};
       \node at (2,-2.8) {Circle action $z_{0,ij}$};
\end{tikzpicture}
\end{minipage}%
\begin{minipage}{.4\textwidth}
\begin{tikzpicture}[scale=1,font=\footnotesize]
\node at (2,2.1 ) {$B +F  -E_1 - E_2 -E_3 -E_4$};
       \fill[black] (2,1.5) ellipse (1 and 0.2);
	\node at (4.2,1.5) {$\displaystyle{1,\mu -1}$};
	\filldraw [black] (0.5,0) circle (2pt);
	\draw[dashed, black!60] (0.5,0) --(1.5,1.5); 
	\node at (0.7,0.8) {$ E_1$};
	\draw[dashed, black!60] (0.5,0) --(1.5,-1.5); 
	\node at (0.3,-0.8) {$ F-E_1$};
	\filldraw [black] (1.5,0) circle (2pt);
	\draw[dashed, black!60] (1.5,0) --(2,1.5); 
	\node at (1.5,0.8) {$ E_2$};
	\draw[dashed, black!60] (1.5,0) --(2,-1.5); 
	\node at (1.4,-0.4) {$ F-E_2$};
	\node at (4.5 ,0) {$ \displaystyle{\frac12}$};
	\filldraw [black] (2.5,0) circle (2pt);
	\draw[dashed, black!60] (2.5,0) --(2.2,1.5); 
	\node at (2.7,0.7) {$ E_3$};
	\draw[dashed, black!60] (2.5,0) --(2.2,-1.5); 
	\node at (2.7,-0.4) {$ F-E_3$};
	\filldraw [black] (3.5,0) circle (2pt);
	\draw[dashed, black!60] (3.5,0) --(2.7,1.5); 
	\node at (3.5,0.7) {$ E_4$};
	\draw[dashed, black!60] (3.5,0) --(2.7,-1.5); 
	\node at (3.7,-0.8) {$ F-E_4$};
        \fill[black] (2,-1.5) ellipse (1 and 0.2);
	\node at (4.2,-1.5) {$\displaystyle{0,\mu -1}$};
	 \node at (2,-2.1) {$B-F$};
       \node at (2,-2.8) {Circle action $z_{1}$};
\end{tikzpicture}
\end{minipage}%
\caption{Family of graphs in the case $\mu >1$}
\label{firstcase}
\end{figure}
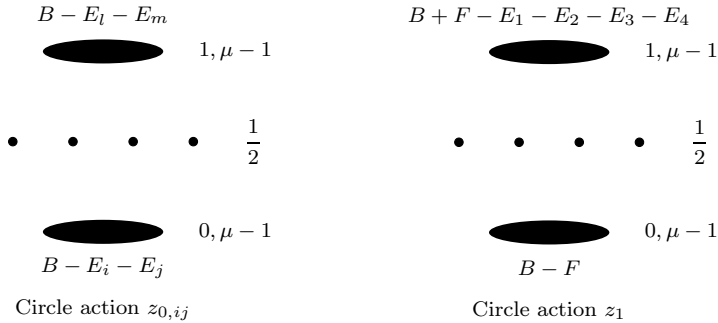
Moreover, as we increase the value of $\mu$, there are more classes that can  be represented by the fixed spheres, as we see next. 
If we consider $ \mu > \frac32  $ then we can add   the graphs in Figure \ref{secondcase} to the previous family. It should be clear that there are 8 such graphs, because $i,j, \ell,m=1,2,3,4$ are all distinct. More precisely we have the graphs representing the following actions: 
  $z_{0,123},z_{0,124},z_{0,134},z_{0,234}$ and $z_{1,1},z_{1,2},z_{1,3},z_{1,4}$.
\begin{figure}[thp]
\begin{minipage}{.4\textwidth}
\begin{tikzpicture}[scale=1,font=\footnotesize]
\node at (2,2.1 ) {$B  - E_m $};
       \fill[black] (2,1.5) ellipse (1 and 0.2);
	\node at (4.2,1.5) {$\displaystyle{1,\mu  - \frac12}$};
	\filldraw [black] (0.5,0) circle (2pt);
	\draw[dashed, black!60] (0.5,0) --(1.5,1.5); 
	\node at (0.6,0.8) {$ E_m$};
	\draw[dashed, black!60] (0.5,0) --(1.5,-1.5); 
	\node at (0.4,-0.8) {$ F-E_m$};
	\filldraw [black] (1.5,0) circle (2pt);
	\draw[dashed, black!60] (1.5,0) --(1.7,-1.5); 
	\node at (1.3,-0.7) {$ E_i$};
	\draw[dashed, black!60] (1.5,0) --(1.7,1.5); 
	\node at (1.3,0.3) {$ F-E_i$};
	\node at (4.5 ,0) {$ \displaystyle{\frac12}$};
	\filldraw [black] (2.5,0) circle (2pt);
	\draw[dashed, black!60] (2.5,0) --(2.3,-1.5); 
	\node at (2.2,-0.7) {$ E_j$};
	\draw[dashed, black!60] (2.5,0) --(2.3,1.5); 
	\node at (2.4,0.5) {$ F-E_j$};
	\filldraw [black] (3.5,0) circle (2pt);
        \fill[black] (2,-1.5) ellipse (1 and 0.2);
        \draw[dashed, black!60] (3.5,0) --(2.5,-1.5); 
	\node at (3.4,-0.7) {$ E_\ell$};
	 \draw[dashed, black!60] (3.5,0) --(2.5,1.5); 
	\node at (3.6,0.7) {$ F-E_\ell$};
	\node at (4.2,-1.5) {$\displaystyle{0,\mu-\frac32}$};
	\node at (2,-2.1) {$B-E_i-E_j-E_\ell$};
       \node at (2,-2.8) {Circle action $z_{0,ij\ell}$};
\end{tikzpicture}
\end{minipage}%
\begin{minipage}{.4\textwidth}
\begin{tikzpicture}[scale=1,font=\footnotesize]
\node at (2,2.1 ) {$B +F  -E_j - E_l -E_m $};
       \fill[black] (2,1.5) ellipse (1 and 0.2);
	\node at (4.2,1.5) {$\displaystyle{1,\mu  - \frac12}$};
	\filldraw [black] (0.5,0) circle (2pt);
	\draw[dashed, black!60] (0.5,0) --(1.5,1.5); 
	\node at (0.6,0.8) {$ E_j$};
	\draw[dashed, black!60] (0.5,0) --(1.5,-1.5); 
	\node at (0.2,-0.8) {$ F-E_j$};
	\filldraw [black] (1.5,0) circle (2pt);
	\draw[dashed, black!60] (1.5,0) --(2,1.5); 
	\node at (1.5,0.8) {$ E_\ell$};
	\draw[dashed, black!60] (1.5,0) --(2,-1.5); 
	\node at (1.6,-0.7) {$ F-E_\ell$};
	\node at (4.5 ,0) {$ \displaystyle{\frac12}$};
	\filldraw [black] (2.5,0) circle (2pt);
	\draw[dashed, black!60] (2.5,0) --(2.2,1.5); 
	\node at (2.7,0.7) {$ E_m$};
	\draw[dashed, black!60] (2.5,0) --(2.2,-1.5); 
	\node at (2.5,-0.3) {$ F-E_m$};
        \filldraw [black] (3.5,0) circle (2pt);
        \draw[dashed, black!60] (3.5,0) --(2.5,-1.5); 
	\node at (3.4,-0.7) {$ E_i$};
	\draw[dashed, black!60] (3.5,0) --(2.5,1.5); 
	\node at (3.7,0.7) {$ F-E_i$};
	  \fill[black] (2,-1.5) ellipse (1 and 0.2);
	\node at (4.2,-1.5) {$\displaystyle{0,\mu -\frac32}$};
	 \node at (2,-2.1) {$B-F- E_i$};
       \node at (2,-2.8) {Circle action $z_{1,i}$};
\end{tikzpicture}
\end{minipage}%
\caption{New family of graphs if $\mu >\frac32$}
\label{secondcase}
\end{figure}

Then there are no more possible classes for the fixed symplectic spheres  unless we consider $\mu >2$. In this case, 8 new circle actions appear, where the following pairs of classes are represented by the fixed spheres: $B$ and $B-E_1-E_2-E_3-E_4$; $B-2F$ and $B+2F-E_1-E_2-E_3-E_4$; and $B-F-E_i-E_j$  and $B+F -E_\ell-E_m$ with $i,j,\ell, m =1,2,3,4$ all distinct. If we restrict the range of values of $\mu$ further, it is easy to see that the number of circle actions  keeps increasing. More precisely,  when $\mu$  passes $k +\frac12$ or $k+1$, for $k \in \Z_{\geq 1}$ the number of actions always increases by 8. Therefore we obtain the following proposition.

\begin{figure}[thp]
\begin{minipage}{.4\textwidth}
\begin{tikzpicture}[scale=1,font=\footnotesize]
\node at (2,2.1 ) {$B +kF  -E_1 - E_2 -E_3 -E_4$};
       \fill[black] (2,1.5) ellipse (1 and 0.2);
	\node at (4.2,1.5) {$\displaystyle{1,\mu +k-2}$};
	\filldraw [black] (0.5,0) circle (2pt);
	\draw[dashed, black!60] (0.5,0) --(1.5,1.5); 
	\node at (0.7,0.8) {$ E_1$};
	\draw[dashed, black!60] (0.5,0) --(1.5,-1.5); 
	\node at (0.3,-0.8) {$ F-E_1$};
	\filldraw [black] (1.5,0) circle (2pt);
	\draw[dashed, black!60] (1.5,0) --(2,1.5); 
	\node at (1.5,0.8) {$ E_2$};
	\draw[dashed, black!60] (1.5,0) --(2,-1.5); 
	\node at (1.4,-0.4) {$ F-E_2$};
	\node at (4.5 ,0) {$ \displaystyle{\frac12}$};
	\filldraw [black] (2.5,0) circle (2pt);
	\draw[dashed, black!60] (2.5,0) --(2.2,1.5); 
	\node at (2.7,0.7) {$ E_3$};
	\draw[dashed, black!60] (2.5,0) --(2.2,-1.5); 
	\node at (2.7,-0.4) {$ F-E_3$};
	\filldraw [black] (3.5,0) circle (2pt);
	\draw[dashed, black!60] (3.5,0) --(2.7,1.5); 
	\node at (3.5,0.7) {$ E_4$};
	\draw[dashed, black!60] (3.5,0) --(2.7,-1.5); 
	\node at (3.7,-0.8) {$ F-E_4$};
        \fill[black] (2,-1.5) ellipse (1 and 0.2);
	\node at (4.2,-1.5) {$\displaystyle{0,\mu -k}$};
	 \node at (2,-2.1) {$B-kF$};
       \node at (2,-2.8) {Circle action $z_{k}$};
\end{tikzpicture}
\end{minipage}%
\begin{minipage}{.4\textwidth}
\begin{tikzpicture}[scale=1,font=\footnotesize]
\node at (2,2.1 ) {$B +kF  -E_j - E_l -E_m $};
       \fill[black] (2,1.5) ellipse (1 and 0.2);
	\node at (4.2,1.5) {$\displaystyle{1,\mu +k - \frac32}$};
	\filldraw [black] (0.5,0) circle (2pt);
	\draw[dashed, black!60] (0.5,0) --(1.5,1.5); 
	\node at (0.6,0.8) {$ E_j$};
	\draw[dashed, black!60] (0.5,0) --(1.5,-1.5); 
	\node at (0.2,-0.8) {$ F-E_j$};
	\filldraw [black] (1.5,0) circle (2pt);
	\draw[dashed, black!60] (1.5,0) --(2,1.5); 
	\node at (1.5,0.8) {$ E_\ell$};
	\draw[dashed, black!60] (1.5,0) --(2,-1.5); 
	\node at (1.6,-0.7) {$ F-E_\ell$};
	\node at (4.5 ,0) {$ \displaystyle{\frac12}$};
	\filldraw [black] (2.5,0) circle (2pt);
	\draw[dashed, black!60] (2.5,0) --(2.2,1.5); 
	\node at (2.7,0.7) {$ E_m$};
	\draw[dashed, black!60] (2.5,0) --(2.2,-1.5); 
	\node at (2.5,-0.3) {$ F-E_m$};
        \filldraw [black] (3.5,0) circle (2pt);
        \draw[dashed, black!60] (3.5,0) --(2.5,-1.5); 
	\node at (3.4,-0.7) {$ E_i$};
	\draw[dashed, black!60] (3.5,0) --(2.5,1.5); 
	\node at (3.7,0.7) {$ F-E_i$};
	  \fill[black] (2,-1.5) ellipse (1 and 0.2);
	\node at (4.2,-1.5) {$\displaystyle{0,\mu -k -\frac12}$};
	 \node at (2,-2.1) {$B-kF- E_i$};
       \node at (2,-2.8) {Circle action $z_{k,i}$};
\end{tikzpicture}
\end{minipage}%

\bigskip 

\begin{minipage}{.4\textwidth}
\begin{tikzpicture}[scale=1,font=\footnotesize]
\node at (2,2.1 ) {$B  +kF- E_l -E_m $};
       \fill[black] (2,1.5) ellipse (1 and 0.2);
	\node at (4.2,1.5) {$1,\mu +k -1$};
	\filldraw [black] (0.5,0) circle (2pt);
	\draw[dashed, black!60] (0.5,0) --(1.5,1.5); 
	\node at (0.7,0.8) {$ E_\ell$};
	\draw[dashed, black!60] (0.5,0) --(1.5,-1.5); 
	\node at (0.3,-0.7) {$ F- E_\ell$};
	\filldraw [black] (1.2,0) circle (2pt);
	\draw[dashed, black!60] (1.2,0) --(1.6,1.5); 
	\node at (1.7,0.8) {$ E_m$};
	\draw[dashed, black!60] (1.2,0) --(1.6,-1.5); 
	\node at (1.8,-0.2) {$ F-E_m$};
	\node at (4.5 ,0) {$ \displaystyle{\frac12}$};
	\filldraw [black] (2.5,0) circle (2pt);
	\draw[dashed, black!60] (2.5,0) --(2.3,1.5); 
	\node at (2,0.3) {$ F-E_i$};
	\draw[dashed, black!60] (2.5,0) --(2.3,-1.5); 
	\node at (2.6,-0.7) {$ E_i$};
	\filldraw [black] (3.5,0) circle (2pt);
	\draw[dashed, black!60] (3.5,0) --(2.5,1.5); 
	\node at (3.6,0.7) {$F- E_j$};
	\draw[dashed, black!60] (3.5,0) --(2.5,-1.5); 
	\node at (3.4,-0.7) {$ E_j$};
        \fill[black] (2,-1.5) ellipse (1 and 0.2);
	\node at (4.2,-1.5) {$0,\mu-k-1$};
	\node at (2,-2.1) {$B-kF-E_i -E_j$};
       \node at (2,-2.8) {Circle action $z_{k,ij}$};
\end{tikzpicture}
\end{minipage}%

\bigskip 

\bigskip 

\begin{minipage}{.4\textwidth}
\begin{tikzpicture}[scale=1,font=\footnotesize]
\node at (2,2.1 ) {$B  +kF- E_m $};
       \fill[black] (2,1.5) ellipse (1 and 0.2);
	\node at (4.2,1.5) {$\displaystyle{1,\mu +k - \frac12}$};
	\filldraw [black] (0.5,0) circle (2pt);
	\draw[dashed, black!60] (0.5,0) --(1.5,1.5); 
	\node at (0.6,0.8) {$ E_m$};
	\draw[dashed, black!60] (0.5,0) --(1.5,-1.5); 
	\node at (0.4,-0.8) {$ F-E_m$};
	\filldraw [black] (1.5,0) circle (2pt);
	\draw[dashed, black!60] (1.5,0) --(1.7,-1.5); 
	\node at (1.3,-0.7) {$ E_i$};
	\draw[dashed, black!60] (1.5,0) --(1.7,1.5); 
	\node at (1.3,0.3) {$ F-E_i$};
	\node at (4.5 ,0) {$ \displaystyle{\frac12}$};
	\filldraw [black] (2.5,0) circle (2pt);
	\draw[dashed, black!60] (2.5,0) --(2.3,-1.5); 
	\node at (2.2,-0.7) {$ E_j$};
	\draw[dashed, black!60] (2.5,0) --(2.3,1.5); 
	\node at (2.4,0.5) {$ F-E_j$};
	\filldraw [black] (3.5,0) circle (2pt);
        \fill[black] (2,-1.5) ellipse (1 and 0.2);
        \draw[dashed, black!60] (3.5,0) --(2.5,-1.5); 
	\node at (3.4,-0.7) {$ E_\ell$};
	 \draw[dashed, black!60] (3.5,0) --(2.5,1.5); 
	\node at (3.6,0.7) {$ F-E_\ell$};
	\node at (4.2,-1.5) {$0,\mu-k -\frac32$};
	\node at (2,-2.1) {$B-kF-E_i-E_j-E_\ell$};
       \node at (2,-2.8) {Circle action $z_{k,ij\ell}$};
\end{tikzpicture}
\end{minipage}%
\begin{minipage}{.4\textwidth}
\begin{tikzpicture}[scale=1,font=\footnotesize]
\node at (2,2.1 ) {$B  +kF $};
       \fill[black] (2,1.5) ellipse (1 and 0.2);
	\node at (4.2,1.5) {$\displaystyle{1,\mu +k }$};
	\filldraw [black] (0.5,0) circle (2pt);
	\draw[dashed, black!60] (0.5,0) --(1.5,1.5); 
	\node at (0.3,0.8) {$ F-E_1$};
	\draw[dashed, black!60] (0.5,0) --(1.5,-1.5); 
	\node at (0.7,-0.8) {$ E_1$};
	\filldraw [black] (1.5,0) circle (2pt);
	\draw[dashed, black!60] (1.5,0) --(1.7,-1.5); 
	\node at (1.3,-0.7) {$ E_2$};
	\draw[dashed, black!60] (1.5,0) --(1.7,1.5); 
	\node at (1.3,0.3) {$ F-E_2$};
	\node at (4.5 ,0) {$ \displaystyle{\frac12}$};
	\filldraw [black] (2.5,0) circle (2pt);
	\draw[dashed, black!60] (2.5,0) --(2.3,-1.5); 
	\node at (2.2,-0.7) {$ E_3$};
	\draw[dashed, black!60] (2.5,0) --(2.3,1.5); 
	\node at (2.4,0.5) {$ F-E_3$};
	\filldraw [black] (3.5,0) circle (2pt);
        \fill[black] (2,-1.5) ellipse (1 and 0.2);
        \draw[dashed, black!60] (3.5,0) --(2.5,-1.5); 
	\node at (3.4,-0.7) {$ E_4$};
	 \draw[dashed, black!60] (3.5,0) --(2.5,1.5); 
	\node at (3.6,0.7) {$ F-E_4$};
	\node at (4.2,-1.5) {$0,\mu-k -2$};
	\node at (2,-2.1) {$B-kF-E_1-E_2-E_3-E_4$};
       \node at (2,-2.8) {Circle action $z_{k,1234}$};
\end{tikzpicture}
\end{minipage}%

\caption{Families of graphs of Hamitonian $S^1$-spaces encoded by the edge $MA$. }
\label{graphMA}
\end{figure}

\begin{proposition}\label{actionslist}
The  Hamiltonian circle actions on the symplectic manifolds belonging to the edge $MA$  of the reduced symplectic cone are the ones represented by the labelled graphs in Figure~\ref{graphMA}. In particular, these actions satisfy the following existence conditions: 
\begin{alignat*}{2}
 & \bullet z_k  \ \mbox{exists} \ \mbox{iff} \ \mu >k \ \mbox{and} \ \mu > 2-k; \\
 & \bullet z_{k,i}  \ \mbox{exists} \ \mbox{iff} \ \mu > k +\frac12 \ \mbox{and} \ \mu > \frac32-k; \\
 & \bullet  z_{k,ij} \ \mbox{exists} \ \mbox{iff}\  \mu > k +1;\\
 & \bullet z_{k,ijl} \ \mbox{exists} \  \mbox{iff} \  \mu > k +\frac32;\\
 & \bullet z_{k,1234} \ \mbox{exists} \ \mbox{iff}  \ \mu > k +2. 
\end{alignat*}
where $k \in \Z_{\geq 0}$ and $i,j,\ell,m =1,2,3,4$ are all distinct.
\end{proposition}

\begin{remark}\label{RemarkNumberCircleActions}
As we saw above, when $ 1 < \mu \leq \frac32$ there exist only four Hamiltonian circle actions: $z_{0,12}, z_{0,13},z_{0,14}$ and $z_1$, so there are not enough circle actions to generate the fundamental group. Then when $\mu$ passes $k+\frac12$, where $k \geq 1$ there exist eight more circle actions, namely 
$z_{k-1,123}, z_{k-1,124}, z_{k-1,134}, z_{k-1,234}$ and $z_{k,i}$ with $i=1,2,3,4$, and  when $\mu$ passes $k+1$ eight more circle actions appear:  $z_{k-1,1234}$, $z_{k,12}, z_{k,13},z_{k,14},z_{k,23},z_{k,24},z_{k,34}$ and~$z_{k+1}$. 
\end{remark}

\begin{remark}
Although the number of Hamiltonian circle actions keeps increasing as the values of $\mu$ increase, we know by the work of J. Li, T-J. Li and W. Wu in 
\cite{LiLiWu2} that the rank of $\pi_1$ remains constant as $\mu$ increases so there can only be at most 5 independent circle actions as elements of the fundamental group. 
\end{remark}

\begin{remark}
In the forthcoming sections we prove that, for $\mu > \frac32$, the fundamental group $\pi_1(\Symp(\Muhalf))\otimes \Q$  is indeed generated by circle actions. We choose as a tentative set of generators the set consisting of the 4 circle actions $z_{0,12}, z_{0,13},z_{0,14}$ and $z_1$, which are the only ones that exist for all values of $\mu$ plus the action $z_{1,4}$, which exists as soon as $\mu$ passes $\frac32$. The reason why we choose this action, among the new 8 actions which appear when $\mu$ passes $\frac32$, is geometric and relates with the work of \cite{LiLiWu2}. This action corresponds to a simultaneous rotation of all the spheres except the base in the configuration of 7-exceptional spheres used in the proof of \cite[Lemma 5.10]{LiLiWu2}, where the authors show that  the rank of $\pi_1 (\Symp_h(\Muhalf))$ is 5. While the first 4 actions fix spheres with self-intersection -2, the action $z_{1,4}$ fixes a -3 self-intersection sphere in class $B-F-E_4$. 
\end{remark}

\subsection{The Seidel morphism along the edge \texorpdfstring{$MA$}{MA}}\label{Section:Seidelelements}
In this section we prove our first main theorem, namely Theorem~\ref{main}. The proof relies on the computation of the Seidel elements associated to the circle actions $z_{0,1i}, i=2,3,4$, $z_1$ and $z_{1,4}$, and on the fact that they are linearly independent in the subgroup of invertible elements of the quantum homology of the manifold $\Muhalf$.

\subsubsection{Seidel elements and deformations}\label{DeformationSeidel}
In order to describe the Seidel morphism
\[\Ss:\pi_1(\Symp(\Muhalf))\to \QH_{2n}(\Muhalf)^{\times}\] 
our strategy is to combine the invariance property of $\Ss$ under the natural $\Symp_{h}(M,\omega)$ action, the invariance of Gromov-Witten invariants with respect to symplectic deformations, and Theorem~\ref{Seidelelements} that describes certain Seidel elements associated to subcircles of toric actions on NEF symplectic manifolds.

More precisely, we first observe that the Seidel morphism
\[\Ss:\pi_1(\Symp_0(M,\om))\to \QH_{2n}(M,\om)^{\times}\]
defined in Section~\ref{DefQuantumHomology} is invariant under the action of $\pi_0(\Symp_h(M,\om))$ on $\pi_1(\Symp_0(M,\om)$. This follows from the definition of $\Ss$ and, in the case of Hamiltonian circle actions, can be seen directly from the formula~(\ref{McDuff-Tolman-Formula}) given by McDuff and Tolman. In particular, given a Hamiltonian circle action $\gamma:S^1\to\Ham(M,\om)$, its image $\Ss(\gamma)$ is determined by the labelled extended graph associated to~$\gamma$.

Next, consider a Hamiltonian circle action $\rho$ on $\Muhalf$ and the corresponding loop $[\rho]$ in $\Symp(\Muhalf)$. Let $\Omega(\rho)$ be the space of all symplectic forms that are invariant under this action and write $\Omega_0(\rho)$ for the connected component of $\omega_{\mu,c_i=1/2}$. A \emph{$\Q$-generic symplectic form} is a symplectic form whose cohomology class $ [\mu;c_1,\cdots,c_4]$ is given by 
 coefficients that are linearly independent over $\Q$.
One can show that invariant $\Q$-generic symplectic forms are dense in $\Omega_0(\rho)$: let $\om$ be any invariant symplectic form and $\delta_i$ be invariant closed two forms whose cohomology classes are a basis for 
$H^2(M,\R)$. Then for sufficiently small $c_i$ the two forms $ \om + \sum_i c_i\delta_i$ are invariant and symplectic. The result follows readily. 

The same argument shows that extended graphs whose moment map labels are small continuous perturbations of the labels associated to $\rho$ correspond to deformation equivalent symplecic forms invariant under the same circle action. Consequently, for any $S^1$-manifold $(M,\om')$ associated to such an extended graph, there is no ambiguity as to what the Seidel element $\Ss([\rho])\in\QH_{2n}(M,\om')^{\times}$ is\footnote{In order to compare the Seidel homomorphisms associated to deformation equivalent symplectic forms, a more general approach would be to use an enlarged Novikov ring as in~\cite{Zh-QuantumRings}.}.

We now observe that, due to deformation invariance of Gromov-Witten invariants, given any symplectic form $\om'$ in $\Omega_0(\rho)$, the quantum homology ring $\QH(M,\om')$ is obtained from the quantum ring of a $\Q$-generic class $\om_{\mu,c_1,c_2,c_3,c_4}$ by setting the values of the coefficients $\mu,c_1,c_2,c_3,c_4$ equal to those of $\om'$. We thus have a natural specialization map $\QH(M,\om_{\mu,c_1,c_2,c_3,c_4})\to \QH(M,\om')$ that sends the Seidel element of $[\rho]$ computed relatively to the generic form $\om_{\mu,c_1,c_2,c_3,c_4}$ to the one computed relatively to the form $\om'$.

Finally, we can apply the previous discussion starting with a Hamiltonian circle actions $\rho$ on $\Muhalf$. From the above remarks, we can find a deformation equivalent $\Q$-generic form $\om_{\mu;c_1,c_2,c_3,c_4}$ such that the sizes of the blow-ups satisfy the inequalities $0 < c_4 < c_3 < c_2 < c_1 <  c_i+c_j < 1 < \mu$, with $i,j \in \{1,2,3,4\}$ distinct. Symplectic cohomology classes satisfying this condition are said to be \emph{reduced generic} or, more simply, \emph{generic}. By choosing the sizes carefully, we can embed the circle action $\rho$ into a toric action of a toric manifold that is NEF, and for which the Theorem~\ref{Seidelelements} applies. This allows us to compute the Seidel element of $\rho$. 

In what follows, we consider Hamiltonian actions fixing spheres in the same homology classes as the ones in Proposition \ref{actionslist} and we list in Figure \ref{graphsgeneric} their graphs. Note that we use the same notation for the circle actions in the  generic case as  for the circle actions along the edge $MA$, we do not distinguish one case from the other with regard to notation.

\begin{figure}[thp]
\begin{minipage}{.45\textwidth}
\begin{tikzpicture}[scale=1.1,font=\footnotesize]
\node at (2.5,1.7 ) {$B +kF  -E_1 - E_2 -E_3 - E_4 $};
       \fill[black] (2,1) ellipse (1 and 0.2);
	\node at (5.1,1) {$1,\mu +k -c_1-c_2-c_3-c_4 $};
	\filldraw [black] (2,0.5) circle (2pt);
        \draw[dashed, black!60] (2,1) --(2,0.5); 
	\node at (1.7,0.6) {$ E_4$};
	 \draw[dashed, black!60] (2,1) --(2,-2); 
	\node at (1.7,-1.3) {$ F-E_4$};
	\filldraw [black] (2.2,0.1) circle (2pt);
	 \draw[dashed, black!60] (2.2,0.8) --(2.2,0.1); 
	\node at (2.4,0.4) {$ E_3$};
	 \draw[dashed, black!60] (2.2,0.8) --(2.2,-2); 
	\node at (2.7,-0.4) {$ F-E_3$};
	\node at (3.6,0.5) {$ 1-c_4$};
	\node at (3.6,0.1) {$ 1-c_3$};
	\node at (3.6,-0.6) {$ 1-c_2$};
	\node at (3.6,-1.4) {$ 1-c_1$};
	\filldraw [black] (1.4,-0.6) circle (2pt);
	 \draw[dashed, black!60] (1.4,-0.6) --(1.4,1); 
	\node at (1.1,0) {$ E_2$};
	 \draw[dashed, black!60] (1.4,-0.6) --(1.4,-2); 
	\node at (0.8,-1) {$ F-E_2$};
	\filldraw [black] (2.7,-1.4) circle (2pt);
	 \draw[dashed, black!60] (2.7,-1.4) --(2.7,1); 
	\node at (3,-1) {$ E_1$};
	\draw[dashed, black!60] (2.7,-1.4) --(2.7,-2); 
	\node at (3.3,-1.7) {$ F-E_1$};
        \fill[black] (2,-2) ellipse (1 and 0.2);
	\node at (3.7,-2) {$0,\mu - k$};
	 \node at (2,-2.7) {$B-kF$};
       \node at (2,-3.4) {Circle action $z_k$};
\end{tikzpicture}
\end{minipage}%
\begin{minipage}{.45\textwidth}
\begin{tikzpicture}[scale=1.1,font=\footnotesize]
\node at (2, 1.7 ) {$B +kF  -E_j - E_l -E_m $};
       \fill[black] (2,1) ellipse (1 and 0.2);
	\node at (4.8,1) {$\displaystyle{1,\mu +k  -c_j-c_\ell -c_m}$};
	\filldraw [black] (1.2,0.3) circle (2pt);
	\draw[dashed, black!60] (1.2,1) --(1.2,0.3); 
	\node at (0.9,0.6) {$ E_m$};
	\draw[dashed, black!60] (1.2,1) --(1.2,-2); 
	\node at (0.6,-1.5) {$ F-E_m$};	
	\filldraw [black] (2,-0.3) circle (2pt);
	\draw[dashed, black!60] (2,1) --(2,-0.3); 
	\node at (2.2,0.1) {$ E_\ell$};
	\draw[dashed, black!60] (2,1) --(2,-2); 
	\node at (2.3,-1.3) {$ F-E_\ell$};
	\filldraw [black] (2.6, -0.8) circle (2pt);
	\draw[dashed, black!60] (2.6,1) --(2.6,-0.8); 
	\node at (2.8,-0.4) {$ E_j$};
	\draw[dashed, black!60] (2.6,1) --(2.6,-2); 
	\node at (3.1,-1.7) {$ F-E_j$};
	\filldraw [black] (1.7,-1.2) circle (2pt);
	\draw[dashed, black!60] (1.7,-2) --(1.7,1); 
	\node at (1.7,-0.7) {$ F-E_i$};
	\draw[dashed, black!60] (1.7,-2) --(1.7,-1.2); 
	\node at (1.5,-1.5) {$ E_i$};
	\node at (3.6,0.3) {$ 1-c_m$};
	\node at (3.6,-0.3) {$ 1-c_\ell$};
	\node at (3.6,-0.8) {$ 1-c_j$};
	\node at (3.3,-1.2) {$ c_i$};
        \fill[black] (2,-2) ellipse (1 and 0.2);
	\node at (4.1,-2) {$\displaystyle{0,\mu -k - c_i}$};
	 \node at (2,-2.7) {$B-kF-E_i$};
       \node at (2,-3.4) {Circle action $z_{k,i}$};
\end{tikzpicture}
\end{minipage}%

\begin{minipage}{.45\textwidth}
\begin{tikzpicture}[scale=1.1,font=\footnotesize]
\node at (2,1.7 ) {$B +kF  -E_l - E_m$};
       \fill[black] (2,1) ellipse (1 and 0.2);
	\node at (4.4,1) {$1,\mu +k -c_\ell-c_m $};
	\filldraw [black] (1.2,0.3) circle (2pt);
	\draw[dashed, black!60] (1.2,1) --(1.2,0.3); 
	\node at (0.9,0.6) {$ E_m$};
	\draw[dashed, black!60] (1.2,-2) --(1.2,0.3); 
	\node at (0.6,-1) {$ F-E_m$};
	\filldraw [black] (2.7,-0.3) circle (2pt);
	\draw[dashed, black!60] (2.7,1) --(2.7,-0.3); 
	\node at (2.9,0.1) {$ E_\ell$};
	\draw[dashed, black!60] (2.7,-2) --(2.7,-0.3); 
	\node at (3.2,-1.6) {$ F-E_\ell$};
	\filldraw [black] (1.7, -0.8) circle (2pt);
	\draw[dashed, black!60] (1.7,-2) --(1.7,-0.8); 
	\node at (1.5,-1.3) {$ E_i$};
	\draw[dashed, black!60] (1.7,1) --(1.7,-0.8); 
	\node at (2.15,0.5) {$ F-E_i$};
	\filldraw [black] (2.1,-1.2) circle (2pt);
	\draw[dashed, black!60] (2.1,-2) --(2.1,-1.2); 
	\node at (2.3,-1.6) {$ E_j$};
	\draw[dashed, black!60] (2.1,1) --(2.1,-1.2); 
	\node at (2.4,-0.6) {$ F-E_j$};
	\node at (3.6,0.3) {$ 1-c_m$};
	\node at (3.6,-0.3) {$ 1-c_\ell$};
	\node at (3.3,-0.8) {$ c_i$};
	\node at (3.3,-1.2) {$ c_j$};
	 \fill[black] (2,-2) ellipse (1 and 0.2);
	\node at (4.4,-2) {$0,\mu - k-c_i-c_j$};
	 \node at (2,-2.7) {$B-kF- E_i-E_j$};
       \node at (2,-3.4) {Circle action $z_{k,ij}$};
\end{tikzpicture}
\end{minipage}%

\bigskip 

\begin{minipage}{.45\textwidth}
\begin{tikzpicture}[scale=1.1,font=\footnotesize]
\node at (2,1.7 ) {$B +kF  -E_m $};
       \fill[black] (2,1) ellipse (1 and 0.2);
	\node at (4.1,1) {$\displaystyle{1,\mu +k - c_m}$};
       \filldraw [black] (2.3,0.3) circle (2pt);
       \draw[dashed, black!60] (2.3,1) --(2.3,0.3); 
	\node at (2.6,0.6) {$ E_m$};
	 \draw[dashed, black!60] (2.3,-2) --(2.3,0.3); 
	\node at (2.9,-1.6) {$ F-E_m$};
	\filldraw [black] (1.1,-0.3) circle (2pt);
	\draw[dashed, black!60] (1.1,-2) --(1.1,-0.3); 
	\node at (0.9,-0.8) {$ E_i$};
	\draw[dashed, black!60] (1.1,1) --(1.1,-0.3); 
	\node at (0.6,0.5) {$ F-E_i$};
	\filldraw [black] (1.5, -0.8) circle (2pt);
	\draw[dashed, black!60] (1.5,-2) --(1.5,-0.8); 
	\node at (1.3,-1.3) {$ E_j$};
	\draw[dashed, black!60] (1.5,1) --(1.5,-0.8); 
	\node at (1,0) {$ F-E_j$};
	\filldraw [black] (1.8,-1.2) circle (2pt);
	\draw[dashed, black!60] (1.8,-2) --(1.8,-1.2); 
	\node at (2.1,-1.4) {$ E_\ell$};
	\draw[dashed, black!60] (1.8,1) --(1.8,-1.2); 
	\node at (2.3,-0.5) {$ F-E_\ell$};
	\node at (3.6,0.3) {$ 1-c_m$};
	\node at (3.3,-0.3) {$ c_i$};
	\node at (3.3,-0.8) {$ c_j$};
	\node at (3.3,-1.2) {$ c_\ell$};
	 \fill[black] (2,-2) ellipse (1 and 0.2);
	\node at (4.8,-2) {$\displaystyle{0,\mu -k - c_i-c_j- c_\ell}$};
	 \node at (2,-2.7) {$B-kF-E_i -E_j -E_l$};
       \node at (2,-3.4) {Circle action $z_{k,ijl}$};
\end{tikzpicture}
\end{minipage}%
\begin{minipage}{.45\textwidth}
\begin{tikzpicture}[scale=1.1,font=\footnotesize]
\node at (2,1.7 ) {$B +kF $};
       \fill[black] (2,1) ellipse (1 and 0.2);
	\node at (3.7,1) {$\displaystyle{1,\mu +k }$};
         \filldraw [black] (2.8,0.3) circle (2pt);
         \draw[dashed, black!60] (2.8,-2) --(2.8,0.3); 
	\node at (2.6,-0.2) {$ E_1$};
	 \draw[dashed, black!60] (2.8,1) --(2.8,0.3); 
	\node at (3.3,0.6) {$ F-E_1$};
	\filldraw [black] (1.2,-0.3) circle (2pt);
	\draw[dashed, black!60] (1.2,1) --(1.2,-0.3); 
	\node at (0.6,0.5) {$F- E_2$};
	\draw[dashed, black!60] (1.2,-2) --(1.2,-0.3); 
	\node at (0.9,-0.8) {$ E_2$};
	\filldraw [black] (1.5, -0.8) circle (2pt);
	\draw[dashed, black!60] (1.5,-2) --(1.5,-0.8); 
	\node at (1.7,-1.3) {$ E_3$};
	\draw[dashed, black!60] (1.5,1) --(1.5,-0.8); 
	\node at (1.7,0.5) {$ F-E_3$};
	\filldraw [black] (2.3,-1.2) circle (2pt);
	\draw[dashed, black!60] (2.3,-2) --(2.3,-1.2); 
	\node at (2.5,-1.6) {$ E_4$};
	\draw[dashed, black!60] (2.3,1) --(2.3,-1.2); 
	\node at (2.3,-0.7) {$ F-E_4$};
	\node at (3.3,0.3) {$ c_1$};
	\node at (3.3,-0.3) {$ c_2$};
	\node at (3.3,-0.8) {$ c_3$};
	\node at (3.3,-1.2) {$ c_4$};
        \fill[black] (2,-2) ellipse (1 and 0.2);
	\node at (5.1,-2) {$0,\mu - k-c_1-c_2-c_3-c_4$};
	 \node at (2,-2.7) {$B-kF- E_1-E_2-E_3-E_4$};
       \node at (2,-3.4) {Circle action $z_{k,1234}$};
\end{tikzpicture}
\end{minipage}%

\caption{Graphs of the circle actions $z_k$, $z_{k,i}$, $z_{k,ij}$, $z_{k,ijl}$ and  $z_{k,1234}$ in the generic case}
\label{graphsgeneric}
\end{figure}
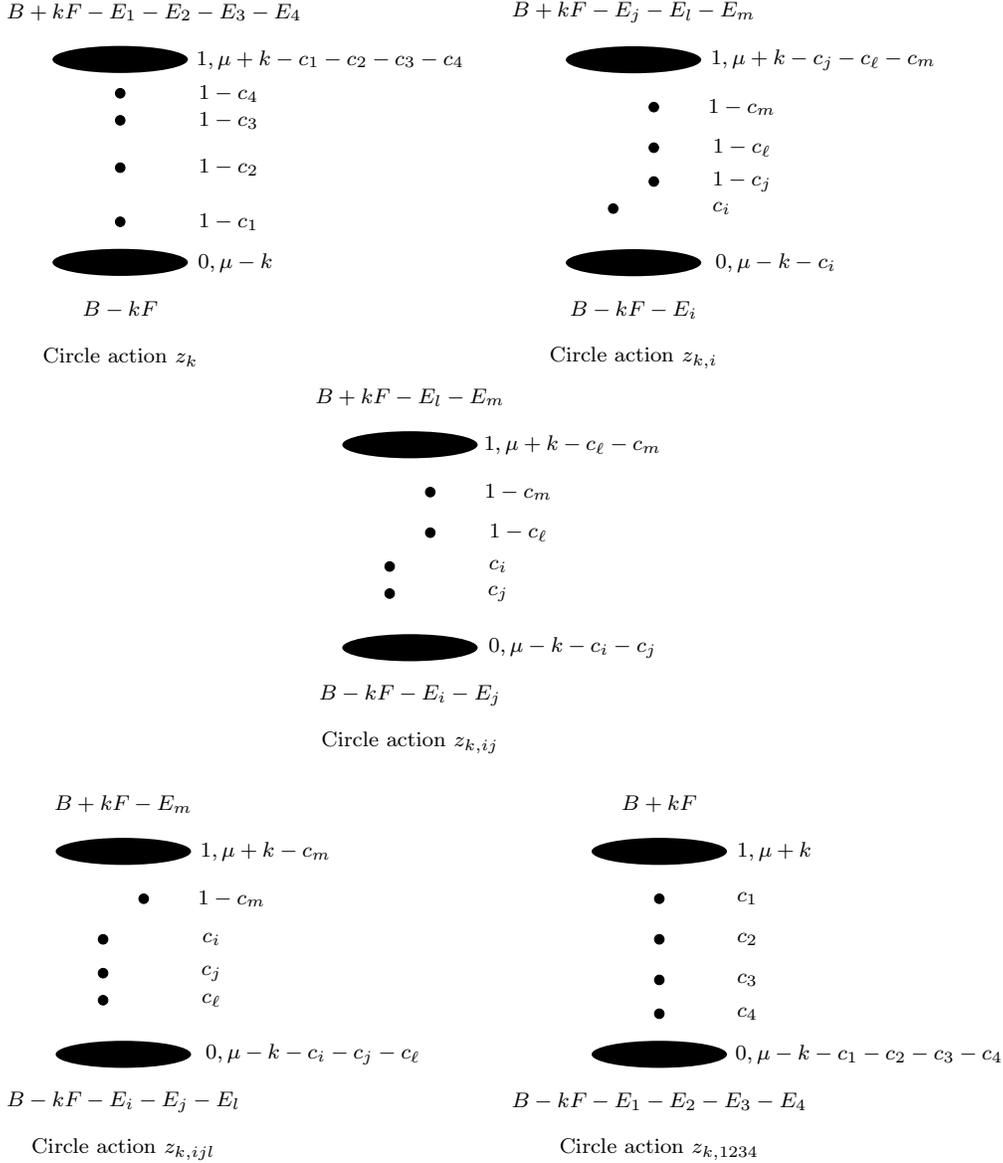

\subsubsection{Computing Seidel elements from toric actions}
First consider the actions $z_{0,12}, z_{0,13}, z_{0,14}$ and the polygon of Figure \ref{Toric picture T}. The action $z_{0,12}$ corresponds to the circle action whose moment map is the first component of the moment map associated to the toric action $T_{0,12}$, represented in this figure. Moreover, it is clear that the homology classes of the fixed spheres are $B-E_1-E_2$ and $B-E_3-E_4$. 
\begin{figure}[thp]
\begin{center}
\begin{tikzpicture}[scale=0.8, roundnode/.style={circle, draw=black!80, thick, minimum size=7mm}, font=\footnotesize]
 
    \draw[->][black!70] (-3.2,0) -- (1,0); 
    \draw[->][black!70] (-3,-0.2) -- (-3,5.5); 

	\draw[brown] (-3,1.5) -- (-1.5,0);
	\draw[brown] (-3,1.5) -- (-3,3.8);
    \draw[brown] (-3,3.8) -- (-1.8,5);
    \draw[brown] (-1.8,5) -- (-1,5);
     \draw[brown] (-1,5) -- (0,4);
    \draw[brown] (0,4) -- (0,0.3);
    \draw[brown] (-1.5,0) -- (-0.3,0);
    \draw[brown] (-0.3,0) -- (0,0.3);
   
	\draw[dashed, black!60] (-1.8,5) --(-1.8,0);
	\draw[dashed, black!60] (-1,5) -- (-1,0);
    \draw[dashed, black!60] (0,0.2) -- (0,0);
	\node[rotate=-45] at (-1.3,-0.3) {$ c_{1} $};
	\node[rotate=-45] at (-0.6,-0.6) {$ 1-c_{3} $};
	\node[rotate=-45] at (-1.8,-0.3) {$  c_{2} $};
	\node[rotate=-45] at (-0, -0.6) {$ 1-c_{4} $};
	\node[blue] at (0.2, 0) {$ E_4$};
	 \node[blue] at (-2.5,.7) {$E_1$};
 \node[blue] at (-2.6,4.6) {$E_2 $};
\node[blue] at (-0.2,4.6) {$E_3 $};
\node[blue] at (1.2,2.1) {$B-E_3-E_4$};
\node[blue] at (-4.2,2.4) {$B-E_1-E_2$};

\end{tikzpicture}
\end{center}
\caption{Toric action $T_{0,12}$}
 \label{Toric picture T}
\end{figure}
Then Theorem \ref{Seidelelements}.(2a) yields
\begin{equation*}
\mathcal S (z_{0,12})= [B-E_3-E_4] \otimes q \, \frac{t^{\epsilon}}{1- t^{c_3+c_ 4-\mu}}
\end{equation*}
where $\epsilon$ is the maximum of the momentum map of the action $z_{0,12}$, $\phi_{\rm max}(F_{\rm max})$, where $F_{\rm max}$ is the maximal 2-sphere whose momentum image is the edge in class $B-E_3-E_4$, in the normalized polygon. 
In general, we obtain 
\begin{equation*}
\mathcal S (z_{0,1i})= [B-E_j-E_\ell] \otimes q \, \frac{t^{\epsilon}}{1- t^{c_j+c_ \ell-\mu}} \quad \mbox{where} \quad j \neq \ell \neq i
\end{equation*}
One can check that the normalized polygon yields
$$\epsilon = \frac{c_j^3+ 3c_1^2-c_1^3+c_\ell^3+3c_i^2-c_i^3-3\mu}{3(c_1^2+c_2^2+c_3^2+c_4^2-2\mu)}.$$
Hence, if $c_i=1/2$ for all $i$ we obtain $\epsilon = 1/2$ and 
\begin{equation}\label{Seidel elements}
\mathcal S (z_{0,1i})= [B-E_j-E_\ell] \otimes q \, \frac{t^{\frac12}}{1- t^{1-\mu}} \quad \mbox{where} \quad j \neq \ell \neq i
\end{equation}
Note that the expression is well defined because $\mu >1$. 

Consider now the polygon of Figure \ref{ToricT1} which represents a toric action on $\Mucccc$, and for which the homology classes of the edges are represented in the figure. The graph of the circle action obtained by projection of the polygon onto the $x$-axis  is also represented in Figure  \ref{ToricT1}.  Note that it becomes the action $z_1$ defined in Figure \ref{firstcase} if $c_i=1/2$ for all $i$.

\begin{figure}[thp]
\begin{minipage}{.48\textwidth}
\begin{tikzpicture}[scale=0.8, roundnode/.style={circle, draw=black!80, thick, minimum size=7mm}, font=\footnotesize]
 
    \draw[->][black!70] (-3.2,0) -- (1,0); 
    \draw[->][black!70] (-3,-2) -- (-3,6.2); 

	\draw[brown]  (-3,0) -- (-1.5,-1.5);
	\draw[brown]  (-3,0) -- (-3,4);
    \draw[brown]  (-3,4) -- (-1.2,5.8);
    \draw[brown]  (-1.2,5.8) -- (-0.3,5.8);
     \draw[brown]  (-0.3,5.8) -- (0,5.5);
    \draw[brown]  (0,5.5) -- (0,-1);
    \draw[brown]  (0,-1) -- (-0.5,-1.5);
    \draw[brown]  (-0.5,-1.5) -- (-1.5,-1.5);
   
       \draw[dashed, black!60] (-1.5,-1.5) --(-1.5,0);
	\draw[dashed, black!60] (-0.5,-1.5) -- (-0.5,0);
	\draw[dashed, black!60] (-1.2,5.8) -- (-1.2,0);
	\draw[dashed, black!60] (-0.3,5.8) -- (-0.3,0);

    \node[blue] at (0.2,5.8) {$E_{4}$};
    \node[blue] at (-1.1,6.1) {$E_2-E_4$};
    \node[blue] at (-3.8,2) {$B-F$};
     \node[blue] at (2.4,2.5) {$B+F -E_1-E_2-E_3-E_4$};
     \node[blue] at (0.1,-1.35) {$E_3$};
     \node[blue] at (-1,-1.8) {$E_1-E_3$};
     \node[blue] at (-2.8,5.1) {$ F-E_2 $};
     \node[blue] at (-2.8,-1.1) {$F-E_1  $};
      \node[rotate=-45] at (-1.9,0.5) {$1-c_1$};
   \node[rotate=-45] at (-0.8,-0.5) {$1-c_2$};
    \node[rotate=-45] at (-0.8,0.5) {$1-c_3$};
    \node[rotate=-45] at (0,-0.5) {$1-c_4$};
\end{tikzpicture}
\end{minipage}%
\begin{minipage}{.48\textwidth}
\begin{tikzpicture}[scale=1.1,font=\footnotesize]
\node at (2.5,1.7 ) {$B +F  -E_1 - E_2 -E_3 - E_4 $};
       \fill[black] (2,1) ellipse (1 and 0.2);
	\node at (5.3,1) {$1,\mu +1 -c_1-c_2-c_3-c_4 $};
	\filldraw [black] (2,0.5) circle (2pt);
        \draw[dashed, black!60] (2,1) --(2,0.5); 
	\node at (1.7,0.6) {$ E_4$};
	 \draw[dashed, black!60] (2,1) --(2,-2); 
	\node at (1.7,-1.3) {$ F-E_4$};
	\filldraw [black] (2.2,0.1) circle (2pt);
	 \draw[dashed, black!60] (2.2,0.8) --(2.2,0.1); 
	\node at (2.4,0.4) {$ E_3$};
	 \draw[dashed, black!60] (2.2,0.8) --(2.2,-2); 
	\node at (2.7,-0.4) {$ F-E_3$};
	\node at (4,0.5) {$ 1-c_4$};
	\node at (4,0.1) {$ 1-c_3$};
	\node at (4,-0.6) {$ 1-c_2$};
	\node at (4,-1.4) {$ 1-c_1$};
	\filldraw [black] (1.4,-0.6) circle (2pt);
	 \draw[dashed, black!60] (1.4,-0.6) --(1.4,1); 
	\node at (1.1,0) {$ E_2$};
	 \draw[dashed, black!60] (1.4,-0.6) --(1.4,-2); 
	\node at (0.8,-1) {$ F-E_2$};
	\filldraw [black] (2.7,-1.4) circle (2pt);
	 \draw[dashed, black!60] (2.7,-1.4) --(2.7,1); 
	\node at (2.9,-1) {$ E_1$};
	\draw[dashed, black!60] (2.7,-1.4) --(2.7,-2); 
	\node at (3.2,-1.7) {$ F-E_1$};
        \fill[black] (2,-2) ellipse (1 and 0.2);
	\node at (4.1,-2) {$0,\mu - 1$};
	 \node at (2,-2.7) {$B-F$};
       \node at (2,-3.4) {Circle action $z_1$};
\end{tikzpicture}
\end{minipage}%
\caption{Toric action $T_1$ and its projection to the $x$-axis}
 \label{ToricT1}
 \end{figure}

Now  Theorem \ref{Seidelelements}.(2a) gives 
\begin{equation*}
\mathcal  S(z_1) = [B+F-E_1-E_2-E_3-E_4] \otimes q \, \frac{t^{1-\epsilon}}{1- t^{c_1+c_2+c_3+c_4-\mu-1}}
\end{equation*}
where in this case the maximum of the momentum map on the invariant sphere is given by
$$ \epsilon = \frac{-1 -c_1^3 + 3c_1^2 +3c_2^2-c_2^3 +3c_3^2 -c_3^3+3c_4^2-c_4^3+ 3c_1c_4^2-3c_3c_4^2-3\mu}{3(c_1^2+c_2^2+c_3^2+c_4^2-2\mu)}$$
which is simply equal to $1/2$ if $c_i=1/2$ for all $i$. Therefore, we obtain 
\begin{equation}\label{Seidelz1}
\mathcal S(z_1)= [B+F-E_1-E_2-E_3-E_4] \otimes q \, \frac{t^{\frac12}}{1- t^{1-\mu}}.
\end{equation}

Finally, we compute the Seidel element of the circle action $z_{1,4}$, seen as an element of the fundamental group of $\Symp_h (\Mucccc)$. In order to do that first consider the Delzant 
polygon of Figure \ref{non-nef}. It represents a toric action on $\Mucccc$ and its projections onto the $x$-axis and $y$-axis are represented  in Figure \ref{graphsnon-nef}. 
Note that the projection onto the $x$-axis corresponds to the graph of the action $z_{1,4}$  given in Figure \ref{secondcase}.  Let us denote the action  whose graph is obtained by projection  onto the $y$-axis by $s_{1,4}$. 

\begin{figure}[thp]

\begin{tikzpicture}[scale=0.7, roundnode/.style={circle, draw=black!80, thick, minimum size=7mm}, font=\footnotesize]
 
    \draw[->][black!70] (-3.2,0) -- (1,0); 
    \draw[->][black!70] (-3,-3) -- (-3,6.2); 

 \draw[brown] (-2.5,-0.5) -- (-0.3, -2.5);
	\draw[brown] (-3,0.5) -- (-2.5,-0.5);
	\draw[brown] (-3,0.5) -- (-3,4);
    \draw[brown] (-3,4) -- (-1.5,5.5);
    \draw[brown] (-1.5,5.5) -- (-0.5,5.5);
     \draw[brown] (-0.5,5.5) -- (0,5);
    \draw[brown] (0,5) -- (0,-2.5);
    \draw[brown] (-0.3,-2.5) -- (0,-2.5);
   
        \draw[dashed, black!60] (-1.5,5.5) --(-3,5.5);
	\draw[dashed, black!60] (-3,5) -- (0,5);
	\draw[dashed, black!60] (-3,-0.5) -- (-2.5,-0.5);
	\draw[dashed, black!60] (-0.3,-2.5) -- (-3,-2.5);
	\draw[dashed, black!60] (-0.3,-2.5) -- (-0.3,0);	
	\draw[dashed, black!60] (-2.5,0) -- (-2.5,-0.5);
	 \draw[dashed, black!60] (-1.5,5.5) --(-1.5,0);
	 \draw[dashed, black!60] (-0.5,5.5) --(-0.5,0);
	 
    \node[blue] at (0.1,5.4) {$E_{2}$};
    \node[blue] at (-1.1,5.8) {$E_1-E_2$};
    \node[blue] at (-4.2,2) {$B-F-E_4$};
     \node[blue] at (2.5,2) {$B+F -E_1-E_2-E_3$};
     \node[blue] at (-0.2,-2.8) {$E_3$};
     \node[blue] at (-2.8,-1.5) {$F- E_3-E_4$};
     \node[blue] at (-2.7,5.25) {$ F-E_1$};
     \node[blue] at (-2.95, -0.2) {$E_4$};
    \node at (-3.3 ,0.5) {$c_4$};
    \node at (-3.5 ,-0.5) {$-c_4$};
    \node at (-3.9 ,-2.5) {$-1+c_3$};
    \node at (-3.6 ,4) {$\mu-1$};
    \node at (-4.2 ,4.9) {$\mu-c_1-c_2$};
      \node at (-3.7 ,5.7) {$\mu-c_1$};
       \node[rotate=45] at (-2.4 ,0.25) {$c_4$};
         \node[rotate=45] at (-1.2 ,0.55) {$1-c_1$};
          \node[rotate=45] at (-0.2 ,0.55) {$1-c_2$};
                \node[rotate=45] at (0.2 ,0.55) {$1-c_3$};
      
\end{tikzpicture}

\caption{Toric action $(z_{1,4}, s_{1,4})$}
 \label{non-nef}
 \end{figure}

\begin{figure}[thp]
\begin{minipage}{.48\textwidth}
\begin{tikzpicture}[scale=1.1,font=\footnotesize]
\node at (2, 1.7 ) {$B +F  -E_1 - E_2 -E_3 $};
       \fill[black] (2,1) ellipse (1 and 0.2);
	\node at (4.8,1) {$\displaystyle{1,\mu +1  -c_1-c_2 -c_3}$};
	\filldraw [black] (1.2,0.3) circle (2pt);
	\draw[dashed, black!60] (1.2,1) --(1.2,0.3); 
	\node at (0.9,0.6) {$ E_3$};
	\draw[dashed, black!60] (1.2,1) --(1.2,-2); 
	\node at (0.6,-1.5) {$ F-E_3$};	
	\filldraw [black] (2,-0.3) circle (2pt);
	\draw[dashed, black!60] (2,1) --(2,-0.3); 
	\node at (2.2,0.1) {$ E_2$};
	\draw[dashed, black!60] (2,1) --(2,-2); 
	\node at (2.3,-1.3) {$ F-E_2$};
	\filldraw [black] (2.6, -0.8) circle (2pt);
	\draw[dashed, black!60] (2.6,1) --(2.6,-0.8); 
	\node at (2.8,-0.4) {$ E_1$};
	\draw[dashed, black!60] (2.6,1) --(2.6,-2); 
	\node at (3.1,-1.7) {$ F-E_1$};
	\filldraw [black] (1.7,-1.2) circle (2pt);
	\draw[dashed, black!60] (1.7,-2) --(1.7,1); 
	\node at (1.7,-0.7) {$ F-E_4$};
	\draw[dashed, black!60] (1.7,-2) --(1.7,-1.2); 
	\node at (1.5,-1.5) {$ E_4$};
	\node at (3.9,0.3) {$ 1-c_3$};
	\node at (3.9,-0.3) {$ 1-c_2$};
	\node at (3.9,-0.8) {$ 1-c_1$};
	\node at (3.7,-1.2) {$ c_4$};
        \fill[black] (2,-2) ellipse (1 and 0.2);
	\node at (4.3,-2) {$\displaystyle{0,\mu -1 - c_4}$};
	 \node at (2,-2.7) {$B-F-E_4$};
       \node at (2,-3.4) {Circle action $z_{1,4}$};
\end{tikzpicture}
\end{minipage}%
\begin{minipage}{.48\textwidth}
\quad \quad 
\begin{tikzpicture}[scale=1.1,font=\footnotesize]
\node at (2,2.4 ) {$E_1-E_2 $};
       \fill[black] (2,1.8) ellipse (1 and 0.2);
	\node at (4.6,1.8) {$\displaystyle{\mu-c_1, c_1-c_2}$};
	\filldraw [black] (1.5,0.4) circle (2pt);
	 \draw[dashed, black!60] (1.5,0.4) --(1.5,1.8); 
	\node at (1,0.9) {$ F-E_1$};
	\filldraw [black] (2.5,1) circle (2pt);
	 \draw[dashed, black!60] (2.5,-1.5) --(2.5,1.8); 
	\node at (2.8,1.3) {$ E_2$};
	 \draw[dashed, black!60] (2.5,1) --(2.5,1.8); 
	\node at (2.8,1.3) {$ E_2$};
	\node at (4.4 ,1) {$\mu -c_1-c_2 $};
        \node at (4 ,0.4) {$\mu -1$};
	\filldraw [black] (1.5,-0.2) circle (2pt);
	\draw[dashed, black!60] (1.5,0.4) --(1.5, -0.2); 
	\filldraw [black] (1.5,-0.8) circle (2pt);
	\node at (3.8 ,-0.2) {$c_4$};
        \node at (3.9 ,-0.8) {$-c_4$};
        \fill[black] (2,-1.5) ellipse (1 and 0.2);
	\node at (4.4,-1.5) {$\displaystyle{-1+c_3,c_3}$};
	 \node at (2,-2) {$E_3$};
	  \node at (0.7,0.1) {$B-F-E_4$};
	   \node at (1.8,-0.5) {$E_4$};
	      \node at (3.9,-0.5) {$B+F -E_1-E_2-E_3$};
	  \draw (1.5,-0.2) -- (1.5,-0.8);
	    \node at (1.3,-0.5) {$2$};
	    \draw[dashed, black!60] (1.5,-0.8) --(1.5, -1.5); 
	      \node at (0.7,-1) {$F-E_3- E_4$};

      \node at (2,-2.7) {Circle action $s_{1,4}$};
\end{tikzpicture}
\end{minipage}%

\caption{Graphs of circle actions $z_{1,4}$ and $s_{1,4}$, respectively. }
 \label{graphsnon-nef}
 \end{figure}

Since $c_1(B-F-E_4)= -1 <0$ the complex manifold corresponding to  Figure \ref{non-nef} is not NEF and we cannot apply immediately Theorem \ref{Seidelelements} to compute  the Seidel element of $z_{1,4}$. Instead we need to consider some auxiliary polygons for which the underlying complex manifolds are NEF  and relate the circle actions represented on those polygons with the actions $z_{1,4}$ and $s_{1,4}$. More precisely, consider the Delzant polygon, on the left in Figure \ref{nef} and apply the $GL(2, \Z)$ transformation represented by the matrix 
\begin{equation}\label{matrix}
 \left (
\begin{array}{cc}
1  & 0 \\
1 & 1 
\end{array} \right)
\end{equation} 
 to this polygon as well as to the polygon of Figure \ref{non-nef}. Then consider the projection onto the $y$-axis of the two transformed polygons and 
 denote the action obtained this way from polygon of Figure \ref{nef} by $t_{1,4}$. It is easy to check that the two graphs coincide 
  which implies that as elements of $\pi_1(\Symp_h (\Mucccc))$ the following identification holds
 \begin{equation}\label{eqnon-nef}
 z_{1,4}+ s_{1,4} = t_{1,4}. 
\end{equation}

\begin{figure}[thp]
\begin{minipage}{.40\textwidth}
\begin{tikzpicture}[scale=0.7, roundnode/.style={circle, draw=black!80, thick, minimum size=7mm}, font=\footnotesize]
 
    \draw[->][black!70] (-3.2,0) -- (0.7,0); 
    \draw[->][black!70] (-3,-2) -- (-3,5.8); 

 \draw[brown] (-1.5,-1.5) -- (-0.5, -2);
	\draw[brown] (-3,0) -- (-1.5,-1.5);
	\draw[brown] (-3,0) -- (-3,4);
    \draw[brown] (-3,4) -- (-1.5,5.5);
    \draw[brown] (-1.5,5.5) -- (-0.5,5.5);
     \draw[brown] (-0.5,5.5) -- (0,5);
    \draw[brown] (0,5) -- (0,-2);
    \draw[brown] (-0.5,-2) -- (0,-2);
    \node[blue] at (0.1,5.4) {$E_{2}$};
    \node[blue] at (-1.1,5.8) {$E_1-E_2$};
    \node[blue] at (-3.8,2) {$B-F$};
     \node[blue] at (2.3,2) {$B+F -E_1-E_2-E_3$};
     \node[blue] at (0.2,-2.3) {$E_3-E_4$};
     \node[blue] at (-3.4,-0.9) {$F- E_3-E_4$};
     \node[blue] at (-2.8,5.1) {$ F-E_1$};
     \node[blue] at (-1.3, -1.9) {$E_4$};

\end{tikzpicture}
\end{minipage}%
\begin{minipage}{.40\textwidth}

\begin{tikzpicture}[scale=0.7, roundnode/.style={circle, draw=black!80, thick, minimum size=7mm}, font=\footnotesize]
 
    \draw[->][black!70] (-3.2,0) -- (0.7,0); 
    \draw[->][black!70] (-3,-0.3) -- (-3,6.2); 

 \draw[brown] (-1.5,0) -- (-1, 0.25);
	\draw[brown] (-3,0) -- (-1.5,0);
	\draw[brown] (-3,0) -- (-3,4);
    \draw[brown] (-3,4) -- (-1.5,5.5);
    \draw[brown] (-1.5,5.5) -- (-0.5,6);
     \draw[brown] (-0.5,6) -- (0,6);
    \draw[brown] (0,6) -- (0,1.25);
    \draw[brown] (-1,0.25) -- (0,1.25);
   
       \draw[dashed, black!60] (-3,6) --(-0.5,6);
    \draw[dashed, black!60] (-3,5.5) -- (-1.5,5.5);
    \draw[dashed, black!60] (-3,1.25) -- (0,1.25);
    \draw[dashed, black!60] (-3,0.25) -- (-1,0.25);

  \node at (-4.5,6) {$\mu+1-c_1-c_2$};
  \node at (-4.1,5.5) {$\mu+1-2c_1$};
    \node at (-3.6,4) {$\mu-1$};
      \node at (-3.3,1.25) {$c_3$};
        \node at (-3.3,0.25) {$c_4$};
        
    \node[blue] at (-0.1,6.25) {$E_{2}$};
    \node[blue] at (-2,5.8) {$E_1-E_2$};
    \node[blue] at (-3.7,2) {$B-F$};
     \node[blue] at (2.3,3) {$B+F -E_1-E_2-E_3$};
     \node[blue] at (0.2,0.5) {$E_3- E_4$};
     \node[blue] at (-2.8,-0.3) {$F- E_3-E_4$};
     \node[blue] at (-2.8,5.1) {$ F-E_1$};
     \node[blue] at (-1.1, -0.1) {$E_4$};

\end{tikzpicture}
\end{minipage}%

\caption{Polygon representing a NEF complex manifold and its transformation by the $GL(2, \Z)$ matrix  \eqref{matrix}}
 \label{nef}
 \end{figure}

On the other hand consider the polygon on the left in Figure \ref{auxiliarynef} which represents a toric action, denoted by $(x_1,y_1)$, on $\Mucccc$ and its transformation by the same $GL(2,\Z)$  matrix \eqref{matrix}. The projection of the transformed polygon onto the $y$-axis yields a graph that, up to translation, coincides with the graph of the action $s_{1,4}$ which  implies that $s_{1,4}=x_1+y_1.$

\begin{figure}[thp]
\begin{minipage}{.40\textwidth}
\begin{tikzpicture}[scale=0.7, roundnode/.style={circle, draw=black!80, thick, minimum size=7mm}, font=\footnotesize]
 
    \draw[->][black!70] (-3.2,0) -- (0.7,0); 
    \draw[->][black!70] (-3,-0.2) -- (-3,5.5); 

	\draw[brown] (-3,0.7) -- (-2.3,0);
	\draw[brown] (-3,0.7) -- (-3,5);
    \draw[brown] (-3,5) -- (-1.5,5);
    \draw[brown] (-1.5,5) -- (-1,4.75);
     \draw[brown] (-1,4.75) -- (0,3.75);
    \draw[brown] (0,3.75) -- (0,0.4);
    \draw[brown] (-2.3,0) -- (-0.4,0);
    \draw[brown] (-0.4,0) -- (0,0.4);
    \node[blue] at (-3.1,0.3) {$ E_3$};
    \node[blue] at (-3.8,2.7) {$B-E_3$};
	\node[blue] at (-2.3,5.25) {$F-E_2$};
	\node[blue] at (-0.9,5) {$E_2 $};
	\node[blue] at (0.3,4.5) {$E_1 -E_2 $};
	\node[blue] at (1.4,2.5) {$ B-E_1 -E_4$};
	\node[blue] at (-1.6,-0.3) {$  F-E_3-E_4$};
	\node[blue] at (0.2, 0) {$ E_4$};

\end{tikzpicture}
\end{minipage}%
\begin{minipage}{.40\textwidth}

\begin{tikzpicture}[scale=0.7, roundnode/.style={circle, draw=black!80, thick, minimum size=7mm}, font=\footnotesize]
 
    \draw[->][black!70] (-3.2,0.5) -- (0.7,0.5); 
    \draw[->][black!70] (-3,0.3) -- (-3,6.2); 

 \draw[brown] (-2,1) -- (-0.5, 2.5);
	\draw[brown] (-3,1) -- (-2,1);
	\draw[brown] (-3,1) -- (-3,4);
    \draw[brown] (-3,4) -- (-1.5,5.5);
    \draw[brown] (-1.5,5.5) -- (-1,5.75);
     \draw[brown] (-1,5.75) -- (0,5.75);
    \draw[brown] (0,5.75) -- (0,3.5);
    \draw[brown] (-0.5,2.5) -- (0,3.5);
   
       \draw[dashed, black!60] (-3,5.75) --(-1,5.75);
       \draw[dashed, black!60] (-3,5.5) --(-1.5,5.5);
        \draw[dashed, black!60] (-3,3.5) --(0,3.5);
       \draw[dashed, black!60] (-3,2.5) --(-0.5,2.5);

  \node at (-4.5,5.4) {$\mu+1-c_1-c_2$};
   \node at (-4.1,5.8) {$\mu+1-c_1$};
 \node at (-3.4,4) {$\mu$};
  \node at (-3.7,3.5) {$1+c_4$};
   \node at (-3.7,2.5) {$1-c_4$};
    \node at (-3.4,1) {$c_3$};
    \node[blue] at (-0.1,6) {$E_1-E_{2}$};
    \node[blue] at (-1.5,5.8) {$E_2$};
    \node[blue] at (-3.7,1.8) {$B-E_3$};
     \node[blue]at (1.4,4.3) {$B-E_1-E_4$};
     \node[blue] at (-2.5,0.7) {$E_3$};
     \node[blue] at (0.3,1.8) {$F- E_3-E_4$};
     \node[blue] at (-3,4.9) {$ F-E_2$};
     \node[blue] at (0.1, 3) {$E_4$};

\end{tikzpicture}
\end{minipage}%
\caption{Toric action $(x_1,y_1)$ and its transformation by the $GL(2, \Z)$ matrix  \eqref{matrix}}
 \label{auxiliarynef}
 \end{figure}

Note that in  Figure \ref{auxiliarynef} we have $c_1(A) \geq 0$ for all homology classes $A$ of the fixed spheres corresponding to edges of the polygon so we can apply Theorem \ref{Seidelelements} to compute the 
Seidel element of $s_{1,4}=x_1+y_1$. Moreover, it follows from equation \eqref{eqnon-nef} that the Seidel element of $z_{1,4}$ is given by 
\begin{equation}\label{Seidelelementz14}
\mathcal S(z_{1,4})= \mathcal S(t_{1,4}) \mathcal S(s_{1,4})^{-1}
\end{equation}
More precisely, Theorem \ref{Seidelelements}(3a) yields  the Seidel elements of $t_{1,4}$ and $-s_{1,4}$ which are readable directly from the polygons on the right in Figures \ref{nef}  and  \ref{auxiliarynef}.
\begin{align*}
& \mathcal S(t_{1,4})  = E_2 \otimes q t^{\mu +1 -c_1-c_2 -\gamma} - (E_1-E_2) \otimes q \frac{t^{\mu + 1-2c_1-\gamma}}{1-t^{c_2-c_1}}\\
& \mathcal S(s_{1,4})^{-1}  = E_3 \otimes q t^{\beta-c_3} - (F-E_3-E_4)  \otimes q \frac{t^{\beta +c_4-1}}{1-t^{c_3+c_4-1}}
\end{align*}
where the exponents of the higher degree terms, namely of $E_2$ and $E_3$,  are the maximal values of the momentum map for the normalized polygons. Therefore one can check that 
$$ \gamma = \frac{-1+ 3c_1^2 -3c_1^3 +3c_2^2 -3c_1c_2^2 -2c_2^3 +c_3^3 +c_4^3 +3c_1^2\mu+3c_2^2\mu -3\mu^2}{3(c_1^2+c_2^2+c_3^2+c_4^2-2\mu)}$$
and 
$$ \beta=\frac{3c_1^2 -2c_1^3 +3c_2^2 -3c_1c_2^2 -c_2^3 +2c_3^3 +3c_4^2-3\mu +3c_1^2\mu+3c_2^2\mu -3\mu^2}{3(c_1^2+c_2^2+c_3^2+c_4^2-2\mu)}.$$
Using \cite[Proposition 5.3]{CM} we can now compute the quantum product \eqref{Seidelelementz14} (we leave the details to the interested reader since it is a long and boring computation) and finally obtain 
\begin{equation*}
\mathcal S(z_{1,4}) =((B+F -E_1- E_2-E_3)\otimes q +  \mathbb{1}\otimes t^{c_4 - \mu})t^{\beta- \gamma},
\end{equation*}
where the identity $ \mathbb{1}$ is the homology class of the manifold, $[\bbcp^2\#\, 5\overline{ \bbcp^2}] \in H_4(\bbcp^2\#\, 5\overline{ \bbcp^2}, \Z)$. In what follows we will suppress the identity from the expressions in order to simplify the notation. 

If $c_i=1/2$ for all $i$ then 
\begin{equation}\label{finalSeidelelement}
\mathcal S(z_{1,4}) = ((B+F -E_1- E_2-E_3)\otimes q +  t^{\frac12 - \mu})t^{\frac{2-3\mu}{3(1-2\mu)}}. 
\end{equation}
Note that this result agrees with McDuff-Tolman's result as the leading term is given by the homology class of the edge where the action is maximal. 

\begin{remark}
In a similar way we can compute the Seidel element of the inverse of $z_{1,4}$:
\begin{equation}\label{inversefinalSeidelelement}
\begin{split}
\mathcal S(z_{1,4})^{-1} = [(B-F -E_4)\otimes q + [pt] \otimes q^2 \, t^{\frac32 - \mu}  + (3F -E_1-E_2-E_3 +E_4)\otimes q\,  t^{1 - \mu} \\  + (B-E_4)\otimes q \,  t^{2 - 2\mu} + t^{\frac12 - \mu}(1 +  t^{2 - 2\mu})] \frac{t^{\frac{1-3\mu}{3(1-2\mu)}}}{(1-t^{1 - \mu})^4}.
\end{split}
\end{equation}
Since all the possible Hamiltonian circle actions on $\Muhalf$ have semifree maximal sets, McDuff-Tolman result always applies. Observe that when $\mu > \frac32$, the leading term is the homology class $B+F -E_4$, corresponding to the homology class of the edge where the moment map is maximal, in accordance with McDuff-Tolman formula. Otherwise, if  $1 < \mu \leq \frac32$, the leading term is given by the class of a point, showing that the quantum homology class \eqref{inversefinalSeidelelement} is not the image of an effective Hamiltonian action under the Seidel homomorphism.
\end{remark}

\subsubsection{Injectivity of the Seidel morphism for \texorpdfstring{$\mu>3/2$}{mu3/2}}

\begin{proposition}\label{free-subgroup-of-rank-5}
Consider the circle actions $z_{0,1i}, i=2,3,4$, $z_1$ and $z_{1,4}$, defined in Figure \ref{graphMA}.  If $\mu > \frac32$ and $c_1=c_2=c_3=c_4=1/2$ then the Seidel elements of these 5 circle actions generate a free subgroup of rank 5 in the group of invertible elements of the quantum homology.
\end{proposition}
\begin{proof}
Using the notation of Section \ref{QuantumHomology} for the quantum homology ring, the Seidel elements of these 5 circle actions are given by the following expressions: 
\begin{gather}\label{Seidel_elements}
\mathcal S (z_{0,12}) = b_{34}, \quad  \mathcal S  (z_{0,13}) = b_{24},  \quad \mathcal S  (z_{0,14}) = b_{23}, \\
\mathcal S (z_1) = b_{12} + f_{34}, \quad 
\mathcal S (z_{1,4}) = \left( b_{12} + f_{34}+e_4 + \frac{t^{1-\mu}}{1-t^{1-\mu}}\right) (1-t^{1-\mu})t^{\frac{1}{6(1-2\mu)}}.
\end{gather}
In order to show they are linearly independent we first consider a simplification of the quantum algebra, namely, we set the $b_{ij},$ for all $i,j$, equal to an element $b$, the $f_{ij}$ to $f$, and $e_j$ to $e$. Then the quantum homology algebra becomes isomorphic to the $\Pi^{\rm univ}$-algebra
\begin{equation}\label{simplified_ring}
\Pi^{\rm univ} [f,b,e]/ I'
\end{equation}
where $I'$ is the ideal generated by 
\begin{gather*}
(1) \ f^2 = 0, \qquad (2) \ b^2 =1, \qquad (3) \ f(b+1)=0,   \\
(4) \ f \left( e + \frac{t^{1-\mu}}{1-t^{1-\mu}} \right) = 0, \qquad 
 (5) \ b \left( e + \frac{t^{1-\mu}}{1-t^{1-\mu}} \right) = f+ e + \frac{t^{1-\mu}}{1-t^{1-\mu}}, \\
(6) \ e^2 = 2(b+f+e) \frac{t^{1-\mu}}{1-t^{1-\mu}} + 2\frac{t^{1-\mu}}{(1-t^{1-\mu})^2} + \frac{t^{2-2\mu}}{(1-t^{1-\mu})^2}.
\end{gather*}

Moreover, it is clear that the Seidel elements simplify to 
\begin{gather*}
\mathcal S(z_{0,12}) =  \mathcal S  (z_{0,13}) =  \mathcal S  (z_{0,14}) = b, \quad \mathcal S(z_1) = b +f, \\
\mbox{and} \quad \mathcal S(z_{1,4}) = \left( b + f + e + \frac{t^{1-\mu}}{1-t^{1-\mu}}\right)(1-t^{1-\mu})\, t^{\frac{1}{6(1-2\mu)}}.
\end{gather*}
We postpone the proof of the next two lemmas to Appendix \ref{QHcomp}.
\begin{lemma}\label{li_simple_ring}
Consider the elements $b$, $b+f$ and $\left( b + f + e + \frac{t^{1-\mu}}{1-t^{1-\mu}}\right)(1-t^{1-\mu})\, t^{\frac{1}{6(1-2\mu)}}$ contained in the subgroup of invertible elements of the algebra $ \Pi^{\rm univ} [f,b,e]/ I'$. They are linearly independent, that is, if 
$$ b^\alpha \,  (b+f)^\beta \, \left( b + f + e + \frac{t^{1-\mu}}{1-t^{1-\mu}}\right)^\gamma(1-t^{1-\mu})^\gamma \, t^{\frac{\gamma}{6(1-2\mu)}}=1, \quad \mbox{where} \quad \alpha, \beta, \gamma \in \Z$$
then $\alpha=\beta=\gamma=0$.
\end{lemma}
\begin{lemma}\label{li_bs}
The Seidel elements $b_{34}$, $b_{24}$ and $b_{23}$ are linearly independent in the subgroup of invertible elements of the quantum homology $\QH_4^\times(\Muhalf)$. 
\end{lemma}

Now putting together the two lemmas it is easy to conclude that the five Seidel elements $\mathcal S (z_{0,12}),$ $\mathcal S (z_{0,13}),$  $\mathcal S (z_{0,14}), \mathcal S (z_1)$ and $ \mathcal S (z_{1,4})$ are linearly independent. If 

\begin{gather*}
\mathcal S (z_{0,12})^{\alpha_1}\, \mathcal S (z_{0,13})^{\alpha_2}\, \mathcal S (z_{0,14})^{\alpha_3}\, \mathcal S (z_1)^\beta\, \mathcal S (z_{1,4})^\gamma = 1
\end{gather*}
then 
\begin{gather*}
b_{34}^{\alpha_1}\ b_{24}^{\alpha_2}\ b_{23}^{\alpha_3}\ (b_{12} + f_{34})^\beta\ \left( b_{12} + f_{34}+e_4 + \frac{t^{1-\mu}}{1-t^{1-\mu}}\right)^\gamma(1-t^{1-\mu})^\gamma\, t^{\frac{\gamma}{6(1-2\mu)}}=1. 
\end{gather*}
Using the simplification above of the quantum homology algebra, one obtains 
$$ b ^{\alpha_1+\alpha_2+\alpha_3}\, (b+f)^\beta \, \left( b + f + e + \frac{t^{1-\mu}}{1-t^{1-\mu}}\right)^\gamma(1-t^{1-\mu})^\gamma \, t^{\frac{\gamma}{6(1-2\mu)}}=1.$$ 
From Lemma \ref{li_simple_ring} we conclude that $\alpha_1+\alpha_2+\alpha_3=\beta=\gamma=0$. Therefore 
$$ b_{34}^{\alpha_1}\ b_{24}^{\alpha_2}\ b_{23}^{\alpha_3}=1,$$
and Lemma \ref{li_bs} implies that $\alpha_1=\alpha_2=\alpha_3=0$, so the five Seidel elements are linearly independent. 
\end{proof}

Next we study the action of $\pi_0(\Symp_h(M,\om))$ on $\pi_1(\Symp_0(M,\om))$.

\begin{proposition}
If $(M,\om)$ belongs to the edge $MA$ of the reduced symplectic cone, and if $\mu>3/2$, then the Seidel homomorphism is injective and the action of $\pi_0(\Symp_h(M,\om))$ on $\pi_1(\Symp_0(M,\om))$ is trivial.
\end{proposition}
\begin{proof}
The Seidel homomorphism factors through 
\[
\begin{tikzcd}
\pi_1(\Symp_0(M,\om))\arrow{r}{\Ss} \arrow[swap]{dr}{\pi} & \QH_{2n}(M,\om)^{\times}  \\
     & \pi_1(\Symp_0(M,\om))\slash\pi_0(\Symp_h(M,\om))\arrow{u}{}
\end{tikzcd}
\]
Consider arbitrary lifts in $\pi_1(\Symp_0(M,\om))\simeq\Z^5$ of the 5 actions. By Proposition~\ref{free-subgroup-of-rank-5}, these lifts generate a free subgroup of rank 5 in $\pi_1(\Symp_0(M,\om))\simeq\Z^5$ on which the Seidel homomorphism $\Ss$ is injective. Consequently, $\Ss$ is injective on the whole group $\pi_1(\Symp_0(M,\om))$. Since $\pi_1(\Symp_0(M,\om))\slash\pi_0(\Symp_h(M,\om))$ is also of rank $5$, the triviality of the action follows readily. 
\end{proof}

\begin{corollary}\label{graph=homotopyclass}
To each labelled extended graph in Figures~\ref{firstcase}, \ref{secondcase}, and~\ref{graphMA} corresponds a well-defined homotopy class in $\pi_1(\Symp_0(M,\om))$.
\end{corollary}

We now prove Theorem \ref{main}. 
\begin{proof}
Consider the extended labelled graph $z_{0,1i}$, $i=2,3,4$, $z_1$ and $z_{1,4}$ of Figure~\ref{graphMA}.  It follows from Corollary~\ref{graph=homotopyclass} that these graphs define five elements of $\pi_1 (\Symp_h(\Muhalf))$. By Proposition~\ref{free-subgroup-of-rank-5}, these five elements are linearly independent. This proves the second statement of the theorem. The first statement follows immediately from Remark \ref{RemarkNumberCircleActions}. 
\end{proof}

Note that the proof of Proposition \ref{free-subgroup-of-rank-5} does not depend on the value of $\mu$, besides the condition $\mu >1$. It follows that the 5 quantum homology classes $b_{34}$, $b_{24}$, $b_{23}$, $b_{12}+f_{34}$ and 
\begin{equation} \label{quantumrepresentative}
 \left( b_{12} + f_{34}+e_4 + \frac{t^{1-\mu}}{1-t^{1-\mu}}\right)(1-t^{1-\mu})\, t^{\frac{1}{6(1-2\mu)}}   
\end{equation}
generate a free subgroup of rank 5 in the group of invertible elements of the quantum homology, for every $\mu >1$. This observation proves the following result. 
\begin{proposition}\label{propquantumrepresentative}
 When $ 1 < \mu \leq \frac32$, the quantum class \eqref{quantumrepresentative} is not contained in the $\Q$-subspace spanned by circle actions inside the $\Q$-vector space formed by Seidel elements. 
\end{proposition}

\subsection{Relations between Hamiltonian circle actions on  \texorpdfstring{$\Mucccc$}{Mucccc}} 
In this section we prove our second main result. The main step is to obtain a classification of all circle actions in $\Symp_h(\Muhalf)$, which also shows that although we have more and more circle actions on $\Muhalf$ as we increase the value of $\mu$, they do not give new generators in the fundamental group of the symplectomorphism group $\Symp_h(\Muhalf)$. We show this by describing relations between the loops $z_k$, $z_{k,i}$, $z_{k,ij}$, $z_{k,ijl}$ and  $z_{k,1234}$, described in Section \ref{Extendedgraphs}, that come from embedding pairs of loops inside torus actions. The tools we use are Delzant's classification of toric actions and Karshon's classification of Hamiltonian circle actions. We always consider the actions in the generic case, that is, when $0 < c_4 < c_3 < c_2 < c_1 <  c_i+c_j < 1 < \mu$, with $i,j \in \{1,2,3,4\}$, in order to obtain the relations between the loops. As explained in Section~\ref{DeformationSeidel}, these relations between actions induce relations between Seidel elements in the quantum ring associated to the generic symplectic form. These relations map to similar relations in the quantum homology ring of $\Muhalf$. By injectivity of the Seidel homomorphism when $\mu>3/2$, we deduce that these relations hold in $\pi_1(\Symp(\Muhalf))$ as well.

Consider the manifold $\Mucccc$ endowed with a toric action, which we denote by $T_k$, such that the momentum polygon is given in Figure \ref{ToricTk}. Besides the homology classes indicated in the figure it should be clear that the classes $E_3-E_4$ and $E_4$ are also represented in the bottom edges. Projecting onto the $x$ and $y$-axis we obtain  the graphs in Figure \ref{GraphsTk} of the actions $z_k$ and $w_k$, respectively, whose momentum maps are the first and second coordinates of the momentum map of the action $T_k$. 

\begin{figure}[thp]

\begin{tikzpicture}[scale=0.7, roundnode/.style={circle, draw=black!80, thick, minimum size=7mm}, font=\footnotesize]
 
    \draw[->][black!70] (0,0) -- (5,0); 
    \draw[->][black!70] (0,-4.2) -- (0,10.2); 

 \draw[brown] (3.5,-3.625) -- (4, -3.6875);
  \draw[brown]  (3.5,-3.625) -- (3, -3.5);
       \draw[brown]  (2,-3) -- (3, -3.5);
	\draw[brown]  (1,-2) -- (2, -3);
	\draw[brown]  (0,0) -- (1,-2);
	\draw[brown]  (0,0) -- (0,4);
	\draw[brown]  (0,4) -- (4,10);
    \draw[brown]  (4,10) -- (4,-3.6875);
 
          \draw[dashed, black!60] (3.5,-3.625) --(3.5,0);   
            \draw[dashed, black!60] (3,-3.5) --(3,0);
            \draw[dashed, black!60] (2,-3) --(2,0);
            \draw[dashed, black!60] (1,-2) --(1,0);
        \draw[dashed, black!60] (0,10) --(4,10);
        \draw[dashed, black!60] (0,-2) --(1,-2);
        \draw[dashed, black!60] (0,-3) --(2,-3);
        \draw[dashed, black!60] (0,-3.5) --(3,-3.5);
        \draw[dashed, black!60] (0,-3.625) --(3.5,-3.625);
        \draw[dashed, black!60] (0,-3.6875) --(4,-3.685);	
           \draw[dashed, black!60] (0,10) --(4,10);
            \draw[dashed, black!60] (0,10) --(4,10);
              \draw[dashed, black!60] (0,10) --(4,10);
               \draw[dashed, black!60] (0,10) --(4,10);
	
    \node[blue] at (1,6) {$F$};
    \node[blue] at (0.6,-2.5) {$E_1-E_2$};
    \node[blue] at (-0.8,2) {$B-kF$};
     \node[blue] at (7,2) {$B+kF -E_1-E_2-E_3-E_4$};
     \node[blue] at (0,-1.35) {$F-E_1$};
     \node[blue] at (1.4,-3.25) {$E_2-E_3$};
  \node at (-0.3,10) {$\mu$};
 \node at (-0.6,4) {$\mu-k$};
 \node at (-0.3,0) {$0$};
  \node at (-1.1,-2) {$-k(1-c_1)$};
  \node at (-2,-3) {$-k(1-c_2) + c_1-c_2$};
    \node at (-2.6,-3.5) {$-k(1-c_3) + c_1+c_2-2c_3$};
        \node at (-0.3,-3.8) {$\vdots$};
     \node[rotate=45] at (1.3,0.55) {$1-c_1$};
      \node[rotate=45] at (2.3,0.55) {$1-c_2$}; 
        \node[rotate=45] at (3.3,0.55) {$1-c_3$}; 
          \node[rotate=45] at (3.8,0.55) {$1-c_4$}; 
         \node at (4.2,0.2) {$1$};             
\end{tikzpicture}

\caption{Toric action $T_k$}
 \label{ToricTk}
 \end{figure}

\begin{figure}[thp]
\begin{minipage}{.48\textwidth}
\begin{tikzpicture}[scale=1.1,font=\footnotesize]
\node at (2.5,1.7 ) {$B +kF  -E_1 - E_2 -E_3 - E_4 $};
       \fill[black] (2,1) ellipse (1 and 0.2);
	\node at (5.3,1) {$1,\mu +k -c_1-c_2-c_3-c_4 $};
	\filldraw [black] (2,0.5) circle (2pt);
        \draw[dashed, black!60] (2,1) --(2,0.5); 
	\node at (1.7,0.6) {$ E_4$};
	 \draw[dashed, black!60] (2,1) --(2,-2); 
	\node at (1.7,-1.3) {$ F-E_4$};
	\filldraw [black] (2.2,0.1) circle (2pt);
	 \draw[dashed, black!60] (2.2,0.8) --(2.2,0.1); 
	\node at (2.4,0.4) {$ E_3$};
	 \draw[dashed, black!60] (2.2,0.8) --(2.2,-2); 
	\node at (2.7,-0.4) {$ F-E_3$};
	\node at (4,0.5) {$ 1-c_4$};
	\node at (4,0.1) {$ 1-c_3$};
	\node at (4,-0.6) {$ 1-c_2$};
	\node at (4,-1.4) {$ 1-c_1$};
	\filldraw [black] (1.4,-0.6) circle (2pt);
	 \draw[dashed, black!60] (1.4,-0.6) --(1.4,1); 
	\node at (1.1,0) {$ E_2$};
	 \draw[dashed, black!60] (1.4,-0.6) --(1.4,-2); 
	\node at (0.8,-1) {$ F-E_2$};
	\filldraw [black] (2.7,-1.4) circle (2pt);
	 \draw[dashed, black!60] (2.7,-1.4) --(2.7,1); 
	\node at (3,-1) {$ E_1$};
	\draw[dashed, black!60] (2.7,-1.4) --(2.7,-2); 
	\node at (3.3,-1.7) {$ F-E_1$};
        \fill[black] (2,-2) ellipse (1 and 0.2);
	\node at (4.1,-2) {$0,\mu - k$};
	 \node at (2,-2.7) {$B-kF$};
       \node at (2,-3.4) {Circle action $z_k$};
\end{tikzpicture}
\end{minipage}%
\begin{minipage}{.48\textwidth}
\quad
\begin{tikzpicture}[scale=0.9,font=\footnotesize]
	\filldraw [black] (2,3.4) circle (2pt);
	\filldraw [black] (2,2.4) circle (2pt);
	\filldraw [black] (2,1) circle (2pt);
	\filldraw [black] (2,0) circle (2pt);
	\filldraw [black] (2,-0.7) circle (2pt);
	\filldraw [black] (2,-1.3) circle (2pt);
	\filldraw [black] (2, -1.8) circle (2pt);
	\filldraw [black] (2,-2.2) circle (2pt);
	
	\draw (2,3.4) -- (2,2.4);
	\draw (2,1) -- (2,-2.2);
	
 \node at (1.7,2.9) {$k$};
  \node at (1.6,0.5) {$-k$};
   \node at (1.4,-0.35) {$-k+1$};
  \node at (1.4,-1.05) {$-k+2$};
   \node at (1.4,-1.55) {$-k+3$};
    \node at (1.4,-2) {$-k+4$};
	 \node at (3.1,3.4) {$\mu$};
	 \node at (3.3,2.4) {$\mu -k$};
	 \node at (3.1,1) {$0$};
	 \node at (3.7,0) {$-k(1-c_1)$};
	 \node at (4.4,-0.7) {$-k(1-c_2) +c_1-c_2$};
	 \node at (4.8,-1.3) {$-k(1-c_3) +c_1+c_2-2c_3$};
     \node at (5.2,-1.8) {$-k(1-c_4) +c_1+c_2+c_3-3c_4$};
      \node at (4.6,-2.2) {$-k +c_1+c_2+c_3+c_4$};
       \node at (2,-2.7) {Circle action $w_{k}$};
\end{tikzpicture}

\end{minipage}%
\caption{Graphs of the circle actions  $z_k$ and $w_k$, respectively}
\label{GraphsTk}
\end{figure}

Performing  the $GL(2, \Z)$ transformation represented by the matrix 
$$ \left (
\begin{array}{cc}
1  & 0 \\
j  & -1 
\end{array} \right)$$ to the polygon of Figure \ref{ToricTk} yields a new polygon representing the same toric manifold. This new polygon has  vertices 
\begin{align*}
& (1, j- \mu), \ (0, k-\mu), \ (0,0), \ (1-c_1, (j+k)(1-c_1)),  \ (1-c_2, (j+k)(1-c_2)-c_1+c_2), \\
& (1-c_3, (j+k)(1-c_3)-c_1-c_2+2c_3),  \ (1-c_4, (j+k)(1-c_4)-c_1-c_2-c_3+3c_4), \ \mbox{and} \\
& (1, j+k-c_1-c_2-c_3-c_4).
\end{align*}
Then perform the $GL(2, \Z)$ transformation represented by the matrix 
$$ \left (
\begin{array}{cc}
1  & 0 \\
k & -1 
\end{array} \right)$$ to the polygon of the toric action $T_j$. It should be clear, looking at the coordinates of the two transformed polygons, that 
 projecting both polygons onto the $y$-axis we obtain the same graph, which implies that 
\begin{equation}\label{firstequation}
jz_k -w_k =kz_j-w_j \quad j,k \geq 1.
\end{equation}

Now consider the toric action $T_0$ on $\Mucccc$ represented in the Delzant polygon of Figure \ref{toricT0}. Consider also  its projections, in Figure \ref{GraphsT0}, to the $x$ and $y$-axis  representing circle actions that we denote by $(z_0, y_0)$. 

\begin{figure}[thp]
\begin{center}
\begin{tikzpicture}[scale=0.7, roundnode/.style={circle, draw=black!80, thick, minimum size=7mm}, font=\footnotesize]
 
    \draw[->][black!70] (-0.2,0) -- (4.6,0); 
    \draw[->][black!70] (0,-0.2) -- (0,5.7); 

    \draw[brown] (0,0) -- (2,0);
    \draw[brown]  (0,0) -- (0,5);
    \draw[brown]  (0,5) -- (4,5);
    \draw[brown]  (2,0) -- (3,1);
     \draw[brown]  (3,1) -- (3.5,2);
    \draw[brown]  (3.5,2) -- (3.75,3);
    \draw[brown]  (3.75,3) -- (4,4);
    \draw[brown]  (4,4) -- (4,5);
   
   \draw[dashed, black!60] (4,4) --(4,0);
	\draw[dashed, black!60] (0,1) -- (3,1);
	\draw[dashed, black!60] (0,4) -- (4,4);
        \draw[dashed, black!60] (0,3) -- (3.75,3);
        	\draw[dashed, black!60] (0,2) -- (3.5,2);	

  \node at (-0.3,5) {$\mu$};
    \node at (-0.8,1) {$c_1 -c_2$};
    \node at (-1.3,2) {$c_1+c_2 - 2c_3$};
     \node at (-1.8,3) {$ c_1+c_2 +c_3 -3c_4$};
    \node at (-1.7,4) {$c_1+c_2 +c_3 +c_4 $};
    \node at (4,-0.3) {$1 $};
   \node[blue] at (4.3, 3.5) {$ E_4 $};
   \node[blue]  at (4.6,2.5) {$ E_3-E_4 $};
   \node[blue]  at (4.2,1.5) {$ E_2-E_3  $};
   \node[blue]  at (3.6,0.5) {$ E_1-E_2  $};

\end{tikzpicture}
\end{center}
\caption{Toric action $T_0$}
 \label{toricT0}
 \end{figure}

\begin{figure}[thp]
\begin{minipage}{.48\textwidth}
\begin{tikzpicture}[scale=1.1,font=\footnotesize]
\node at (2,1.7 ) {$B -E_1 - E_2 -E_3 - E_4 $};
       \fill[black] (2,1) ellipse (1 and 0.2);
	\node at (5,1) {$1,\mu -c_1-c_2-c_3-c_4 $};
	\filldraw [black] (2,0.5) circle (2pt);
        \draw[dashed, black!60] (2,1) --(2,0.5); 
	\node at (1.7,0.6) {$ E_4$};
	 \draw[dashed, black!60] (2,1) --(2,-2); 
	\node at (1.7,-1.3) {$ F-E_4$};
	\filldraw [black] (2.2,0.1) circle (2pt);
	 \draw[dashed, black!60] (2.2,0.8) --(2.2,0.1); 
	\node at (2.4,0.4) {$ E_3$};
	 \draw[dashed, black!60] (2.2,0.8) --(2.2,-2); 
	\node at (2.7,-0.4) {$ F-E_3$};
	\node at (4,0.5) {$ 1-c_4$};
	\node at (4,0.1) {$ 1-c_3$};
	\node at (4,-0.6) {$ 1-c_2$};
	\node at (4,-1.4) {$ 1-c_1$};
	\filldraw [black] (1.4,-0.6) circle (2pt);
	 \draw[dashed, black!60] (1.4,-0.6) --(1.4,1); 
	\node at (1.1,0) {$ E_2$};
	 \draw[dashed, black!60] (1.4,-0.6) --(1.4,-2); 
	\node at (0.8,-1) {$ F-E_2$};
	\filldraw [black] (2.7,-1.4) circle (2pt);
	 \draw[dashed, black!60] (2.7,-1.4) --(2.7,1); 
	\node at (2.9,-1) {$ E_1$};
	\draw[dashed, black!60] (2.7,-1.4) --(2.7,-2); 
	\node at (3.2,-1.7) {$ F-E_1$};
        \fill[black] (2,-2) ellipse (1 and 0.2);
	\node at (3.8,-2) {$0,\mu $};
	 \node at (2,-2.7) {$B$};
       \node at (2,-3.4) {Circle action $z_0$};
\end{tikzpicture}
\end{minipage}%
\begin{minipage}{.48\textwidth}
\quad \quad 
\begin{tikzpicture}[scale=1.1,font=\footnotesize]
\node at (2,2.4 ) {$F $};
       \fill[black] (2,1.8) ellipse (1 and 0.2);
	\node at (3.7,1.8) {$\displaystyle{1,\mu}$};
	\filldraw [black] (2,0.4) circle (2pt);
	\filldraw [black] (2,1) circle (2pt);
	 \draw[dashed, black!60] (1.4,1.8) --(1.4,-1.5); 
	  \node at (1.2,0.3) {$ B$};
	 \draw[dashed, black!60] (2,1) --(2,1.8); 
	  \node at (3.5,1.4) {$ B-E_1 -E_2 -E_3 -E_4$};
	\node at (4.7 ,1) {$c_1+c_2 +c_3 +c_4 $};
        \node at (4.8 ,0.4) {$c_1+c_2 +c_3 -3c_4$};
	\filldraw [black] (2,-0.2) circle (2pt);
	\filldraw [black] (2,-0.8) circle (2pt);
	\node at (4.5 ,-0.2) {$c_1+c_2 - 2c_3$};
        \node at (4.1 ,-0.8) {$c_1 -c_2$};
         \draw[dashed, black!60] (2,-1.5) --(2,-0.8); 
         \node at (2.7,-1.1) {$ E_1-E_2$};
          \node at (2.7,-0.5) {$ E_2-E_3$};
            \node at (2.7,0.1) {$ E_3-E_4$};
              \node at (2.3,0.7) {$ E_4$};
        \fill[black] (2,-1.5) ellipse (1 and 0.2);
	\node at (4.2,-1.5) {$\displaystyle{0,1-c_1}$};
	 \node at (2,-2) {$F-E_1$};
	  \draw (2,-0.8) -- (2,1);
	   \node at (1.7,0.7) {$4$};
	    \node at (1.7,0.1) {$3$};
	     \node at (1.7,-0.5) {$2$};
      \node at (2,-2.7) {Circle action $y_0$};
\end{tikzpicture}
\end{minipage}%

\caption{Graphs of circle actions $z_0$ and $y_0$, respectively}
 \label{GraphsT0}
 \end{figure}

Again, performing  the $GL(2, \Z)$ transformation represented by the matrix 
$$ \left (
\begin{array}{cc}
1  & 0 \\
-k & 1 
\end{array} \right)$$ to the polygon of Figure \ref{toricT0} yields a new polygon representing the same toric manifold. This new polygon has  vertices 
\begin{align*}
& (0, \mu),  (1, \mu-k),   (0,0),  (1-c_1, -k(1-c_1)),   (1-c_2, -k(1-c_2)+c_1-c_2), \\
& (1-c_3, -k(1-c_3)+c_1+c_2-2c_3),   (1-c_4, -k(1-c_4)+c_1+c_2+c_3-3c_4), \ \mbox{and} \\
& (1, -k+c_1+c_2+c_3+c_4).
\end{align*}
Therefore, projecting this new polygon onto the $y$-axis, it is easy to check that we obtain the graph of the circle action $w_k$, which means we have the following identification 
\begin{equation}\label{auxiliary1}
w_k = -kz_0+y_0, \quad k \geq 1. 
\end{equation}
Combining equations \eqref{firstequation} and \eqref{auxiliary1} yields 
$$ jz_k +kz_0= kz_j +jz_0.$$
Finally, setting $j=1$ implies that 
\begin{equation}\label{result1}
z_k=kz_1 +(1-k)z_0, \quad k \geq 0.
\end{equation}
Next, we use a similar argument in order to obtain more relations between other circle actions listed in Proposition \ref{actionslist}. For that we need to consider the toric action $T_{k,4}$ on $\Mucccc$, represented in the  momentum potytope is in Figure \ref{toricTk4}.  Note that are edges, whose homology classes are not indicated in the figure, namely  $E_2-E_3$ and $E_3$. 

\begin{figure}[thp]

\begin{tikzpicture}[scale=0.7, roundnode/.style={circle, draw=black!80, thick, minimum size=7mm}, font=\footnotesize]
 
    \draw[->][black!70] (0,0) -- (5,0); 
    \draw[->][black!70] (0,-5.9) -- (0,10.2); 
    \draw[brown] (0,0.5) -- (0.5, -1);
          \draw[brown] (3.5,-5.25) -- (4, -5.375);
       \draw[brown] (3,-5) -- (3.5, -5.25);
	\draw[brown] (2,-4) -- (3, -5);
	\draw[brown] (0.5,-1) -- (2,-4);
	\draw[brown] (0,0.5) -- (0,4);
	\draw[brown] (0,4) -- (4,10);
    \draw[brown] (4,10) -- (4,-5.375);
 
          \draw[dashed, black!60] (0,-1) --(0.5,-1);   
            \draw[dashed, black!60] (3,-5) --(3,0);
            \draw[dashed, black!60] (2,-4) --(2,0);
         \draw[dashed, black!60] (0.5,-1) --(0.5,0);
        \draw[dashed, black!60] (0,10) --(4,10);
        \draw[dashed, black!60] (0,-5.25) --(3.5,-5.25);
        \draw[dashed, black!60] (0,-4) --(2,-4);
        \draw[dashed, black!60] (0,-5) --(3,-5);
        \draw[dashed, black!60] (0,-5.375) --(4,-5.375);	
           \draw[dashed, black!60] (0,10) --(4,10);
  \draw[dashed, black!60] (3.5,0) --(3.5,-5.25);	
    \node[blue]  at (1,6) {$F$};
    \node[blue]  at (0,-2.5) {$F-E_1-E_4$};
    \node[blue]  at (-1.3,2) {$B-kF-E_4$};
    \node[blue]  at (6.4,2) {$B+kF -E_1-E_2-E_3$};
    \node[blue]  at (1.5,-4.5) {$E_1-E_2$};
    \node[blue]  at (0,-0.25) {$E_4$};

  \node at (-0.3,10) {$\mu$};
 \node at (-0.6,4) {$\mu-k$};
 \node at (-0.3,0.5) {$c_4$};
  \node at (-0.6,-1) {$-kc_4$};
  \node at (-1.1,-4) {$-k(1-c_1)$};
  \node at (-2,-4.9) {$-k(1-c_2) + c_1-c_2$};
    \node at (-0.3,-5.5) {$\vdots$};
     \node[rotate=45] at (2.3,0.55) {$1-c_1$};
      \node[rotate=45] at (3.3,0.55) {$1-c_2$}; 
        \node[rotate=45] at (3.8,0.55) {$1-c_3$}; 
     \node[rotate=45] at (0.5,0.25) {$c_4$}; 

\end{tikzpicture}

\caption{Toric action $T_{k,4}$}
 \label{toricTk4}
 \end{figure}

Projecting onto the $x$ and $y$-axis we obtain  the graphs in Figure \ref{GraphsTk4} of the actions $(z_{k,4},w_{k,4})$, respectively, whose momentum maps are the first and second coordinates of the momentum map of the action $T_{k,4}$. 

\begin{figure}[thp]
\begin{minipage}{.45\textwidth}
\begin{tikzpicture}[scale=1.1,font=\footnotesize]
\node at (2, 1.7 ) {$B +kF  -E_1 - E_1 -E_3 $};
       \fill[black] (2,1) ellipse (1 and 0.2);
	\node at (4.8,1) {$\displaystyle{1,\mu +k  -c_1-c_2 -c_3}$};
	\filldraw [black] (1.2,0.3) circle (2pt);
	\draw[dashed, black!60] (1.2,1) --(1.2,0.3); 
	\node at (0.9,0.6) {$ E_3$};
	\draw[dashed, black!60] (1.2,1) --(1.2,-2); 
	\node at (0.6,-1.5) {$ F-E_3$};	
	\filldraw [black] (2,-0.3) circle (2pt);
	\draw[dashed, black!60] (2,1) --(2,-0.3); 
	\node at (2.2,0.1) {$ E_2$};
	\draw[dashed, black!60] (2,1) --(2,-2); 
	\node at (2.3,-1.3) {$ F-E_2$};
	\filldraw [black] (2.6, -0.8) circle (2pt);
	\draw[dashed, black!60] (2.6,1) --(2.6,-0.8); 
	\node at (2.8,-0.4) {$ E_1$};
	\draw[dashed, black!60] (2.6,1) --(2.6,-2); 
	\node at (3.1,-1.7) {$ F-E_1$};
	\filldraw [black] (1.7,-1.2) circle (2pt);
	\draw[dashed, black!60] (1.7,-2) --(1.7,1); 
	\node at (1.7,-0.7) {$ F-E_4$};
	\draw[dashed, black!60] (1.7,-2) --(1.7,-1.2); 
	\node at (1.5,-1.5) {$ E_4$};
	\node at (3.9,0.3) {$ 1-c_3$};
	\node at (3.9,-0.3) {$ 1-c_2$};
	\node at (3.9,-0.8) {$ 1-c_1$};
	\node at (3.7,-1.2) {$ c_4$};
        \fill[black] (2,-2) ellipse (1 and 0.2);
	\node at (4.3,-2) {$\displaystyle{0,\mu -k - c_4}$};
	 \node at (2,-2.7) {$B-kF-E_4$};
       \node at (2,-3.4) {Circle action $z_{k,4}$};
\end{tikzpicture}
\end{minipage}%
\begin{minipage}{.45\textwidth}
\quad
\begin{tikzpicture}[scale=0.9,font=\footnotesize]
	\filldraw [black] (2,3.4) circle (2pt);
	\filldraw [black] (2,2.4) circle (2pt);
	\filldraw [black] (2,1) circle (2pt);
	\filldraw [black] (2,0.4) circle (2pt);
	\filldraw [black] (2,-0.7) circle (2pt);
	\filldraw [black] (2,-1.3) circle (2pt);
	\filldraw [black] (2, -1.8) circle (2pt);
	\filldraw [black] (2,-2.2) circle (2pt);
	
	\draw (2,3.4) -- (2,2.4);
	\draw (2,1) -- (2,-2.2);
	
 \node at (1.7,2.9) {$k$};
  \node at (1.4,0.7) {$-k-1$};
   \node at (1.6,-0.15) {$-k$};
  \node at (1.4,-1.05) {$-k+1$};
   \node at (1.4,-1.55) {$-k+2$};
    \node at (1.4,-2) {$-k+3$};
	 \node at (3.1,3.4) {$\mu$};
	 \node at (3.3,2.4) {$\mu -k$};
	 \node at (3.1,1) {$c_4$};
	 \node at (3.3,0.4) {$-kc_4$};
	 \node at (3.7,-0.7) {$-k(1-c_1)$};
	 \node at (4.4,-1.3) {$-k(1-c_2) +c_1-c_2$};
     \node at (4.8,-1.8) {$-k(1-c_3) +c_1+c_2-2c_3$};
      \node at (4.2,-2.2) {$-k +c_1+c_2+c_3$};
       \node at (2,-2.7) {Circle action $w_{k,4}$};
\end{tikzpicture}

\end{minipage}%
\caption{Graphs of the circle actions  $z_{k,4}$ and $w_{k,4}$}
\label{GraphsTk4}
\end{figure}

Using the same argument as before, that is, performing  the $GL(2, \Z)$ transformation represented by the matrix 
$$ \left (
\begin{array}{cc}
1  & 0 \\
j  & -1 
\end{array} \right)$$ to the polygon of Figure \ref{toricTk4} yields a new polygon with  vertices 
\begin{align*}
& (1, j+k-c_1-c_2-c_3),  (1, j- \mu),  (0, k-\mu),  (0,-c_4), \ (c_4, (j+k)c_4),  (1-c_1, (j+k)(1-c_1)),  \\ 
& (1-c_2, (j+k)(1-c_2)-c_1+c_2),  \ \mbox{and} \ (1-c_3, (j+k)(1-c_3)-c_1-c_2+2c_3).
\end{align*}
Interchanging the role of $k$ and $j$ and then projecting onto the $y$-axis it follows that
\begin{equation}\label{secondequation}
jz_{k,4} -w_{k,4} =kz_{j,4}-w_{j,4} \quad j,k \geq 1.
\end{equation}

Now consider the toric action on $\Mucccc$ represented in the polygon of Figure \ref{toricT04}. Consider also  its projections, in Figure \ref{GraphsT04}, to the $x$ and $y$-axis  representing circle actions that we denote by $(z_{0,4}, y_{0,4})$. 

\begin{figure}[thp]
\begin{center}
\begin{tikzpicture}[scale=0.7, roundnode/.style={circle, draw=black!80, thick, minimum size=7mm}, font=\footnotesize]
 
    \draw[->][black!70] (-0.2,0) -- (4.6,0); 
    \draw[->][black!70] (0,-0.2) -- (0,5.7); 

 \draw[brown] (0,0.5) -- (0.5,0);
    \draw[brown] (0.5,0) -- (2,0);
    \draw[brown] (0,0.5) -- (0,5);
    \draw[brown] (0,5) -- (4,5);
    \draw[brown] (2,0) -- (3,1);
     \draw[brown] (3,1) -- (3.5,2);
    \draw[brown] (3.5,2) -- (4,3.5);
    \draw[brown] (4,3.5) -- (4,5);
   
   \draw[dashed, black!60] (4,4) --(4,0);
    \draw[dashed, black!60] (0,1) -- (3,1);
    \draw[dashed, black!60] (0,3.5) -- (4,3.5);
   \draw[dashed, black!60] (0,2) -- (3.5,2);	

     \node at (-0.3,5) {$\mu$};
    \node at (-0.8,1) {$c_1 -c_2$};
    \node at (-1.3,2) {$c_1+c_2 - 2c_3$};
     \node at (-0.3,0.5) {$ c_4$};
    \node at (-1.3,3.5) {$c_1+c_2 +c_3 $};
    \node at (4,-0.3) {$1 $};
   \node[blue] at (0.6, 0.4) {$ E_4 $};T
   \node[blue] at (4.1,2.5) {$ E_3 $};
   \node[blue] at (4.2,1.5) {$ E_2-E_3  $};
    \node[blue] at (3.6,0.5) {$ E_1-E_2  $};

\end{tikzpicture}
\end{center}
\caption{Toric action $T_{0,4}$}
 \label{toricT04}
 \end{figure}

\begin{figure}[thp]
\begin{minipage}{.45\textwidth}
\begin{tikzpicture}[scale=1.1,font=\footnotesize]
\node at (2, 1.7 ) {$B  -E_1 - E_2 -E_3 $};
       \fill[black] (2,1) ellipse (1 and 0.2);
	\node at (4.6,1) {$\displaystyle{1,\mu  -c_1-c_2 -c_3}$};
	\filldraw [black] (1.2,0.3) circle (2pt);
	\draw[dashed, black!60] (1.2,1) --(1.2,0.3); 
	\node at (0.9,0.6) {$ E_3$};
	\draw[dashed, black!60] (1.2,1) --(1.2,-2); 
	\node at (0.6,-1.5) {$ F-E_3$};	
	\filldraw [black] (2,-0.3) circle (2pt);
	\draw[dashed, black!60] (2,1) --(2,-0.3); 
	\node at (2.2,0.1) {$ E_2$};
	\draw[dashed, black!60] (2,1) --(2,-2); 
	\node at (2.3,-1.3) {$ F-E_2$};
	\filldraw [black] (2.6, -0.8) circle (2pt);
	\draw[dashed, black!60] (2.6,1) --(2.6,-0.8); 
	\node at (2.8,-0.4) {$ E_1$};
	\draw[dashed, black!60] (2.6,1) --(2.6,-2); 
	\node at (3.1,-1.7) {$ F-E_1$};
	\filldraw [black] (1.7,-1.2) circle (2pt);
	\draw[dashed, black!60] (1.7,-2) --(1.7,1); 
	\node at (1.7,-0.7) {$ F-E_4$};
	\draw[dashed, black!60] (1.7,-2) --(1.7,-1.2); 
	\node at (1.5,-1.5) {$ E_4$};
	\node at (3.9,0.3) {$ 1-c_3$};
	\node at (3.9,-0.3) {$ 1-c_2$};
	\node at (3.9,-0.8) {$ 1-c_1$};
	\node at (3.7,-1.2) {$ c_4$};
        \fill[black] (2,-2) ellipse (1 and 0.2);
	\node at (4.1,-2) {$\displaystyle{0,\mu - c_4}$};
	 \node at (2,-2.7) {$B-E_4$};
       \node at (2,-3.4) {Circle action $z_{0,4}$};
\end{tikzpicture}
\end{minipage}%
\begin{minipage}{.45\textwidth}
\quad \quad 
\begin{tikzpicture}[scale=1.1,font=\footnotesize]
\node at (2,2.4 ) {$F $};
       \fill[black] (2,1.8) ellipse (1 and 0.2);
	\node at (3.8,1.8) {$\displaystyle{1,\mu}$};
	\filldraw [black] (2,0.4) circle (2pt);
	\filldraw [black] (2,1) circle (2pt);
	\draw[dashed, black!60] (2,1) --(2,2); 
	\node at (3.2,1.4) {$ B-E_1-E_2-E_3$};
	\node at (2.7,0.1) {$ E_2-E_3$};
	\node at (4.4 ,1) {$c_1+c_2 +c_3 $};
	\node at (2.3,0.7) {$ E_3$};
        \node at (4.5 ,0.4) {$c_1+c_2 - 2c_3$};
	\filldraw [black] (2,-0.2) circle (2pt);
	\draw[dashed, black!60] (2,-0.2) --(2,-1.5); 
	\node at (2.7,-0.6) {$ E_1-E_2$};
	\filldraw [black] (1.5,-0.8) circle (2pt);
	\draw[dashed, black!60] (1.5,-0.8) --(1.5,-1.5); 
	\node at (1.2,-1.1) {$ E_4$};
	\draw[dashed, black!60] (1.5,-0.8) --(1.5,2); 
	\node at (1,1.2) {$ F- E_4$};
	\node at (4.1 ,-0.2) {$c_1 -c_2$};
        \node at (3.8 ,-0.8) {$c_4$};
        \fill[black] (2,-1.5) ellipse (1 and 0.2);
	\node at (4.5,-1.5) {$\displaystyle{0,1-c_1-c_4}$};
	 \node at (2,-2) {$F-E_1$};
	  \draw (2,-0.2) -- (2,1);
	   \node at (1.7,0.7) {$3$};
	    \node at (1.7,0.1) {$2$};

       \node at (2,-2.7) {Circle action $y_{0,4}$};
\end{tikzpicture}
\end{minipage}%

\caption{Graphs of circle actions $z_{0,4}$ and $y_{0,4}$, respectively. }
 \label{GraphsT04}
 \end{figure}

In this case, performing  the $GL(2, \Z)$ transformation represented by the matrix 
$$ \left (
\begin{array}{cc}
1  & 0 \\
-k & 1 
\end{array} \right)$$ to the polygon of Figure \ref{toricT04} yields a new polygon whose vertices are
\begin{align*}
& (0, \mu),  (1, \mu-k),   (0,c_4), (c_4, -kc_4), (1-c_1, -k(1-c_1)),   (1-c_2, -k(1-c_2)+c_1-c_2), \\
& (1-c_3, -k(1-c_3)+c_1+c_2-2c_3), \ \mbox{and} \  (1, -k+c_1+c_2+c_3).
\end{align*}
Therefore, projecting this new polygon onto the $y$-axis, it is clear one obtains the graph of the circle action $w_{k,4}$, which means we have the following relation 
\begin{equation}\label{auxiliary2}
w_k = -kz_{0,4}+y_{0,4}, \quad k \geq 1. 
\end{equation}
Combining equations \eqref{secondequation} and \eqref{auxiliary2} yields
$$ jz_{k,4} +kz_{0,4}= kz_{j,4} +jz_{0,4},$$
and, setting $j=1$ implies that 
\begin{equation*}
z_{k,4}=kz_{1,4} +(1-k)z_{0,4}, \quad k \geq 0.
\end{equation*}
Therefore, it is clear that in fact we have the following identifications
\begin{equation}\label{result2}
z_{k,i}=kz_{1,i} +(1-k)z_{0,i}, \quad k \geq 0,\quad  i=1,2,3,4. 
\end{equation}

Using a similar argument three more times applied to the appropriate toric actions on the manifold $\Mucccc$ it follows that the following identifications hold:
\begin{equation}\label{result3}
z_{k,ij}=kz_{1,ij} +(1-k)z_{0,ij}, \quad k \geq 0,\quad  i,j=1,2,3,4, 
\end{equation}
\begin{equation}\label{result4}
z_{k,ij\ell}=kz_{1,ij\ell} +(k-1)z_{0,m}, \quad  k \geq 0,\quad  i,j,\ell, m=1,2,3,4 \quad \mbox{and all indices are distinct};
\end{equation}
and
\begin{equation}\label{result5}
z_{k,1234}=kz_{1,1234} +(k-1)z_{0}, \quad  k \geq 0.
\end{equation}

In what follows we will often use the next proposition. The proof follows from techniques similar to the ones used to obtain the previous relations so we postpone it to 
the appendix. 
\begin{proposition}\label{basicrelations}
Consider the circle actions  $z_0$, $z_{0,i} \ i=1,2,3,4$,  $z_{0,1i}, i=2,3,4$, $z_1$ and $z_{1,4}$, defined in Figures \ref{graphMA} or \ref{graphsgeneric}  through their graphs, as elements of 
$\pi_1 (\Symp_h(\Mucccc))$. Then the following identifications hold
\begin{align*}
z_{0,1}&= z_1 -z_{1,4} + z_{0,14}&  z_{0,2}&= z_1 -z_{1,4} -z_{0,13} \\
z_{0,3}&= z_1 -z_{1,4} - z_{0,12}&  z_{0,4}&= z_1 -z_{1,4} - z_{0,12}-z_{0,13}+ z_{0,14}
\end{align*}
and  $z_{0}= 2z_1 -2z_{1,4} - z_{0,12}-z_{0,13}+ z_{0,14}$. 
\end{proposition}
Putting together Proposition \ref{basicrelations} with equations \ref{result1}, \ref{result2}, \ref{result3}, \ref{result4} and \ref{result5} we then obtain the following result. 
\begin{proposition}\label{mainrelations}
  Consider the circle actions    $z_{0,1i}, i=2,3,4$, $z_1$ and $z_{1,4}$, defined in Figures \ref{graphMA} or \ref{graphsgeneric}  through their graphs, as elements of  $\pi_1 (\Symp_h(\Mucccc))$. Let $t= z_{0,12}+z_{0,13}- z_{0,14}$. Then for $k \in \Z_{\geq 0}$ and $i,j=1,2,3,4$ with $i \neq j$ we have the following identifications
 \begin{align*}
 & z_k= (2-k) z_1 + (k-1)(2z_{1,4} +t) &  &z_{k,ij}=2kz_{1,4}-kz_1 +kt + z_{0,ij} \\
 & z_{k,1}= (2k-1)z_{1,4} + (1-k)z_1 +kt +z_{0,14} &  &z_{k,124}= (2k+1)z_{1,4}-(k+1)z_1 +kt + z_{0,12} \\
 & z_{k,2}= (2k-1)z_{1,4} + (1-k)z_1 +kt -z_{0,13} & &z_{k,134}= (2k+1)z_{1,4}-(k+1)z_1 +kt + z_{0,13} \\
 & z_{k,3}= (2k-1)z_{1,4} + (1-k)z_1 +kt -z_{0,12} & &z_{k,234}= (2k+1)z_{1,4}-(k+1)z_1 +kt - z_{0,14} \\
  & z_{k,4}= (2k-1)z_{1,4} + (1-k)z_1 +(k-1)t  & &z_{k,123}= (2k+1)z_{1,4}-(k+1)z_1 +(k+1)t 
 \end{align*}
 and $z_{k,1234}=(2k+2)z_{1,4}-(k+2)z_1 +(k+1)t $.
\end{proposition}

\begin{remark}
It is clear from the definition of the circle actions $z_{0,ij}$ that 
$$z_{0,ij}=-z_{0,kl}.$$ 
Thus, we get 
$$z_{23}= -z_{0,14}, \quad z_{24}= -z_{0,13}, \quad  z_{34}= -z_{0,12}$$
which allows us to write $z_{k,ij}$ completely in terms of the chosen generators. 
\end{remark}

\begin{corollary}\label{notcircleactions}
If $\mu > \frac32$ then there are loops in $\pi_1(\Symp(\Muhalf))$  which are not represented by circle actions, although the fundamental group, rationally,  is generated by circle actions. 
\end{corollary}
\begin{proof}
We use the classification of all circle actions on $\Symp(\Muhalf)$ obtained in Proposition \ref{mainrelations} to show that those actions do not fill in the lattice $\Z^5$. 
In fact it is sufficient to look at the plane of the actions $z_{1,4}$ and $z_1$. There we have 5 families of points defined by the pairs $(2k-2,2-k)$, $(2k-1,1-k)$, $(2k,-k)$, $(2k+1,-(k+1))$ and $(2k+2,-(k+2))$. Each family is contained in one line with slope $-\frac12$ while the $y$-intercepts are given by $1,\frac12, 0, -\frac12$ and $-1$, respectively. We should also consider the integer multiples of these actions as they represent circle actions although they are not effective. The corresponding points lie in lines connecting the origin to the points representing the effective actions. The set of all these points clearly do not fill in the full lattice $\Z^2$ in this plane. For example, the primitive point $(2,3)$ is not contained in this set. Therefore there are elements in $\pi_1$ which cannot be represented by circle actions. 
\end{proof}

We now prove Theorem \ref{main2}.
\begin{proof}
In the case $1<\mu \leq  \frac32$ it follows from Theorem \ref{main} and in the case $\mu > \frac32$ it follows from Corollary \ref{notcircleactions}. 
\end{proof}

\section{Further Questions}
In this paper we deal with a particular case in the symplectic cone of $\Mucccc$, namely the edge $MA$, when $\mu >1$ and $c_i=1/2$ for all $i \in \{1,2,3,4 \}$. A very natural question is whether there are  other points in the symplectic cone   where the fundamental group of $\Symp_h(\Mucccc)$ is not generated Hamiltonian circle actions, similarly to what happens  along some points of the edge $MA$. For example, it is possible to check that along the edge $MD$, where 
$\mu=1, c_1=c_2=c_3=1/2 >c_4$ there are no circle actions. On one hand there are no graphs representing Hamiltonian $S^1$-spaces along this edge. However, the graph only encodes equivariant blow-ups. Not having such graphs does not a-priori rule out the possibility of \emph{exotic} circle actions, obtained by equivariant blow-ups corresponding to parameters $\mu',c_1',c_2',c_3',c_4'$ such that the symplectic manifold corresponding to these parameters is symplectomorphic to the symplectic manifold corresponding to the parameters $\mu,c_1,c_2,c_3,c_4$. On the other hand, \emph{exotic} circle actions were ruled out, first, by Pinsonnault in \cite{Pin2} and later by Karshon-Kessler-Pinsonnault in \cite{KKP}. 
Note that if we forget the $4^{th}$ blow-up, this case corresponds to the monotone case in $\Symp_h({S^2 \times S^2\#\,3\overline{ \bbcp}\,\!^2})$ and it is well-known that there are no Hamiltonian circle actions in this case. In fact this symplectomorphism group is contractible (see \cite{Eva}). However, by the work of \cite{LiLiWu2} we know that the rank of the fundamental group of $\Symp(\Mucccc)$ along the edge $MD$ is 5, so none of these 5 generators can be represented by  circle actions. 
Moreover, we believe that there is a neighbourhood of the monotone point $M$, including points in the generic case, such that the generators of the fundamental group of $\Symp(\Mucccc)$ cannot all be realized by circle actions. The main reason appears to be that one circle action of the type $z_{0,i}$ or $z_{1,i}$ for some $i \in  \{1,2,3,4 \}$ always have to be included in the set of generators, in order to have the required number of generators, but this implies that there must exist a fixed sphere with positive area in class $B-E_j-E_k- E_\ell$ or $B-F-E_i$, that is, $ \mu -c_j-c_k-c_\ell > 0$ or $\mu -1 -c_i >0$, respectively. However this condition does not necessarily hold for all points in the symplectic cone, in particular, for points close to the monotone point $M$. 

An alternative way of proving that there is an element of $ \pi_1(\Symp_h(\Mucccc))$ which cannot be represented by a circle action would be to compute the Samelson product of this loop with itself and check if it is non-zero, as done by O. Buse in \cite[Proposition 3.3]{Bus} (if it was generated by a circle action this Samelson product would be trivially 0). Moreover, it would be interesting to know if this loop in $\pi_1$ gives rise to new elements in higher homotopy groups, via iterated Samelson products. The problem with this approach is that it is not clear yet how to obtain the necessary information about the higher homotopy groups of $\Symp_h(\Mucccc)$ which is fundamental to work on these  ideas. The answer to these questions will be pursued in a different paper. 

On another direction, it seems very likely that Theorem 1.4 holds not only rationally but also in the integer case, that is, 

\begin{conjecture}
If $\mu > \frac32$ then the fundamental group $\pi_1(\Symp(\Muhalf))$ is generated by Hamiltonian circle actions. 
\end{conjecture}

In fact it is possible to show that the Seidel elements of the classes of the actions $z_{0,12}$, $z_{0,13}$, $z_{0,14}$, $z_1$ and $z_{1,4}$, seen as elements of $\pi_1(\Symp(\Muhalf))$, are primitive in the subgroup of invertible elements of the quantum homology of $\Symp(\Muhalf)$. This result  together with a geometric interpretation of the generators of $\pi_1(\Symp(\Muhalf))$ given in~\cite[Lemma 5.10]{LiLiWu2} might lead to a proof of this conjecture. We believe this should involve a detailed analysis of the strata of the space of almost complex structures and the generators of their homology groups. 

%
%

\begin{appendix}
\section{Computations on the quantum ring}\label{QHcomp}
In these section we prove Lemmas \ref{li_simple_ring} and \ref{li_bs}. 
\begin{proof}(of Lemma \ref{li_simple_ring})
Recall that we wish to prove that the invertible elements $b$, $b+f$ and $\left( b + f + e + \frac{t^{1-\mu}}{1-t^{1-\mu}}\right)(1-t^{1-\mu})\, t^{\frac{1}{6(1-2\mu)}}$ are linearly independent in the subgroup of invertible elements of the ring \eqref{simplified_ring}. Suppose 
$$ b^\alpha \,  (b+f)^\beta \, \left( b + f + e + \frac{t^{1-\mu}}{1-t^{1-\mu}}\right)^\gamma(1-t^{1-\mu})^\gamma \, t^{\frac{\gamma}{6(1-2\mu)}}=1, \quad \mbox{where} \quad \alpha, \beta, \gamma \in \Z.$$
Using the relations in the ideal $I'$ in \eqref{simplified_ring}, in particular relations (3) and (5), it follows that this is equivalent 
to 
$$ b^\alpha \,  b^\beta (1-f)^\beta \, b^\gamma \left(e + \frac{1}{1-t^{1-\mu}}\right)^\gamma(1-t^{1-\mu})^\gamma \, t^{\frac{\gamma}{6(1-2\mu)}}=1 \iff $$
$$ b^{\alpha+\beta +\gamma} \, (1-f)^\beta \,(e(1-t^{1-\mu}) + 1)^\gamma = t^{\frac{-\gamma}{6(1-2\mu)}} \iff $$
$$ b^{\alpha+\beta +\gamma} \, (1-\beta f) \,(e(1-t^{1-\mu}) + 1)^\gamma = t^{\frac{-\gamma}{6(1-2\mu)}}, $$
because $f^2=0$. Moreover, since $(1-\beta f)(1+\beta f)=1 $, if $\alpha+\beta +\gamma$ is even, then we obtain
\begin{equation}\label{powers_e}
(e(1-t^{1-\mu}) + 1)^\gamma = (1+ \beta f)\ t^{\frac{-\gamma}{6(1-2\mu)}}. 
\end{equation}
If $\gamma \geq 0$ then it should be clear from  relations in $I'$, in particular,  from relation (6) that we  cannot never obtain the right hand side of the last expression, since it is not possible to obtain such power on $t$, unless $\gamma=0$, which implies that $\beta =0$ and then $\alpha=0$. If $\gamma \leq 0$ then equation \eqref{powers_e} is equivalent to 
$$ (e(1-t^{1-\mu}) + 1)^{-\gamma} = (1- \beta f)\ t^{\frac{\gamma}{6(1-2\mu)}}$$ 
and again we we can conclude that $\gamma=0$. 
Similarly, 
if $\alpha+\beta +\gamma$ is odd, we obtain 
$$(e(1-t^{1-\mu}) + 1)^\gamma = (b- \beta f)\ t^{\frac{-\gamma}{6(1-2\mu)}} $$
and again we can conclude that $\gamma=\beta=0$ and $\alpha+\beta +\gamma$
cannot be odd. This concludes the proof of the lemma. 
\end{proof}

\begin{proof}(of Lemma \ref{li_bs})
We first prove by induction on $n \in \N$ that 
\begin{align}
& b_{ij}^{2n}= n(2n \, b_{ij}f_{ij}+ (2n-1)f_{ij}+ f_{k\ell}) +1; \label{b2n}\\
& b_{ij}^{2n+1}= -n(2(n+1) b_{ij}f_{ij}+ (2n+1)f_{ij}+ f_{k\ell}) +b_{ij}.\label{b2n+1}
\end{align}
where  $i,j,k,\ell \in \{1,2,3,4 \}$ are all distinct. 

Note that if $n=1$ then \eqref{b2n} gives $b_{ij}^{2}= 2 b_{ij}f_{ij} + 1 f_{ij}+ f_{k\ell} $ which agrees with relation (3) in Proposition \ref{QHring}. Moreover, it follows from relation (3) together with relations (8) and (11) that 
\begin{align*}
  b_{ij}^2f_{ij} & =   2 b_{ij}f_{ij}^2 + f_{ij}^2+ f_{ij}f_{k\ell} + f_{ij} \\
  & = -f_{ij}^2 + f_{ij}  \\
  & = -2 b_{ij}f_{ij} - f_{ij}.
\end{align*}
Now if we assume \eqref{b2n} then  the previous equation together with  relation (11) yields
\begin{align*}
   b_{ij}^{2(n+1)} &  = b_{ij}^{2n}\, b_{ij}  \\
   & = n (2n \, b_{ij}^2f_{ij}+ (2n-1)f_{ij}b_{ij}+ f_{k\ell}b_{ij}) +b_{ij} = \\
   & = n (2n (-2 b_{ij}f_{ij} - f_{ij}) + (2n-1)f_{ij}b_{ij} - f_{ij}b_{ij} -  f_{ij}- f_{k\ell}) +b_{ij} \\
   & = -n(2(n+1) b_{ij}f_{ij}+ (2n+1)f_{ij}+ f_{k\ell})+b_{ij}.
\end{align*}
Similarly, it is easy to check that 
$$b_{ij}^{2n+1} b_{ij} = (n+1)(2(n+1) \, b_{ij}f_{ij}+ (2n+1)f_{ij}+ f_{k\ell}) +1.$$ 
Next, using \eqref{b2n} and \eqref{b2n+1} we compute $b_{ij}^{\alpha_1}\, b_{ik}^{\alpha_2}\, b_{jk}^{\alpha_3}$, where $\alpha_1,\alpha_2, \alpha_3 \in \Z$, to conclude that indeed  $b_{ij}, b_{ik}, b_{jk}$ are linearly independent. 
Assume first that $\alpha_1,\alpha_2,\alpha_3 \geq 0$ and, in particular, 
$\alpha_i=2n_i$, with $n_i \in \N_0$. 
Then 
\begin{equation*}
  b_{ij}^{\alpha_1}\, b_{ik}^{\alpha_2} = [n_1(2n_1 \, b_{ij}f_{ij}+ (2n-1)f_{ij}+ f_{k\ell}) +1 ] \, [n_2(2n_2 \, b_{ik}f_{ik}+ (2n-1)f_{ik}+ f_{j\ell}) +1 ]. 
\end{equation*}
Recall relation (4) in Proposition \ref{QHring}: $b_{ij}f_{ik}=-f_{ik}$. This  yields 
\begin{align*}
 b_{ij}^{\alpha_1}\, b_{ik}^{\alpha_2}   = & \,  4 n_1^2 \,n_2^2 \,f_{ij}f_{ik} -n_2 (2n_2-1) 2n_1^2 \, f_{ij}f_{ik}  -n_1 (2n_1-1) 2n_2^2 \, f_{ij}f_{ik} - 2n_1^2 \,n_2 \, f_{ij}f_{j\ell}  \\  & + 2n_1^2\,  b_{ij}f_{ij} + n_1 (2n_1-1) \, n_2 (2n_2-1) \, f_{ij}f_{ik} + n_1 (2n_1-1) \,n_2 f_{ij}f_{j\ell} + n_1 (2n_1-1) \,f_{ij} \\  & - 2n_1 \, n_2^2 \,f_{ik}f_{k\ell}
  + n_1\,n_2 (2n_2-1) \,f_{ik}f_{k\ell} + n_1\, n_2\, f_{k\ell}f_{j\ell} + n_1\, f_{k\ell} + 2n_2^2\, b_{ik}f_{ik} \\
  &  + n_2 (2n_2-1) \, f_{ik} + n_2\, f_{j\ell} + 1 \\
  = & \, n_1\, n_2 (f_{ij}f_{ik} -f_{ij}f_{j\ell} -f_{ik}f_{k\ell}+ f_{k\ell}f_{j\ell}) + 2\, n_1^2  b_{ij}f_{ij} + 2\, n_2^2  b_{ik}f_{ik} +  n_1 (2n_1-1) \,f_{ij}\\ 
 & + n_2 (2n_2-1) \, f_{ik} + n_1\, f_{k\ell} +  n_2\, f_{j\ell} + 1 \\
  = & \, 2\, n_1^2  b_{ij}f_{ij} +  2\, n_2^2  (b_{ij}f_{ij} + e_k -e_j) +  n_1 (2n_1-1) \,f_{ij} +  n_2 (2n_2-1) \, (f_{ij} +e_j-e_k) \\
  & + n_1 (f_{ij} +e_i+e_j - e_k - e_l) + n_2 (f_{ij} +e_i- e_l) +1 \\
  = & \, 2 (n_1^2 +n_2^2) f_{ij} (b_{ij} +1 ) + e_i(n_1+n_2) + e_j(n_1-n_2) -e_k(n_2-n_1) - e_\ell(n_1+n_2) +1
\end{align*}
where the step before the last follows from relations (2) and (9) in Proposition \ref{QHring} and the definition of $f_{ij}$ and $e_i$. 
Similar computations then give 
\begin{align*}
b_{ij}^{\alpha_1}\, b_{ik}^{\alpha_2}\, b_{jk}^{\alpha_3} = & \,2 \,(n_1^2 +n_2^2+ n_3^2) f_{ij} (b_{ij} +1 ) + e_i(n_1+n_2-n_3) + e_j(n_1-n_2+n_3)  \\ & +e_k(n_3+n_2-n_1) - e_\ell(n_1+n_2+n_3) +1.
\end{align*}
Now it is clear that $b_{ij}^{\alpha_1}\, b_{ik}^{\alpha_2}\, b_{jk}^{\alpha_3}=1$ iff $n_1=n_2=n_3=0$. 
In the case $\alpha_1= 2n_1+1$, $\alpha_2=2n_2$ and $\alpha_3=2n_3$ analog computations yield 
\begin{align*}
b_{ij}^{\alpha_1}\, b_{ik}^{\alpha_2}\, b_{jk}^{\alpha_3} = & \,-2 \,(n_1^2 +n_2^2+ n_3^2+n_1) f_{ij} (b_{ij} +1 ) - e_i(n_1+n_2-n_3) - e_j(n_1-n_2+n_3)  \\ & - e_k(n_3+n_2-n_1) + e_\ell(n_1+n_2+n_3) + b_{ij}.
\end{align*}
If $\alpha_1= 2n_1+1$, $\alpha_2=2n_2+1$ and $\alpha_3=2n_3$ then 
\begin{align*}
b_{ij}^{\alpha_1}\, b_{ik}^{\alpha_2}\, b_{jk}^{\alpha_3} = & \,[2 \,(n_1^2 +n_2^2+ n_3^2+n_1+n_2)+1] f_{ij} (b_{ij} +1 ) + e_i(n_1+n_2-n_3+1) \\ & + e_j(n_1-n_2+n_3)  + e_k(n_3+n_2-n_1) - e_\ell(n_1+n_2+n_3+1) + 1.
\end{align*}
Finally, if $\alpha_i= \, 2n_i+1 $ one obtains 
\begin{align*}
b_{ij}^{\alpha_1}\, b_{ik}^{\alpha_2}\, b_{jk}^{\alpha_3} = & \,-[2 \,(n_1^2 +n_2^2+ n_3^2+n_1+n_2+n_3)+1] f_{ij} (b_{ij} +1 ) - e_i(n_1+n_2-n_3) \\& - e_j(n_1-n_2+n_3)  - e_k(n_3+n_2-n_1+1) - e_\ell(n_1+n_2+n_3+1)  + b_{ij}.
\end{align*}
So it follows that $b_{ij}^{\alpha_1}\, b_{ik}^{\alpha_2}\, b_{jk}^{\alpha_3}=1$ iff $\alpha_i=0$ for all $i$. 

If $\alpha_m< 0$, then we obtain similar expressions because $b_{ij}^{\alpha_m}= b_{k\ell}^{-\alpha_m}$. 
More precisely, in the case $\alpha_m >0$ the products above contain the term  $n_m(e_i +e_j -e_k-e_\ell)$, but if $\alpha_m< 0$ it is not hard to check that then the term is replaced by its symmetric $n_m(-e_i -e_j +e_k + e_\ell)$. Therefore, we prove that the elements $b_{ij}, b_{ik}, b_{jk}$ are linearly independent, and, in particular, we finish the proof of  Lemma \eqref{li_bs}. 
\end{proof}

\section{Proof of Proposition \ref{basicrelations}}\label{auxiliaryrelations}
This section is devoted to a proof of Proposition \ref{basicrelations}. First we need to prove that the following identifications hold. 
\begin{lemma}\label{basicz0}
Consider the Hamiltonian circle actions $z_0, z_{0,1i}, z_{0,j}$, $ z_1$ and $z_{1,4}$, where $i \in \{2,3,4\}$ and $j  \in \{1,2,3,4\}$, whose graphs are represented in Figure \ref{graphsgeneric}, as elements of the fundamental group of 
$\Symp_h(\Mucccc)$. Then we have the following relations
\begin{align}
& z_1   +z_{0,4}= z_{1,4}+ z_0. \label{relation2}\\
& z_0  = z_{0,3} +z_{0,4}+ z_{0,12}, \label{relation1}\\
& z_{0,1}  = z_{0,14}+z_{0,12} +z_{0,3}, \label{basicz01}\\
& z_{0,2}  = z_{0,12}- z_{0,13} +z_{0,3},  \label{basicz02} \quad \mbox{and}\\
& z_{0,4}  = z_{0,14}- z_{0,13} +z_{0,3},  \label{basicz04}
\end{align}
\end{lemma}
\begin{proof}
In order to prove relation \eqref{relation2} we need to consider  the two toric actions on $\Mucccc$ represented in the polygons of Figure \ref{toricT1T14}. 

\begin{figure}[thp]
\begin{minipage}{.45\textwidth}
\begin{tikzpicture}[scale=0.7, roundnode/.style={circle, draw=black!80, thick, minimum size=7mm}, font=\footnotesize]
 
    \draw[->][black!70] (-3.2,0) -- (0.5,0); 
    \draw[->][black!70] (-3,-3) -- (-3,5.4); 

  \draw[brown] (-3,0) -- (-0.3,-2.7);
  \draw[brown] (-3,0) -- (-3,4);
    \draw[brown] (-3,4) -- (-2 ,5);
    \draw[brown] (-2,5) -- (-1,5);
     \draw[brown] (-1,5) -- (-0.5,4.5);
      \draw[brown] (0,3.5) -- (-0.5,4.5);
    \draw[brown] (0,3.5) -- (0,-2.7);
    \draw[brown] (0,-2.7) -- (-0.3,-2.7);
   
 \draw[dashed, black!60] (-0.3,-2.7) --(-0.3,0);
 \draw[dashed, black!60] (-2,0) -- (-2,5);
\draw[dashed, black!60] (-1,0) -- (-1,5);
\draw[dashed, black!60] (-0.5,4.5) -- (-0.5,0);
\draw[dashed, black!60] (-3,5) -- (-2,5);
\draw[dashed, black!60] (-3,4.5) -- (-0.5,4.5);
\draw[dashed, black!60] (-3,3.5) -- (0,3.5);
\draw[dashed, black!60] (-3,-2.7) -- (-0.3,-2.7);

  \node at (-3.8,5.1) {$\mu -c_1$};
   \node at (-4.7,4.5) {$ \mu -c_1-c_2+c_3$};
     \node at (-4.7,3.5) {$ \mu -c_1-c_2-c_3$};
 \node[blue] at (-3.8,2) {$B-F$};
 \node[blue] at (2.8,2.7) {$B+F -E_1-E_2-E_3-E_4$};
     \node at (-3.8,-2.7) {$-1+c_4$};
      \node[rotate=45] at (-2.4,-0.5) {$1-c_1$};
   \node[rotate=45] at (-1.5,-0.5) {$1-c_2$};
    \node[rotate=45] at (-0.9,-0.5) {$1-c_3$};
     \node[rotate=45] at (0.1,0.5) {$1-c_4$};
 
\end{tikzpicture}
\end{minipage}%
\begin{minipage}{.5\textwidth}
\quad 
\begin{tikzpicture}[scale=0.7, roundnode/.style={circle, draw=black!80, thick, minimum size=7mm}, font=\footnotesize]
 
  \draw[->][black!70] (-3.2,0) -- (0.5,0); 
    \draw[->][black!70] (-3,-3.2) -- (-3,5.4); 

  \draw[brown] (-2.7,-0.3) -- (-0,-3);
  \draw[brown] (-3,0.3) -- (-3,4);
   \draw[brown] (-2.7,-0.3) -- (-3,0.3);
    \draw[brown] (-3,4) -- (-2 ,5);
    \draw[brown] (-2,5) -- (-1,5);
     \draw[brown] (-1,5) -- (-0.5,4.5);
      \draw[brown] (0,3.5) -- (-0.5,4.5);
    \draw[brown] (0,3.5) -- (0,-3);

 \draw[dashed, black!60] (-2.7,-0.2) --(-2.7,0);
 \draw[dashed, black!60] (-2,0) -- (-2,5);
\draw[dashed, black!60] (-1,0) -- (-1,5);
\draw[dashed, black!60] (-0.5,4.5) -- (-0.5,0);
\draw[dashed, black!60] (-3,5) -- (-2,5);
\draw[dashed, black!60] (-3,4.5) -- (-0.5,4.5);
\draw[dashed, black!60] (-3,3.5) -- (0,3.5);
\draw[dashed, black!60] (-3,-3) -- (0,-3);
\draw[dashed, black!60] (-3,-0.3) -- (-2.7,-0.3);

  \node at (-3.8,5.1) {$\mu -c_1$};
   \node at (-4.7,4.5) {$ \mu -c_1-c_2+c_3$};
     \node at (-4.7,3.5) {$ \mu -c_1-c_2-c_3$};
        \node[blue] at (-4.2,2) {$B-F-E_4$};
   \node[blue] at (2.3,2.4) {$B+F -E_1-E_2-E_3$};
    \node at (-3.5,-0.3) {$-c_4$};
    \node at (-3.4,-3) {$-1$};
  \node[rotate=45] at (-2.4,-0.5) {$1-c_1$};
   \node[rotate=45] at (-1.5,-0.5) {$1-c_2$};
    \node[rotate=45] at (-0.9,-0.5) {$1-c_3$};
    \node[rotate=45] at (-2.6,0.3) {$c_4$};

\end{tikzpicture}
\end{minipage}%

\caption{Auxiliary toric actions}
 \label{toricT1T14}
 \end{figure}
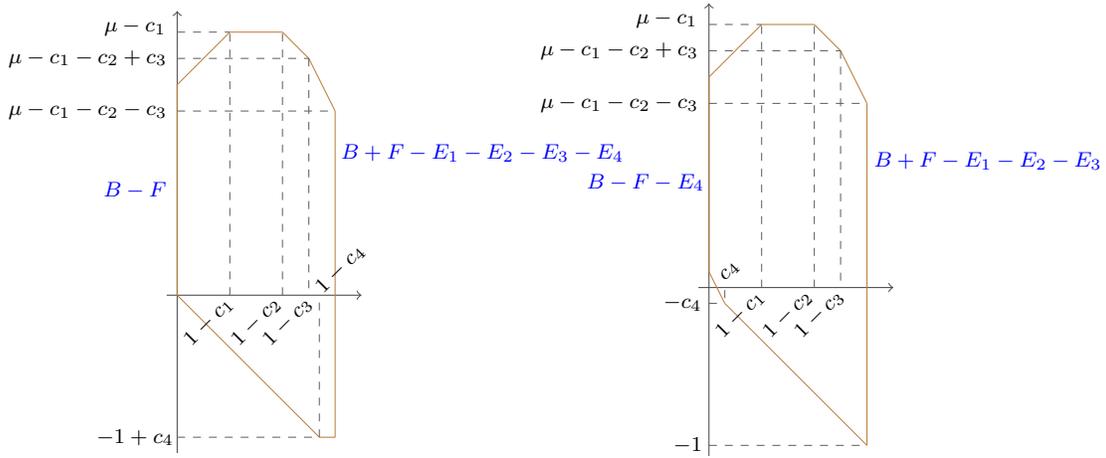

Note that projecting onto the $x$-axis we obtain the graphs of the circle actions $z_1$ on the left and $z_{1,4}$ on the right. On the other hand, projecting onto the $y$-axis yields the graphs of two circle actions (see Figure \ref{graphsc1c14}) that we will denote by $c_1$ and $c_{1,4}$. 

\begin{figure}[thp]
\begin{minipage}{.45\textwidth}
\begin{tikzpicture}[scale=1.1,font=\footnotesize]
       \fill[black] (2,1.8) ellipse (1 and 0.2);
	\node at (4.6,1.8) {$\displaystyle{\mu -c_1, \ c_1-c_2 }$};
	\filldraw [black] (2.1,0.3) circle (2pt);
	\filldraw [black] (2.1,1) circle (2pt);
	 \draw[dashed, black!60] (2.1,1) --(2.1,1.8); 
	\node at (2.7,1.4) {$ E_2-E_3$};
	 \draw (2.1,0.3) -- (2.1,1);
	 \node at (1.9,0.7) {$2$};
	  \node at (2.3,0.7) {$E_3$};
	   \draw[dashed, black!60] (2.1,0.3) --(2.1,-1.5); 
	    \node at (3.9,-1) {$B+F-E_1-E_2-E_3-E_4$};
	\node at (4.6 ,1) {$\mu-c_1 -c_2 +c_3$};
        \node at (4.6 ,0.2) {$\mu-c_1 -c_2 -c_3$};
	\filldraw [black] (1.5,0.5) circle (2pt);
	 \draw[dashed, black!60] (1.5,0.5) --(1.5,1.8); 
	\node at (1,1) {$ F-E_1$};
        \filldraw [black] (1.5,-0.5) circle (2pt);
        \draw[dashed, black!60] (1.5,-0.5) --(1.5,-1.5); 
        \node at (1 ,-0.1) {$ B-F$};
         \draw[dashed, black!60] (1.5,-0.5) --(1.5,0.5); 
	\node at (1 ,-1) {$ F-E_4$};
	\node at (4 ,0.6) {$\mu-1$};
        \node at (3.8 ,-0.5) {$0$};
        \fill[black] (2,-1.5) ellipse (1 and 0.2);
	\node at (4.3,-1.5) {$ -1+c_4, \ c_4$};
     \node at (2,-2.2) {Circle action $c_{1}$};
\end{tikzpicture}

\end{minipage}%
\begin{minipage}{.45\textwidth}
\quad 
\begin{tikzpicture}[scale=1.1,font=\footnotesize]
       \fill[black] (2,1.8) ellipse (1 and 0.2);
	\node at (4.55,1.8) {$\displaystyle{\mu -c_1, \ c_1-c_2 }$};
	\filldraw [black] (2.3,0.3) circle (2pt);
	\filldraw [black] (2.3,1) circle (2pt);
	 \draw[dashed, black!60] (2.3,1) --(2.3,1.8); 
	\node at (2.9,1.4) {$ E_2-E_3$};
	 \draw (2.3,0.3) -- (2.3,1);
	 \node at (2.1,0.7) {$2$};
	  \node at (2.6,0.7) {$E_3$};	
	   \filldraw [black] (1.5,0.5) circle (2pt);
	 \draw[dashed, black!60] (1.5,0.5) --(1.5,1.8); 
	\node at (1,1) {$ F-E_1$};
        \node at (4.6 ,1) {$\mu-c_1 -c_2 +c_3$};
        \node at (4.6 ,0.2) {$\mu-c_1 -c_2 -c_3$};	
        
	\filldraw [black] (1.5,-0.5) circle (2pt);
	\node at (3.9 ,0.6) {$\mu-1$};
	        \node at (3.7 ,-0.5) {$c_4$};
        \filldraw [black] (2,-1.5) circle (2pt);
         \filldraw [black] (1.5,-1) circle (2pt);
         \draw[dashed, black!60] (1.5,-0.5) --(1.5,0.5); 
         \node at (0.7 ,-0.1) {$ B-F-E_4$};
          \node at (1.75,-0.8) {$E_4$};
          \node at (1.3 ,-0.75) {$2$};
         \draw (1.5,-1) -- (1.5,-0.5);
          \draw[dashed, black!60] (1.5,-1) --(2,-1.5); 
            \draw[dashed, black!60] (2.3,0.3) --(2,-1.5); 
	\node at (1.3 ,-1.4) {$ F-E_4$};
	\node at (3.6,-0.2) {$B+F-E_1-E_2-E_3$};
         \node at (3.8,-1) {$ -c_4$};
	\node at (3.7,-1.5) {$ -1$};
     \node at (2,-2.2) {Circle action $c_{1,4}$};
\end{tikzpicture}
\end{minipage}%

\caption{Graphs of the actions $c_1$ and $c_{1,4}$}
 \label{graphsc1c14}
 \end{figure}
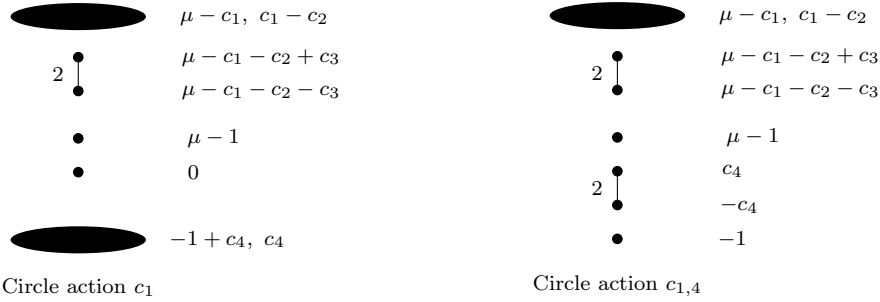

It is easy to check that  performing first a $GL(2, \Z)$ transformation represented by the matrix 
$$ \left (
\begin{array}{cc}
1  & 0 \\
1 & 1 
\end{array} \right)$$ 
to both polygons of Figure \ref{toricT1T14}
and then a projection onto the $y$-axis we obtain the same graph which implies that the following identification holds 
\begin{equation} \label{auxiliaryrelation}
z_1+ c_1 = z_{1,4}+ c_{1,4}. 
\end{equation}

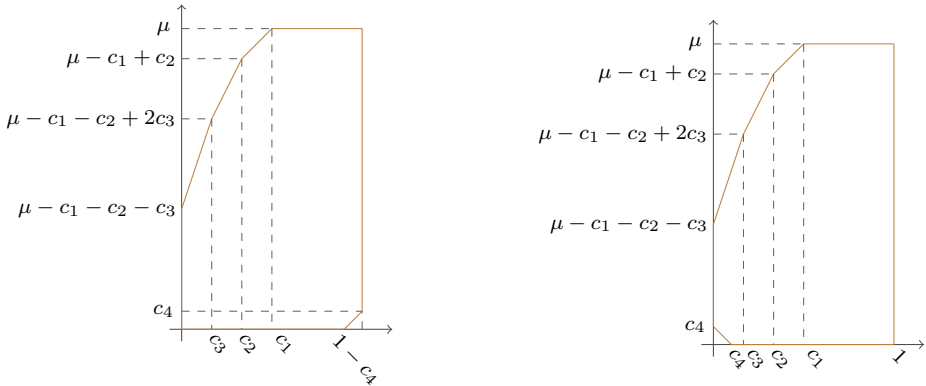
\begin{figure}[thp]
\begin{minipage}{.45\textwidth}
\begin{tikzpicture}[scale=0.8, roundnode/.style={circle, draw=black!80, thick, minimum size=7mm}, font=\footnotesize]
 
    \draw[->][black!70] (-3.2,0) -- (0.5,0); 
    \draw[->][black!70] (-3,-0.2) -- (-3,5.4); 

	\draw[brown] (-0.3,0) -- (0,0.3);
	\draw[brown] (-3,0) -- (-3,2);
    \draw[brown] (-2,4.5) -- (-1.5,5);
     \draw[brown] (-2.5,3.5) -- (-2,4.5);
     \draw[brown] (-2.5,3.5) -- (-3,2);
    \draw[brown] (-1.5,5) -- (0,5);
    \draw[brown] (0,5) -- (0,0.3);
    \draw[brown] (-3,0) -- (-0.3,0);

	\draw[dashed, black!60] (-1.5,5) --(-1.5,0);
	\draw[dashed, black!60] (-2,4.5) -- (-2,0);
	\draw[dashed, black!60] (-3,5) -- (-1.5,5);
	\draw[dashed, black!60] (-3,4.5) -- (-2,4.5);
	\draw[dashed, black!60] (-2.5,3.5) -- (-2.5,0);
	\draw[dashed, black!60] (-3,3.5) -- (-2.5,3.5);
	\draw[dashed, black!60] (-3,0.3) -- (0,0.3);
        \draw[dashed, black!60] (0,0) -- (0,0.3);

    \node at (-3.3,0.3) {$c_{4}$};
    \node at (-4.4,2) {$\mu-c_1-c_2-c_3$};
	\node at (-3.3,5) {$\mu$};
	\node at (-4,4.5) {$\mu-c_{1}+c_2 $};
	\node at (-4.5,3.5) {$\mu-c_1-c_{2}+2c_3 $};
	\node[rotate=-45] at (-1.3,-0.25) {$ c_{1} $};
        \node[rotate=-45] at (-1.9,-0.25) {$  c_{2} $};
         \node[rotate=-45] at (-2.4,-0.25) {$  c_{3} $};
	\node[rotate=-45] at (-0.1, -0.5) {$ 1-c_4$};

\end{tikzpicture}
\end{minipage}%
\begin{minipage}{.45\textwidth}
\quad 
\begin{tikzpicture}[scale=0.8, roundnode/.style={circle, draw=black!80, thick, minimum size=7mm}, font=\footnotesize]
 
    \draw[->][black!70] (-3.2,0) -- (0.5,0); 
    \draw[->][black!70] (-3,-0.2) -- (-3,5.4); 

	\draw[brown] (-3,0.3) -- (-2.7,0);
	\draw[brown] (-3,0.3) -- (-3,2);
    \draw[brown] (-2,4.5) -- (-1.5,5);
     \draw[brown] (-2.5,3.5) -- (-2,4.5);
     \draw[brown] (-2.5,3.5) -- (-3,2);
    \draw[brown] (-1.5,5) -- (0,5);
    \draw[brown] (0,5) -- (0,0);
    \draw[brown] (-2.7,0) -- (-0,0);

	\draw[dashed, black!60] (-1.5,5) --(-1.5,0);
	\draw[dashed, black!60] (-2,4.5) -- (-2,0);
	\draw[dashed, black!60] (-3,5) -- (-1.5,5);
	\draw[dashed, black!60] (-3,4.5) -- (-2,4.5);
	\draw[dashed, black!60] (-2.5,3.5) -- (-2.5,0);
	\draw[dashed, black!60] (-3,3.5) -- (-2.5,3.5);

    \node at (-3.3,0.3) {$c_{4}$};
    \node at (-4.4,2) {$\mu-c_1-c_2-c_3$};
	\node at (-3.3,5) {$\mu$};
	\node at (-4,4.5) {$\mu-c_{1}+c_2 $};
	\node at (-4.5,3.5) {$\mu-c_1-c_{2}+2c_3 $};
	\node[rotate=-45] at (-1.3,-0.25) {$ c_{1} $};
        \node[rotate=-45] at (-1.9,-0.25) {$  c_{2} $};
         \node[rotate=-45] at (-2.3,-0.25) {$  c_{3} $};
        \node[rotate=-45] at (-2.6,-0.25) {$  c_{4} $};
	\node[rotate=-45] at (0.1, -0.2) {$ 1$};

\end{tikzpicture}
\end{minipage}%
\caption{Toric actions $(a_1,b_1)$ and $(a_{1,4},b_{1,4})$}
 \label{auxiliarypolytopes}
 \end{figure}

Next consider the polygons of Figure \ref{auxiliarypolytopes}. Again they represent particular toric actions on $\Mucccc$. We denote the circle actions whose graphs are obtained by projecting to the $x$ and $y$-axis by $(a_1,b_1)$ and $(a_{1,4}, b_{1,4})$. Since, clearly the actions $b_1$ and $b_{1,4}$ are represented by the same graph, it follows that $b_1=b_{1,4}$. Next, performing 
a $GL(2, \Z)$ transformation represented by the matrix 
$$ \left (
\begin{array}{cc}
1  & 0 \\
-1 & 1 
\end{array} \right)$$ 
to both polygons and then projecting to the $y$-axis  we obtain the graphs of $c_1$ and $c_{1,4}$. Therefore, as elements of $\pi_1 (\Symp_h( \Mucccc))$, the following identifications hold 
$$ b_1 - a_1= c_1 \quad \mbox{and} \quad  b_{1,4} -a_{1,4}= c_{1,4}.$$ On the other hand note that $a_1$ is the generator  $-z_{0,4}$ while  $a_{1,4}$  is  $ -z_0$. Therefore we have
$$ b_1 +z_{0,4}= c_1 \quad \mbox{and} \quad  b_{1,4} +z_0= c_{1,4}. $$
Finally, substituting $c_1$ and $c_{1,4}$ in relation \eqref{auxiliaryrelation} and using $b_1=b_{1,4}$, we obtain the desired relation \eqref{relation2}.

In order to prove the remaining relations in this lemma we use an argument involving several auxiliary polygons representing different toric actions on $\Mucccc$ that can be related between each other using 
Karshon's and Delzant's classifications. Since the argument is similar for all relations we give the proof for relation \eqref{relation1} and leave the other proofs for the interested reader.

\begin{figure}[thp]
\begin{minipage}{.33\textwidth}
\begin{tikzpicture}[roundnode/.style={circle, draw=black!80, thick, minimum size=6mm}, font=\small]

	\node[roundnode] at (7,6) (maintopic) {1} ;
 
    \draw[->][black!70] (7.8,0) -- (9.2,0); 
    \draw[->][black!70] (8,-1.1) -- (8,5.2); 

    \draw[brown] (9,4.3) -- (9,-0.9);
    \draw[brown] (9,-0.9) -- (8.5,-0.5);
    \draw[brown] (8.5,-0.5) -- (8.3,-0.1);
    \draw[brown] (8.3,-0.1) -- (8,0.8);
    \draw[brown] (8,0.8) -- (8,4);
    \draw[brown] (8,4) -- (8.5,4.5);
    \draw[brown] (8.8,4.5) -- (8.5,4.5);
    \draw[brown] (8.8,4.5) -- (9,4.3);
    
	\draw[dashed, black!60] (8.5,4.5) -- (8,4.5);
	\draw[dashed, black!60] (9,4.3) -- (8,4.3);
	\draw[dashed, black!60] (8.8,4.5) -- (8.8,0);
	\draw[dashed, black!60] (8.5,4.5) -- (8.5,0);
	\draw[dashed, black!60] (8.5,-0.5) -- (8.5,0);
	\draw[dashed, black!60] (8.5,-0.5) -- (8,-0.5);
	\draw[dashed, black!60] (8.3,-0.1) -- (8.3,0);
	\draw[dashed, black!60] (8.3,-0.1) -- (8,-0.1);

	\node at (7.5,4.6) {$\mu-c_1$};
	\node at (7.1,4.3) {$\mu-c_1-c_2$};
	\node at (7.5,3.9) {$\mu -1 $};
	\node at (7.2,0.8) {$c_{3} + c_{4} $};
	\node at (7.2,-0.1) {$c_{3} - 2 c_{4}$};
	\node at (7.5,-0.5) {$-c_{3}$};
	\node[rotate=45] at (8.8,0.4) {$1-c_{1}$};
	\node[rotate=45] at (8.3,-0.2) {$ c_{4} $};
	\node[rotate=45] at (8.7,-0.2) {$ c_{3} $};
	\node[rotate=45] at (9.1,0.4) {$ 1 - c_{2}$};
    
	\node at (8,-2) {($x_{1}$, $y_{1}$)};

\end{tikzpicture}
\end{minipage}%
\begin{minipage}{.33\textwidth}
\begin{tikzpicture}[roundnode/.style={circle, draw=black!80, thick, minimum size=6mm}, font=\small]

	\node[roundnode] at (7,6) (maintopic) {2} ;
 
    \draw[->][black!70] (7.8,0) -- (9.2,0); 
    \draw[->][black!70] (8,-1.1) -- (8,5.2); 

    \draw[brown] (9,4.3) -- (9,-0.9);
    \draw[brown] (9,-0.9) -- (8.5,-0.5);
    \draw[brown] (8.5,-0.5) -- (8,0.5);
    \draw[brown] (8,0.5) -- (8,3.8);
    \draw[brown] (8.5,4.5) -- (8.1,4.1);
    \draw[brown] (8,3.8) -- (8.1,4.1);
    \draw[brown] (8.8,4.5) -- (8.5,4.5);
    \draw[brown] (8.8,4.5) -- (9,4.3);
    
	\draw[dashed, black!60] (8.5,4.5) -- (8,4.5);
	\draw[dashed, black!60] (9,4.3) -- (8,4.3);
	\draw[dashed, black!60] (8.8,4.5) -- (8.8,0);
	\draw[dashed, black!60] (8.5,4.5) -- (8.5,0);
	\draw[dashed, black!60] (8.5,-0.5) -- (8.5,0);
	\draw[dashed, black!60] (8.1,4.1) -- (8,4.1);
	\draw[dashed, black!60] (8.1,4.1) -- (8.1,0);
	\draw[dashed, black!60] (8.5,-0.5) -- (8,-0.5);

	\node at (7.5,4.6) {$\mu-c_1$};
	\node at (7.1,4.35) {$\mu-c_1-c_2$};
	\node at (7.2,3.8) {$\mu -1 -c_4$};
	\node at (7.8,0.5) {$c_{3}  $};
        \node at (7.6,-0.5) {$-c_{3}$};
	\node[rotate=45] at (8.1,-0.2) {$c_{4}$};
	\node[rotate=45] at (8.8,0.4) {$1-c_{1}$};
	\node at (7.2,4.1) { $\mu -1 +c_4 $};
	\node[rotate=45] at (8.7,-0.2) {$ c_{3} $};
	\node[rotate=45] at (9.1,0.4) {$ 1 - c_{2}$};
    
	\node at (8,-2) {($x_{2}$, $y_{2}$)};

\end{tikzpicture}
\end{minipage}%
\begin{minipage}{.33\textwidth}
\begin{tikzpicture}[roundnode/.style={circle, draw=black!80, thick, minimum size=6mm}, font=\small]

	\node[roundnode] at (7,6) (maintopic) {3} ;
 
    \draw[->][black!70] (7.8,0) -- (9.2,0); 
    \draw[->][black!70] (8,-1.1) -- (8,5.2); 

    \draw[brown] (9,4.3) -- (9,-0.9);
       \draw[brown] (9,-0.9) -- (8.1,-0.2);
    \draw[brown] (8.1,-0.2) -- (8,0.2);
    \draw[brown] (8,0.5) -- (8,3.7);
    \draw[brown] (8.5,4.5) -- (8.2,4.3);
    \draw[brown] (8,3.7) -- (8.2,4.3);
    \draw[brown] (8.8,4.5) -- (8.5,4.5);
    \draw[brown] (8.8,4.5) -- (9,4.3);
    
	\draw[dashed, black!60] (8.5,4.5) -- (8,4.5);
	\draw[dashed, black!60] (9,4.3) -- (8,4.3);
	\draw[dashed, black!60] (8.8,4.5) -- (8.8,0);
	\draw[dashed, black!60] (8.5,4.5) -- (8.5,0);
	\draw[dashed, black!60] (8.2,4.3) -- (8.2,0);
	\draw[dashed, black!60] (8.1,-0.2) -- (8.1,0);
	\draw[dashed, black!60] (8.1,-0.2) -- (8,-0.2);

	\node at (7.5,4.6) {$\mu-c_1$};
	\node at (7.1,4.35) {$\mu-c_1-c_2$};
	\node at (7.2,3.7) {$\mu -1 -c_3$};
	\node at (7.8,0.2) {$c_{4}  $};
        \node at (7.6,-0.2) {$-c_{4}$};
	\node[rotate=45] at (8.25,0.2) {$c_{3}$};
	\node[rotate=45] at (8.8,0.4) {$1-c_{1}$};
	\node at (7.2,4.1) { $\mu -1 +c_3 $};
	\node[rotate=45] at (9.1,0.4) {$ 1 - c_{2}$};
    
	\node at (8,-2) {($x_{3}$, $y_{3}$)};

\end{tikzpicture}
\end{minipage}%

\bigskip 
\begin{minipage}{.33\textwidth}
\begin{tikzpicture}[roundnode/.style={circle, draw=black!80, thick, minimum size=6mm}, font=\small]

	\node[roundnode] at (7,6) (maintopic) {4} ;
 
    \draw[->][black!70] (7.8,0) -- (9.2,0); 
    \draw[->][black!70] (8,-1.1) -- (8,5.2); 

    \draw[brown] (9,4.3) -- (9,-0.9);
    \draw[brown] (9,-0.9) -- (8.9,-0.9);
    \draw[brown] (8.9,-0.9) -- (8,0);
    \draw[brown] (8,0) -- (8,3.7);
    \draw[brown] (8.5,4.5) -- (8.2,4.3);
    \draw[brown] (8,3.7) -- (8.2,4.3);
    \draw[brown] (8.8,4.5) -- (8.5,4.5);
    \draw[brown] (8.8,4.5) -- (9,4.3);
    
	\draw[dashed, black!60] (8.5,4.5) -- (8,4.5);
	\draw[dashed, black!60] (9,4.3) -- (8,4.3);
	\draw[dashed, black!60] (8.8,4.5) -- (8.8,0);
	\draw[dashed, black!60] (8.5,4.5) -- (8.5,0);
	\draw[dashed, black!60] (8.2,4.3) -- (8.2,0);
	\draw[dashed, black!60] (8.9,-0.9) -- (8.9,0);
	\draw[dashed, black!60] (8.9,-0.9) -- (8,-0.9);

	\node at (7.5,4.6) {$\mu-c_1$};
	\node at (7.1,4.35) {$\mu-c_1-c_2$};
	\node at (7.2,3.7) {$\mu -1 -c_3$};
	\node[rotate=-45] at (9.2,-0.4) {$1-c_{4}  $};
        \node at (7.4,-0.9) {$-1+c_{4}$};
	\node[rotate=45] at (8.25,0.2) {$c_{3}$};
	\node[rotate=45] at (8.8,0.4) {$1-c_{1}$};
	\node at (7.2,4.1) { $\mu -1 +c_3 $};
	\node[rotate=45] at (9.1,0.4) {$ 1 - c_{2}$};

	\node[rotate=45] at (9.1,0.4) {$ 1 - c_{2}$};
    
	\node at (8,-2) {($x_{4}$, $y_{4}$)};

\end{tikzpicture}
\end{minipage}
\begin{minipage}{.33\textwidth}
\begin{tikzpicture}[roundnode/.style={circle, draw=black!80, thick, minimum size=6mm}, font=\small]

	\node[roundnode] at (7,6) (maintopic) {5} ;
 
    \draw[->][black!70] (7.8,0) -- (9.2,0); 
    \draw[->][black!70] (8,-1.1) -- (8,5.2);

    \draw[brown] (8,4) -- (8.5,4.5);
    \draw[brown] (8.8,4.5) -- (8.5,4.5);
    \draw[brown] (8.8,4.5) -- (9,4.3);
    \draw[brown] (9,4.3) -- (9,-0.9);
    \draw[brown] (9,-0.9) -- (8.9,-0.9);
    \draw[brown] (8.9,-0.9) -- (8.5,-0.5);
    \draw[brown] (8.5,-0.5) -- (8,0.5);
    \draw[brown] (8,0.5) -- (8,3.8);
    \draw[brown] (8.5,4.5) -- (8,4);
    \draw[brown] (8.8,4.5) -- (8.5,4.5);
    \draw[brown] (8.8,4.5) -- (9,4.3);
    
	\draw[dashed, black!60] (8.5,4.5) -- (8,4.5);
	\draw[dashed, black!60] (9,4.3) -- (8,4.3);
	\draw[dashed, black!60] (8.8,4.5) -- (8.8,0);
	\draw[dashed, black!60] (8.5,4.5) -- (8.5,0);
	\draw[dashed, black!60] (8.5,-0.5) -- (8.5,0);
	\draw[dashed, black!60] (8.5,-0.5) -- (8,-0.5);
	\draw[dashed, black!60] (8.9,-0.9) -- (8.9,0);
	\draw[dashed, black!60] (8.9,-0.9) -- (8,-0.9);

	\node at (7.5,4.6) {$\mu-c_1$};
	\node at (7.1,4.35) {$\mu-c_1-c_2$};
	\node at (7.6,4) {$\mu -1$};
	\node at (7.8,0.5) {$c_{3}  $};
    \node at (7.6,-0.5) {$-c_{3}$};
	\node[rotate=45] at (8.8,0.4) {$1-c_{1}$};
	\node[rotate=45] at (8.7,-0.2) {$ c_{3} $};
	\node[rotate=45] at (9.1,0.4) {$ 1 - c_{2}$};
	\node[rotate=-45] at (9.2,-0.4) {$1-c_{4}  $};
    \node at (7.4,-0.9) {$-1+c_{4}$};
	\node at (8,-2) {($x_{5}$, $y_{5}$)};

\end{tikzpicture}
\end{minipage}%
\begin{minipage}{.33\textwidth}
\begin{tikzpicture}[roundnode/.style={circle, draw=black!80, thick, minimum size=6mm}, font=\small]

	\node[roundnode] at (7,6) (maintopic) {6} ;
    \draw[->][black!70] (7.8,0) -- (9.2,0); 
    \draw[->][black!70] (8,-1.1) -- (8,5.2); 

    \draw[brown] (9,4.3) -- (9,-0.9);
    \draw[brown] (9,-0.9) -- (8.5,-0.5);
    \draw[brown] (8.5,-0.5) -- (8.3,-0.1);
    \draw[brown] (8.3,-0.1) -- (8,0.8);
    \draw[brown] (8,0.8) -- (8,4);
    \draw[brown] (8,4) -- (8.5,4.5);
    \draw[brown] (8.8,4.5) -- (8.5,4.5);
    \draw[brown] (8.8,4.5) -- (9,4.3);
    
	\draw[dashed, black!60] (8.5,4.5) -- (8,4.5);
	\draw[dashed, black!60] (9,4.3) -- (8,4.3);
	\draw[dashed, black!60] (8.8,4.5) -- (8.8,0);
	\draw[dashed, black!60] (8.5,4.5) -- (8.5,0);
	\draw[dashed, black!60] (8.5,-0.5) -- (8.5,0);
	\draw[dashed, black!60] (8.5,-0.5) -- (8,-0.5);
	\draw[dashed, black!60] (8.3,-0.1) -- (8.3,0);
	\draw[dashed, black!60] (8.3,-0.1) -- (8,-0.1);

	\node at (7.5,4.6) {$\mu-c_1$};
	\node at (7.1,4.3) {$\mu-c_1-c_2$};
	\node at (7.5,3.9) {$\mu -1 $};
	\node at (7.8,0.8) {$c_{3} $};
	\node at (7.2,-0.1) {$-c_{3} + 2 c_{4}$};
	\node at (7.3,-0.5) {$-c_{3}-c_4$};
	\node[rotate=45] at (8.5,0.4) {$ c_3-c_{4} $};
	\node[rotate=45] at (8.9,0.4) {$ c_{3} +c_4$};
	\node at (8,-2) {($x_{6}$, $y_{6}$)};
\end{tikzpicture}
\end{minipage}%
\caption{Auxiliary Delzant polygons}\label{aux1}
\end{figure}
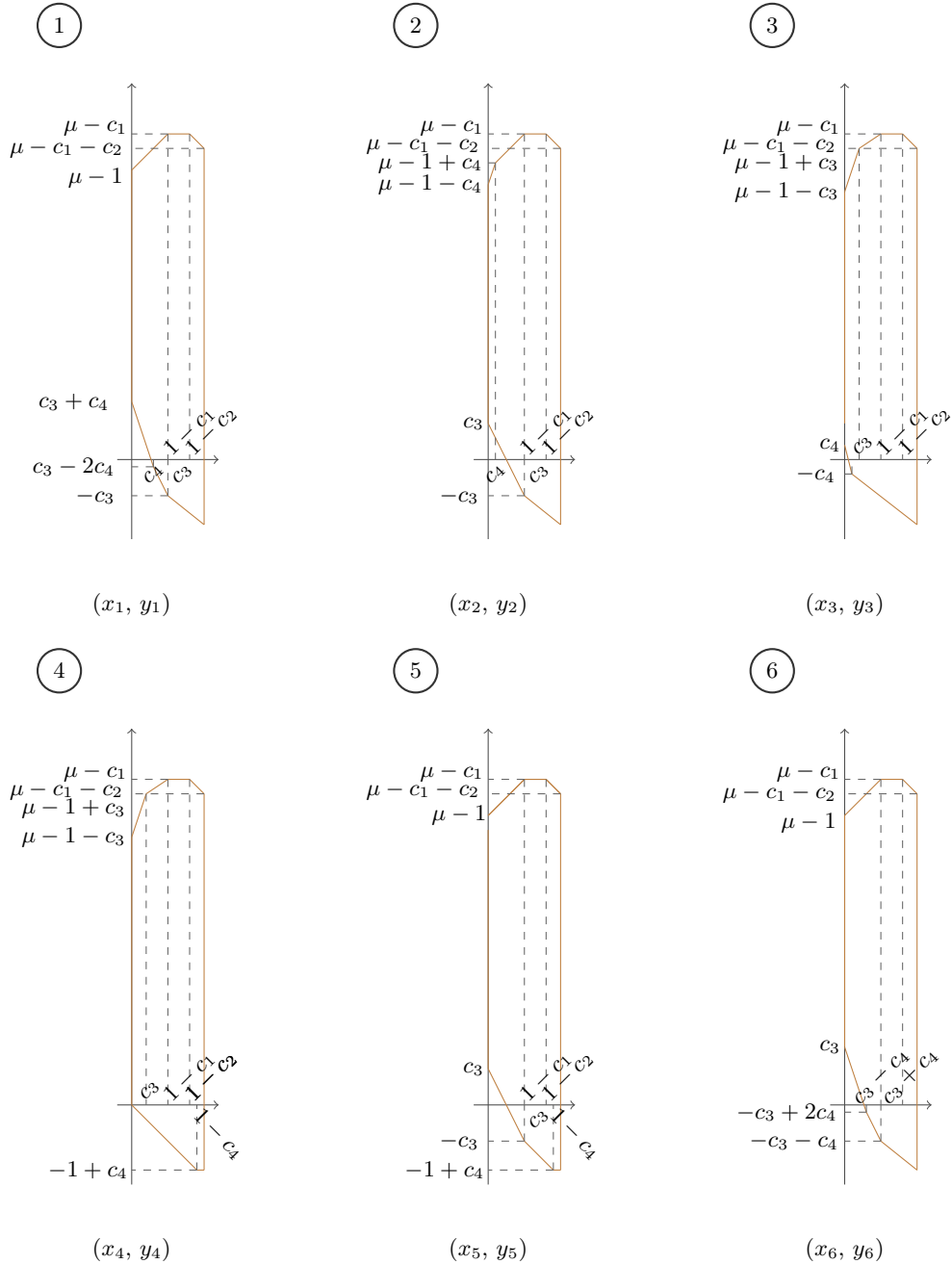

First consider the toric actions on $\Mucccc$ represented in the polygons of Figure \ref{aux1}. Denote the toric action in polygon $n \in \{1,2,3,4,5,6\}$ in Figure \ref{aux1} by $(x_n, y_n)$. Projecting onto the $x$-axis to obtain the graph of $x_n$, using Karshon's classification of Hamiltonian circle actions, it is clear that 
\begin{equation}\label{basicx}
x_1=x_2=x_3 \quad \mbox{and}\quad  x_4=x_5. 
\end{equation}
Moreover, applying the $GL(2, \Z)$ transformation represented by the matrix 
$$ \left (
\begin{array}{cc}
1  & 0 \\
1 & 1 
\end{array} \right)$$ 
to the polygons 3,4,5,6 and then projecting onto the $y$-axis it is easy to check that, as elements of the fundamental group, the following identifications hold 
\begin{align}
x_3+y_3 & =x_4+y_4, \label{relxy1}\\ 
x_5+y_5 & =x_6+y_6.  \label{relxy2}
\end{align}
Furthemore, the $GL(2, \Z)$ transformation represented by the matrix 
$$ \left (
\begin{array}{cc}
1  & 0 \\
2 & 1 
\end{array} \right)$$ 
applied to the polygons 1 and 6 yields, after a projection onto the $y$-axis and a comparision of  the graphs obtained, the following relation
\begin{equation}
2x_1+y_1=2x_6+y_6. \label{relxy3}
\end{equation}

Next consider the toric actions on $\Mucccc$ represented in the polygons of Figure \ref{aux2} and denote the toric action of polygon $n \in \{1,2,3,4,5,6\}$ by $(s_n,t_n)$.
\begin{figure}[thp]
\begin{minipage}{.33\textwidth}
\begin{tikzpicture}[scale=0.7, roundnode/.style={circle, draw=black!80, thick, minimum size=6mm}, font=\footnotesize]
 
 \node[roundnode] at (-4,7) (maintopic) {1} ;
 
    \draw[->][black!70] (-3.2,0) -- (0.6,0);
    \draw[->][black!70] (-3,-0.2) -- (-3,6.2); 
\draw[brown](-3,1) -- (-3,2.5);
   \draw[brown] (-2,4.5) -- (-1.5,5);
     \draw[brown] (-3,2.5) -- (-2,4.5);

    \draw[brown] (-1.5,5) -- (0,5);
  \draw[brown] (0,5) -- (0,0);
    \draw[brown] (-2.3,0) -- (0,0);
   \draw[brown] (-2.3,0) -- (-2.7,0.4);
    \draw[brown] (-3,1) -- (-2.7,0.4);
   
	\draw[dashed, black!60] (-1.5,5) --(-1.5,0);
	\draw[dashed, black!60] (-2,4.5) -- (-2,0);
	\draw[dashed, black!60] (-3,5) -- (-1.5,5);
	\draw[dashed, black!60] (-3,4.5) -- (-2,4.5);
	\draw[dashed, black!60] (-2.9,0.4) -- (-2.7,0.4);
	\draw[dashed, black!60] (-2.7,0) -- (-2.7,0.4);
    \node at (-3.7,0.4) {$c_3- c_{4}$};
    \node at (-3.7,1) {$c_3+c_{4}$};
	\node at (-3.3,5) {$\mu$};
	\node at (-4.1,4.5) {$\mu-c_{1}+c_2 $};
	\node at (-4.1,2.5) {$\mu-c_1-c_{2} $};
	\node[rotate=-45] at (-1.3,-0.25) {$ c_{1} $};
    \node[rotate=-45] at (-1.9,-0.25) {$  c_{2} $};
    \node[rotate=-45] at (-2.3,-0.25) {$  c_{3} $};
    \node[rotate=-45] at (-2.6,-0.25) {$  c_{4} $};

\node at (-1.5,-1.5) {($s_{1}$, $t_{1}$)};

\end{tikzpicture}
\end{minipage}%
\begin{minipage}{.33\textwidth}
\begin{tikzpicture}[scale=0.7, roundnode/.style={circle, draw=black!80, thick, minimum size=6mm}, font=\footnotesize]
 
  \node[roundnode] at (-4,7) (maintopic) {2} ;
 
    \draw[->][black!70] (-3.2,0) -- (0.6,0); 
    \draw[->][black!70] (-3,-0.2) -- (-3,6.2);

\draw[brown](-3,0.7) -- (-3,2.5);
   \draw[brown] (-2,4.5) -- (-1.5,5);
     \draw[brown] (-3,2.5) -- (-2,4.5);

    \draw[brown] (-1.5,5) -- (-0.3,5);
  \draw[brown] (0,4.7) -- (0,0);
    \draw[brown] (-2.3,0) -- (0,0);
   \draw[brown] (-2.3,0) -- (-3,0.7);
   \draw[brown] (0,4.7) -- (-0.3,5);
   
	\draw[dashed, black!60] (-1.5,5) --(-1.5,0);
	\draw[dashed, black!60] (-2,4.5) -- (-2,0);
	\draw[dashed, black!60] (-3,5) -- (-1.5,5);
	\draw[dashed, black!60] (-3,4.5) -- (-2,4.5);
	\draw[dashed, black!60] (-3,4.7) -- (0,4.7);
	\draw[dashed, black!60] (-0.3,0) -- (-0.3,5);
    \node at (-3.7,4.7) {$\mu- c_{4}$};
    \node at (-3.3,0.7) {$c_3$};
	\node at (-3.3,5) {$\mu$};
	\node at (-4.1,4.4) {$\mu-c_{1}+c_2 $};
	\node at (-4.1,2.5) {$\mu-c_1-c_{2} $};
	\node[rotate=-45] at (-1.3,-0.25) {$ c_{1} $};
    \node[rotate=-45] at (-1.9,-0.25) {$  c_{2} $};
    \node[rotate=-45] at (-2.3,-0.25) {$  c_{3} $};
    \node[rotate=-45] at (-0.1,-0.5) {$ 1- c_{4} $};
\node at (-1.5,-1.5) {($s_{2}$, $t_{2}$)};
\end{tikzpicture}
\end{minipage}%
\begin{minipage}{.33\textwidth}

\begin{tikzpicture}[scale=0.7, roundnode/.style={circle, draw=black!80, thick, minimum size=6mm}, font=\footnotesize]
 
  \node[roundnode] at (-4,7) (maintopic) {3} ;
 
    \draw[->][black!70] (-3.2,0) -- (0.6,0); 
    \draw[->][black!70] (-3,-0.2) -- (-3,6.2);

\draw[brown](-3,0.3) -- (-3,2.5);
   \draw[brown] (-2,4.5) -- (-1.5,5);
     \draw[brown] (-3,2.5) -- (-2,4.5);

    \draw[brown] (-1.5,5) -- (-0.7,5);
  \draw[brown] (0,4.3) -- (0,0);
    \draw[brown] (-2.7,0) -- (0,0);
   \draw[brown] (-2.7,0) -- (-3,0.3);
   \draw[brown] (0,4.3) -- (-0.7,5);
   
	\draw[dashed, black!60] (-1.5,5) --(-1.5,0);
	\draw[dashed, black!60] (-2,4.5) -- (-2,0);
	\draw[dashed, black!60] (-3,5) -- (-1.5,5);
	\draw[dashed, black!60] (-3,4.5) -- (-2,4.5);

	\draw[dashed, black!60] (-3,4.3) -- (0,4.3);
	\draw[dashed, black!60] (-0.7,0) -- (-0.7,5);
    \node at (-3.7,4.3) {$\mu- c_{3}$};
    \node at (-3.3,0.3) {$c_4$};
	\node at (-3.3,5) {$\mu$};
	\node at (-4.1,4.6) {$\mu-c_{1}+c_2 $};
	\node at (-4.1,2.5) {$\mu-c_1-c_{2} $};
	\node[rotate=-45] at (-1.3,-0.25) {$ c_{1} $};
    \node[rotate=-45] at (-1.9,-0.25) {$  c_{2} $};
    \node[rotate=-45] at (-2.6,-0.25) {$  c_{4} $};
    \node[rotate=-45] at (-0.4,-0.5) {$ 1- c_{3} $};
\node at (-1.5,-1.5) {($s_{3}$, $t_{3}$)};

\end{tikzpicture}
\end{minipage}%

\bigskip
\begin{minipage}{.33\textwidth}

\begin{tikzpicture}[scale=0.7, roundnode/.style={circle, draw=black!80, thick, minimum size=6mm}, font=\footnotesize]
 
  \node[roundnode] at (-4,7) (maintopic) {4} ;
 
    \draw[->][black!70] (-3.2,0) -- (0.6,0); 
    \draw[->][black!70] (-3,-0.2) -- (-3,6.2); 

\draw[brown] (-0.3,0) -- (0,0.3);
\draw[brown](-3,0) -- (-3,2.5);
   \draw[brown] (-2,4.5) -- (-1.5,5);
     \draw[brown] (-3,2.5) -- (-2,4.5);

    \draw[brown] (-1.5,5) -- (-0.7,5);
  \draw[brown] (0,4.3) -- (0,0.3);
    \draw[brown] (-3,0) -- (-0.3,0);
  
   \draw[brown] (0,4.3) -- (-0.7,5);
   
	\draw[dashed, black!60] (-1.5,5) --(-1.5,0);
	\draw[dashed, black!60] (-2,4.5) -- (-2,0);
	\draw[dashed, black!60] (-3,5) -- (-1.5,5);
	\draw[dashed, black!60] (-3,4.5) -- (-2,4.5);
	\draw[dashed, black!60] (-3,4.3) -- (0,4.3);
	\draw[dashed, black!60] (-0.7,0) -- (-0.7,5);
    \draw[dashed, black!60] (-3,0.3) -- (0,0.3);
    \node at (-3.7,4.3) {$\mu- c_{3}$};
    \node at (-3.3,0.3) {$c_4$};
	\node at (-3.3,5) {$\mu$};
	\node at (-4.1,4.6) {$\mu-c_{1}+c_2 $};
	\node at (-4.1,2.5) {$\mu-c_1-c_{2} $};
	\node[rotate=-45] at (-1.3,-0.25) {$ c_{1} $};
        \node[rotate=-45] at (-1.9,-0.25) {$  c_{2} $};
         \node[rotate=-45] at (0,-0.5) {$  1-c_{4} $};
          \node[rotate=-45] at (-0.4,-0.5) {$ 1- c_{3} $};

\node at (-1.5,-1.5) {($s_{4}$, $t_{4}$)};

\end{tikzpicture}
\end{minipage}%
\begin{minipage}{.33\textwidth}

\begin{tikzpicture}[scale=0.7, roundnode/.style={circle, draw=black!80, thick, minimum size=6mm}, font=\footnotesize]
 
  \node[roundnode] at (-4,7) (maintopic) {5} ;
 
    \draw[->][black!70] (-3.2,0) -- (0.6,0); 
    \draw[->][black!70] (-3,-0.2) -- (-3,6.2); 

\draw[brown] (-0.3,0) -- (0,0.3);
\draw[brown](-3,0.7) -- (-3,2.5);
   \draw[brown] (-2,4.5) -- (-1.5,5);
     \draw[brown] (-3,2.5) -- (-2,4.5);

    \draw[brown] (-1.5,5) -- (0,5);
  \draw[brown] (0,5) -- (0,0.3);
    \draw[brown] (-2.3,0) -- (-0.3,0);
 
   \draw[brown] (-3,0.7) -- (-2.3,0);
   
	\draw[dashed, black!60] (-1.5,5) --(-1.5,0);
	\draw[dashed, black!60] (-2,4.5) -- (-2,0);
	\draw[dashed, black!60] (-3,5) -- (-1.5,5);
	\draw[dashed, black!60] (-3,4.5) -- (-2,4.5);
    \draw[dashed, black!60] (-3,0.3) -- (0,0.3);
    \node at (-3.3,0.7) {$ c_{3}$};
    \node at (-3.3,0.3) {$c_4$};
	\node at (-3.3,5) {$\mu$};
	\node at (-4.1,4.6) {$\mu-c_{1}+c_2 $};
	\node at (-4.1,2.5) {$\mu-c_1-c_{2} $};
	\node[rotate=-45] at (-1.3,-0.25) {$ c_{1} $};
    \node[rotate=-45] at (-1.8,-0.25) {$  c_{2} $};
    \node[rotate=-45] at (0,-0.5) {$  1-c_{4} $};
    \node[rotate=-45] at (-2.1,-0.25) {$c_{3} $};

\node at (-1.5,-1.5) {($s_{5}$, $t_{5}$)};

\end{tikzpicture}
\end{minipage}%
\begin{minipage}{.33\textwidth}

\begin{tikzpicture}[scale=0.7, roundnode/.style={circle, draw=black!80, thick, minimum size=6mm}, font=\footnotesize]
 \node[roundnode] at (-4,7) (maintopic) {6} ;
 
    \draw[->][black!70] (-3.2,0) -- (0.6,0); 
    \draw[->][black!70] (-3,-0.2) -- (-3,6.2);

   \draw[brown](-3,1) -- (-3,2.5);
   \draw[brown] (-2,4.5) -- (-1.5,5);
    \draw[brown] (-3,2.5) -- (-2,4.5);

    \draw[brown] (-1.5,5) -- (0,5);
  \draw[brown] (0,5) -- (0,0);
    \draw[brown] (-1.7,0) -- (0,0);
   \draw[brown] (-1.7,0) -- (-2.3,0.3);
    \draw[brown] (-3,1) -- (-2.3,0.3);
   
	\draw[dashed, black!60] (-1.5,5) --(-1.5,0);
	\draw[dashed, black!60] (-2,4.5) -- (-2,0);
	\draw[dashed, black!60] (-3,5) -- (-1.5,5);
	\draw[dashed, black!60] (-3,4.5) -- (-2,4.5);
	\draw[dashed, black!60] (-3,0.3) -- (-2.3,0.3);
	\draw[dashed, black!60] (-2.3,0) -- (-2.3,0.3);
    \node at (-3.3,0.3) {$c_{4}$};
    \node at (-3.3,1) {$c_3$};
	\node at (-3.3,5) {$\mu$};
	\node at (-4.1,4.5) {$\mu-c_{1}+c_2 $};
	\node at (-4.1,2.5) {$\mu-c_1-c_{2} $};
	\node[rotate=-45] at (-1.2,-0.25) {$ c_{1} $};
    \node[rotate=-45] at (-1.95,-0.25) {$  c_{2} $};
    \node[rotate=-45] at (-2,-0.5) {$  c_{3} -c_4 $};
    \node[rotate=-45] at (-1.3,-0.5) {$  c_3+c_{4} $};
\node at (-1.5,-1.5) {($s_{6}$, $t_{6}$)};
\end{tikzpicture}
\end{minipage}%
\caption{Auxiliary Delzant polygons}
 \label{aux2}
 \end{figure}
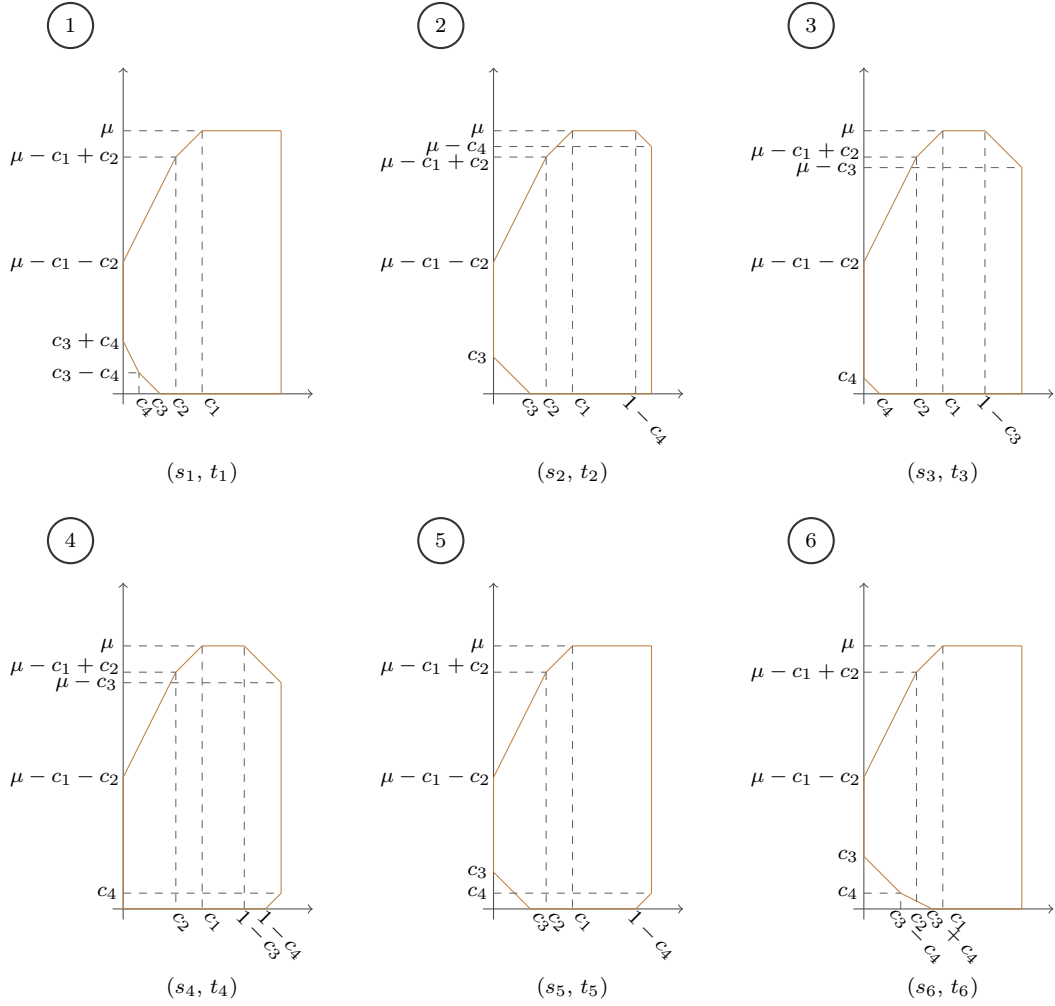
Applying to all these polygons the $GL(2, \Z)$ transformation represented by the matrix 
$$ \left (
\begin{array}{cc}
1  & 0 \\
-1 & 1 
\end{array} \right)$$ 
and then projecting onto the $y$-axis it is easy to see that we obtain, as elements of the fundamental group of $\Symp (\Mucccc)$,  the following identification 
\begin{equation}\label{rely}
y_i = t_i -s_i, \quad i \in \{1,2,3,4,5,6\}
\end{equation}
because the graphs obtained after the projection clearly coincide with the graphs of the actions $y_i$. It is also clear that 
\begin{equation}\label{relt}
t_3=t_4 \quad \mbox{and} \quad t_5 =t_6 
\end{equation}
since the graphs of these actions coincide in pairs. 
Finally, using the $GL(2, \Z)$ transformation represented by the matrix 
$$ \left (
\begin{array}{cc}
1  & 0 \\
1 & 1 
\end{array} \right)$$ applied to polygons 1 and 6 in Figure \ref{aux2} we obtain the following identification
\begin{equation}\label{relst}
s_1+t_1=s_6+t_6. 
\end{equation}

In order to finish the proof of relation \eqref{relation1} we just need to combine all the relations obtained above. More precisely, 
consider relation \eqref{relxy3} and using  first \eqref{rely} and then \eqref{relst} we get 
$$x_1-s_1= x_6-s_6.$$ 
From \eqref{basicx} and \eqref{relxy2} it follows that 
$x_3-s_1=x_5+y_5-y_6-s_6.$
Then using \eqref{rely}  to substitute $y_5$ and $y_6$ in the previous equation yields 
$x_3-s_1=x_5+t_5-s_5 -t_6.$
Hence, relations \eqref{basicx} and \eqref{relt} imply that 
$x_3-s_1= x_4-s_5.$
Finally, use first relation \eqref{relxy1} to obtain
$y_4-y_3= s_1-s_5$ and then \eqref{rely} one more time to get 
$$s_3-s_4=s_1-s_5.$$
Now notice that, using the notation for the circle actions defined in Figure \ref{graphsgeneric}, we have 
$$s_3= -z_{0,3}, \quad s_4=z_{0,12}, \quad s_1 = -z_0 \quad \mbox{and}  \quad s_5= -z_{0,4}$$
so  we proved the desired relation \eqref{relation1}. This concludes the proof of the lemma. 

\end{proof}

Combining relations \eqref{relation2} and \eqref{relation1} we obtain one of the identifications in Proposition \ref{basicrelations}, namely
\begin{equation} \label{basicz03}
z_{0,3}= z_1- z_{1,4} -z_{0,12}.
\end{equation}
The remaining identifications in Proposition \ref{basicrelations}, follow from the previous relation together with relations  \eqref{basicz01},
  \eqref{basicz02} and \eqref{basicz04}. 
Therefore we obtain
\begin{align*}
z_{0,1}&= z_1 -z_{1,4} + z_{0,14}   \\
z_{0,2}&= z_1 -z_{1,4} - z_{0,13} \\
 z_{0,4}&= z_1 -z_{1,4} - z_{0,12}-z_{0,13}+ z_{0,14}.
\end{align*}
Finally, using the last equation  and relation  \eqref{basicz03}  in relation  \eqref{relation1} it follows that 
$$ z_{0}= 2z_1 -2z_{1,4} - z_{0,12}-z_{0,13}+ z_{0,14}.$$
This concludes the proof of Proposition \ref{basicrelations}.
\end{appendix}

%
%

\end{document}